\documentclass[12pt, 
]{amsart}
\usepackage{amscd,amsmath, wasysym}
\title[]{Dirichlet and Neumann problems for planar domains with parameter}
\author[]{Florian Bertrand \and  Xianghong Gong}

 \address{Department of Mathematics,
 University of Wisconsin, Madison, WI 53706, U.S.A.}
 \email{bertrand@math.wisc.edu}\email{gong@math.wisc.edu}

  \thanks{Research of X.~Gong was  supported in part by NSF grant DMS-0705426.}

\newcommand{\dist}{\operatorname{dist}}

\newtheorem{thm}{Theorem}[section]
\newtheorem{cor}[thm]{Corollary}
\newtheorem{prop}[thm]{Proposition}
\newtheorem{lemma}[thm]{Lemma}

\theoremstyle{definition}

\newtheorem{rem}[thm]{Remark}

\renewcommand{\th}[1]{\begin{thm}\label{#1}}
\newcommand{\eth}{\end{thm}}
\newcommand{\co}[1]{\begin{cor}\label{#1}}
\newcommand{\eco}{\end{cor}}
\renewcommand{\le}[1]{\begin{lemma}\label{#1}}
\newcommand{\ele}{\end{lemma}}
\newcommand{\pr}[1]{\begin{prop}\label{#1}}
\newcommand{\epr}{\end{prop}}

\newcommand{\ga}{\begin{gather}}
\newcommand{\ega}{\end{gather}}
\newcommand{\gan}{\begin{gather*}}
\newcommand{\egan}{\end{gather*}}
\newcommand{\al}{\begin{align}}
\newcommand{\eal}{\end{align}}
\newcommand{\aln}{\begin{align*}}
\newcommand{\ealn}{\end{align*}}
\newcommand{\eq}[1]{\begin{equation}\label{#1}}
\newcommand{\eeq}{\end{equation}}

\newcommand{\pdoz}{\partial_{\overline z}}

\newcommand{\ci}{~\cite}

\newcommand{\f}[2]{\frac{#1}{#2}}

\newcommand{\df}{\overset{\text{def}}{=\!\!=}}
\newcommand{\D}{\mathbb{D}}
\newcommand{\cc}{{\bf C}}

\newcommand{\rr}{{\bf R}}

\newcommand{\ov}{\overline}

\newcommand{\RE}{\operatorname{Re}}
\newcommand{\IM}{\operatorname{Im}}

\newcommand{\jq}[1]{<\!#1\!>}

\newcommand{\cL}{\mathcal}

\newcommand{\all}{\alpha}

\newcommand{\gaa}{\gamma}
\newcommand{\Gaa}{\Gamma}
\newcommand{\del}{\delta}
\newcommand{\Del}{\Delta}
\newcommand{\var}{\varphi}
\newcommand{\e}{\epsilon}
\newcommand{\om}{\omega}
\newcommand{\Om}{\Omega}

\newcommand{\la}{\lambda}
\newcommand{\ta}{\tau}

\newcommand{\pd}{\partial}

\newcommand{\yt}{\frac{1}{2}}

\newcommand{\re}[1]{(\ref{#1})}
\newcommand{\rea}[1]{$(\ref{#1})$}
\newcommand{\rl}[1]{Lemma~\ref{#1}}
\newcommand{\nrc}[1]{Corollary~\ref{#1}}
\newcommand{\rp}[1]{Proposition~\ref{#1}}
\newcommand{\rt}[1]{Theorem~\ref{#1}}

\newcommand{\rla}[1]{Lemma~$\ref{#1}$}

\newcommand{\rpa}[1]{Proposition~$\ref{#1}$}

\newcommand{\supp}{\operatorname{supp}}

\newcounter{pp}
\newcommand{\bpp}{\begin{list}{$\hspace{-1em}\alph{pp})$}{\usecounter{pp}}}
\newcommand{\epp}{\end{list}}

\newcounter{ppp}
\newcommand{\bppp}{\begin{list}{$\hspace{-1em}(\roman{ppp})$}{\usecounter{ppp}}}
\newcommand{\eppp}{\end{list}}

\begin{document}
\begin{abstract}
Let $\Gamma(\cdot,\lambda)$ be smooth, i.e.\, $\mathcal C^\infty$,
 embeddings from $\overline{\Omega}$ onto $\overline{\Omega^{\lambda}}$, where
$\Omega$ and $\Omega^\lambda$ are bounded domains with smooth boundary in the complex plane
and $\lambda$ varies in $I=[0,1]$.
 Suppose that $\Gamma$ is smooth on $\overline\Omega\times I$ and $f$ is a
  smooth function on $\partial\Omega\times I$. Let $u(\cdot,\lambda)$ be the harmonic functions on $\Omega^\lambda$ with
  boundary values $f(\cdot,\lambda)$. We show that $u(\Gamma(z,\lambda),\lambda)$ is
smooth on $\overline\Omega\times I$.  Our main result is proved for suitable H\"older spaces
for  the Dirichlet and Neumann problems with parameter. By observing that the regularity of solutions
of the two problems with parameter is not local, we show the existence of   smooth
embeddings $\Gamma(\cdot,\lambda)$ from $\overline{\mathbb D}$,
the closure of the unit disc, onto $\overline{\Omega^\lambda}$
such that $\Gamma$ is smooth on $\overline{\mathbb D}\times I$
and real analytic at $(\sqrt{-1},0)\in\overline{\mathbb D}\times I$, but for every family
of Riemann mappings $R(\cdot,\lambda)$
from $\overline{\Omega^\lambda}$ onto $\overline{\mathbb D}$, the function $R(\Gamma(z,\lambda),\lambda)$
is not real analytic at $(\sqrt{-1},0)\in\overline{\mathbb D}\times I$.
\end{abstract}


 \maketitle


\setcounter{section}{0}
\setcounter{thm}{0}\setcounter{equation}{0}
\section{Introduction}\label{sec1}

Let $k\geq0$ be an integer and $0<\all<1$.  Let $\Om^\lambda\ (0\leq\la\leq1)$ be
a   family of bounded domains in $\cc$ of $\cL C^{k+1+\all}$ boundary.
Let $f^\la$ and $g^\la$ be $\cL C^\all$ functions on $\pd\Om^\la$.
We   consider the   Dirichlet problem with parameter
\eq{dpb}
 \Delta u^\la=0 \quad \text{on $\Om^\lambda$}, \qquad
      u^\la=f^\la \quad \text{on $\pd\Om^\lambda$.}
      \end{equation}
By analogy,  the Neumann problem with parameter is
\eq{npb}
 \Delta v^\la=0 \quad \text{on $\Om^\lambda$}, \qquad
     \pd_{\nu^\la} v^\la=g^\la \quad \text{on $\pd\Om^\lambda$}.
      \end{equation}
Here $\Del$ is the Laplacian and  $\nu^\la$ is
the outer unit normal vector of $\pd\Om^\la$. For the existence and uniqueness
 of solutions $v^\la$, we impose conditions
\eq{npb+}
\int_{\pd\Om^\la}g^\la\, d\sigma^\la=0,\qquad\int_{\pd\Om^\la}v^\la\, d\sigma^\la=0
      \end{equation}
with $d\sigma^\la$  being  the arc-length element of $\pd\Om^\la$.
We are interested in the regularity of
solutions $u^\la,v^\la$ in the parameter $\la$. To state our results, we first define two H\"older spaces.
Let integers $k,j$ satisfy  $k\geq j\geq0$. By an element $\{u^\la\}$ in
 $\cL C^{k+\all,j}(\pd\Om)$ (resp.\,$\cL C^{k+\all,j}(\ov\Om)$) we   mean
 a family of functions $u^\la$ on $\pd\Om$ (resp.\,$\ov\Om$) such
that,  for every integer $i$ with  $0\leq i\leq j$,  $\la\to\pd_\la^iu^\la$ is a continuous map from $[0,1]$ into $\cL C^{k-i+\all}(\pd\Om)$
(resp.\,$\cL C^{k-i+\all,j}(\ov\Om)$).
We will prove the following.
\pr{dpp} Let non negative integers $l, k$ and $ j$   satisfy
 $k\geq j$ and $k+1\geq l\geq j$. Let   $0<\all<1$.
Let $\Om$ be  a bounded domain in $\cc$ with $\cL C^{k+1+\all}$ boundary.
 Let $\Gaa^\la$ $(\la\in[0,1])$ embed $\ov\Om$ onto $\ov{\Om^\la}(\subset\cc)$ with
  $\{\Gaa^\la\}$ in $\cL C^{k+1+\all,j}(\ov\Om)$.
Assume that $ f^\la$ and $ g^\la$    are   functions on $\pd\Om^\la$ such that
 $\{f^\la\circ\Gaa^\la\}$ is in $\cL C^{l+\all,j}(\pd\Om)$ and
  $ \{g^\la\circ\Gaa^\la\}$ is in $\cL C^{k+\all,j}(\pd\Om)$.
For each $\la$, let $u^\la\in\cL C^\all(\ov{\Om^\la})$ be the unique solution  to \rea{dpb}
and let $v^\la\in\cL C^{1}(\ov{\Om^\la})$ be the unique
  solution
  to \rea{npb}-\rea{npb+}. Then $\{u^\la
  \circ\Gaa^\la\}$ is in $\cL C^{l+\all,j}(\ov\Om)$ and
    $\{v^\la\circ\Gaa^\la\}$ is in $\cL C^{k+1+\all,j}(\ov\Om)$.
\epr
We observe that
if a function $u$ is harmonic on the unit disc $\D$ and is continuous on $\ov\D$, then
  the product $b(\la)u(z)$ for a function $b$ on $[0,1]$  is still harmonic on
$\D$.
Thus, even if $b u$ is    real analytic
near a point $(p,0)\in\pd\ov\D\times[0,1]$, $b u$ might not be $\cL C^1$
near the same point   $(p,0)\in\ov\D\times[0,1]$. Such an example is provided, when
  $u|_{\pd\D}$ vanishes near $p$ but is not identically zero
 and $b$ is continuous on $[0,1]$ but not differentiable at $0$.
Therefore, the
regularity of solutions for the Dirichlet problem with parameter
 is not a local property.  By contrast, the harmonic function  $u$ must be $\cL C^\omega$ near $p\in\ov\D$
   when $u|_{\pd\D}$ is $\cL C^\omega$ near $p\in\pd\D$. The observation leads us to   demonstrate
 the failure of the local Schwarz reflection principle with parameter by the following result.
\th{nrefl} There are embeddings $\Gaa(\cdot,\la)$ from $\ov\D$ onto $\ov{\Om^\la}$ such that $\Gaa$ is
$C^{\infty}$  on $E=
\ov\D\times[0,1]$ and real analytic at $(1,0)\in E$, but
 $R(\Gaa(z,\la),\la)$ is not real analytic  at $(1,0)\in E$
 for every family of Riemann mappings $R(\cdot,\la)$
from $\ov{\Om^\la}$ onto $\ov\D$.
\eth
The existences of solutions $u^\la, v^\la$  in \rp{dpp}  are  classical results;
see  Kellogg~\ci{Keontw} for
the Dirichlet problem and Miranda
  \cite{Mirseze} (p.~84) for work of Giraud on the Neumann problem. For higher dimensional
Dirichlet problem, see Gilbarg-Trudinger (\cite{GTzeon}, p. 211, Theorem 8.34).
  The reader is referred to \cite{Mirseze} for extensive references.
We will use the   Fredholm theory on compact integral operators.
Of course, the compactness of the integral operators  is valid
when    the parameter is fixed and it will play important rules in our arguments, although
there is no compactness  when all variables are considered.
With some modifications, we will follow  Kellogg's approach to   the Dirichlet problem
(\cite{Kezeei}--\cite{Ketwni}).
For instance, by constructing  a
   second resolvent,  Kellogg proved  the $\cL C^{1+\all'}$-regularity  of the solutions
to the Dirichlet problem for $\cL C^{1+\all}$ boundary (\cite{Kezeeib}). Instead, we will
obtain the  regularity of solutions to
the Dirichlet problem via the integral equations associated to
 the Neumann problem.    The reduction
can be achieved, because solving Dirichlet problem
on a simply connected planar domain
can be reduced to finding a harmonic conjugate of the solution. We do
not meet   difficulties in the reduction
for multi-connected domains.
Using
the Cauchy transform, we will also refine
   Kellogg's original arguments to recover a loss of regularity.
We   mention that  Courant proved a version
of Carath\'eodory's Riemann mapping theorem
for variable Jordan domains (see~\ci{Tsfini}, p.~383).
  Courant's theorem implies the  continuous, i.e. $\cL C^0$,  dependence of solutions to the Dirichlet problems
 for Jordan domains with parameter.
    One of  applications of
solutions of the planar Dirichlet problem is Kellogg's theorem on the
boundary regularity of Riemann mappings for Jordan domains of $\cL C^{1+\all}$
boundary~\ci{Kezeeib}.    Warschawski proved the sharp version
of Kellogg's Riemann mapping theorem for Jordan domains of
$\cL C^{k+\all}$ boundary for all $k>0$ (\cite{Wathtw},   \cite{Wathtwb});
see also
Pommerenke~\ci{Ponitw} (p.~49). As an immediate consequence of \rp{dpp} we
 get a parameter version of Kellogg's Riemann mapping
theorem in \nrc{krm}.

The paper is organized as follows.

In section~\ref{sec2}, we define various H\"older spaces for domains and
functions with parameter. We     discuss the
 dependence of functions spaces on the parameterizations of domains
 and their boundaries.
Section~\ref{sec3} contains some standard estimates on   Cauchy transform (see Vekua~\cite{Vesitw}).
We present details as the arguments are used in the   parameter case.
In section~\ref{sec4}, we refine Kellogg's estimates on kernels for
the integral equations; lacking a
  reference to the sharp regularity    on solutions to
   the integral equations,  we  provide  some details.
      These arguments are generalized in section~\ref{sec5}
for the parameter case. In section~\ref{sec6}, after collecting
 results about compact operators for the Dirichlet and Neumann problems,  we deduce
 the $\cL C^1$ regularity of solutions of the integral equation
 for the Dirichlet problem in \rl{klem}.

Section~\ref{sec7} consists of our main results about the regularity
of solutions of integral equations
with parameter. For the proofs, we   differentiate integral equations
and orthogonal projections   onto  the null spaces of $I\pm \cL K^\la$ and $I\pm \cL K^{\la *}$
and we then derive estimates by using the compactness of integral operators
$\cL K^\la$ and $\cL K^{\la*}$ for fixed parameter $\la$.
In section~\ref{sec8}, we thoroughly discuss the H\"older spaces  defined
in section~\ref{sec2} before
 we define the spaces for exterior domains with parameter.
In section~\ref{sec9}, we solve the real analytic integral equations for the Dirichlet
and Neumann problems
with a real analytic parameter.
Our main results, theorems \ref{dnpb} and   \ref{nrefl},  are  proved in section~\ref{sec9}.
\rp{dpp} is contained in \rt{dnpb}.

Note that when domains $\Om^\la$ are  fixed and only the boundary values
vary with a parameter,
our results  essentially follow
from the solutions of   Dirichlet and Neumann problems without parameter.
Furthermore, the results hold for
general H\"older spaces with parameter (see the remark at the end of section~\ref{sec9}).
With   H\"older spaces  to be defined in section~\ref{sec2}, we conclude the
 introduction with the following open problem:

 \smallskip

\noindent{\bf Problem.}  Let $k,l, j$ be non negative integers.
Let $l\leq k+1$ and  $0<\all<1$. Let $\Gaa^\la$ embed
$\ov\Om$ onto $\Om^\la$, where $\Om$ and $\Om^\la$ are bounded domains in $\cc$.
 Let $u^\la\in\cL C^0(\ov{\Om^\la})$ be harmonic functions on
 $\Om^\la$. Suppose that  $\pd\Om\in
\cL C^{k+1+\all}$, $\Gaa\in\cL C_*^{k+1+\all,j}(\ov\Om)$,   and $\{u^\la\circ\Gaa^\la\}\in
\cL C_*^{l+\all,j}(\pd\Om)$. Is $\{u^\la\circ\Gaa^\la\}$ in
 $\cL C_*^{l+\all,j}(\ov\Om)$ for $j>0$?

\setcounter{thm}{0}\setcounter{equation}{0}
\section{H\"older spaces for interior domains  with parameter
}\label{sec2}

To deal with the Dirichlet problem with parameter, we will introduce
two types of H\"older spaces   with parameter,  ${\cL C}^{k+\all, j}(\ov\Om_\Gaa)$
 and ${\cL B}^{k+\all,j}(\ov\Om_\Gaa)$. Both are  suitable for the formulation and
 proofs of our  results.
 In this paper the parameter $\la$ will be in $[0,1]$, unless it is restricted to a subinterval.

We first define spaces when a domain is fixed.
 Let $k,j$ be non-negative integers  and let $0\leq\all<1$. Let $\Om$ be a bounded domain in $\cc$.
Let $\cL C^{k+\all}(\ov\Om)$ be the standard H\"older spaces with norm
$|\cdot|_{k+\all}$ on $\ov\Om$.
Let $u^\la$ be a family of functions on $\ov\Om$. We say that
  $\{u^\la\}$ belongs to ${\cL B}_*^{k+\all, j}(\ov\Om)$, abbreviated by $u=\{u^\la\}
  \in{\cL B}_*^{k+\all, j}(\ov\Om)$,
    if
$\la\to\pd_\lambda^iu^\lambda$ maps
$[0,1]$ {\it continuously}   into $\cL C^{k}(\ov\Om)$
and   {\it boundedly}  into $\cL C^{k+\all}(\ov\Om)$ for each $i$ with $0\leq i\leq j$.
We say $\{u^\lambda\}\in
{\cL C}_*^{k+\all, j}(\ov\Om)$, if $\pd_\lambda^iu^\lambda$ maps  $[0,1]$  {\it continuously}
 into $\cL C^{k+\all}(\ov\Om)$ for $0\leq i\leq j$. We define
   ${\cL B}_*^{k+\all, j}(\pd\Om)$
and its subspace  ${\cL C}_*^{k+\all, j}(\pd\Om)$    by substituting  $\Om$ with
$\pd\Om\in\cL C^{k+\all}\cap \cL C^1$ in the above expressions.

Next, we define spaces on domains with parameter.
Let $\Gamma^\lambda$   ($0\leq\la\leq1$) be a family of $\cL C^1$ embeddings from
 $\ov\Om$  onto $\ov{\Om^\lambda}$,
and let $\gamma^\lambda$   ($0\leq\la\leq1$) be a family of $\cL C^1$ embeddings
from $\pd\Om$  onto $\pd{\Om^\lambda}$.
Suppose that $u^\lambda$ is a family of functions on $\ov{\Om^\lambda}$ or on $\pd\Om^\la$.
 Define the following:
\begin{itemize}
\item $\{u^\lambda\}\in{\cL B}_*^{k+\all,j}(\ov\Omega_\Gamma)$,
if $\{u^\lambda\circ\Gamma^\lambda\}\in {\cL B}_*^{k+\all,j}(\ov{\Om})$;
\item
 $\{u^\lambda\}\in{\cL C}_*^{k+\all,j}(\ov\Omega_\Gamma)$,
if  $\{u^\lambda\circ\Gamma^\lambda\}\in {\cL C}_*^{k+\all,j}(\ov{\Om})$;
\item $\{u^\la\}\in\cL B_*^{k+\all,j}(\pd\Om_\gaa)$, if
$\{u^\la\circ\gaa^\la\}\in\cL B_*^{k+\all,j}(\pd\Om)$;
\item  $\{u^\la\}\in\cL C_*^{k+\all,j}(\pd\Om_\gaa)$, if
$\{u^\la\circ\gaa^\la\}\in\cL C_*^{k+\all,j}(\pd\Om)$.
\end{itemize}
For  integers $k\geq j\geq0$, define
\gan
\cL B^{k+\all,j}(\ov\Om_\Gaa)= \cap_{i=0}^j{\cL B}_*^{k-i+\all, i}(\ov\Om_\Gaa),\quad
\cL C^{k+\all,j}(\ov\Omega_\Gamma)= \cap_{i=0}^j{\cL C}_*^{k-i+\all, i}(\ov\Omega_\Gamma).
\end{gather*}
Substituting $\ov\Om_\Gaa$ with $\pd\Om_\gaa$ in the above identities, we define
 $\cL B^{k+\all,j}(\pd\Om_\gaa)$ and $\cL C^{k+\all,j}(\pd\Om_\gaa)$;
  dropping the subscripts
 $\Gaa$ and $\gaa$ from the above identities, we   define   $\cL B^{k+\all,j}(\ov\Om)$,
$\cL C^{k+\all,j}(\ov\Omega)$,
$\cL B^{k+\all,j}(\pd\Om)$ and
$\cL C^{k+\all,j}(\pd\Om)$, respectively.
 The norms on these   spaces are   defined and abbreviated as follows:
 \ga
 \label{ukaj}
|u|_{k+\all,j }=\sup_{  0\leq i\leq j,  \lambda\in[0,1]}
\{|\pd_\lambda^iu^\lambda |_{k+\all}\},\quad\text{if $u\in\cL B_*^{k+\all,j}(\pd\Om)$
or $\cL B_*^{k+\all,j}(\ov\Om)$},
\\
 |u|_{k+\all,j}=|\{u^\lambda\circ\Gamma^\lambda\}|_{k+\all,j},\quad  |u|_{k+\all,j}
 =|\{u^\lambda\circ\gamma^\lambda\}|_{k+\all,j},\\
 \label{norm2}
 \|u\|_{k+\all,j}=\max\{|u|_{k-i+\all, i}\colon 0\leq i\leq j\}, \quad j\leq k.
 \end{gather}
The definition  of
spaces $\cL B_*^{k+\all,j}(\pd\Om)$  requires $\pd\Om\in\cL C^{k+\all}\cap\cL C^1$
implicitly.
Throughout the paper, we assume that $\ov\Om$ is bounded,   $\pd\Om\in\cL C^1$,
 $\Gaa\in\cL C^{1,0}(\ov\Om)$ and $\gaa\in\cL C^{1,0}(\pd\Om)$.
 For $X=\pd\Om$ or $\ov\Om$ and $0\leq j,k\leq\infty$,
 define $\cL B_*^{k+\all,j}(X)=
\cap_{l<k+1,i<j+1}^\infty\cL B_*^{l+\all,i}(X)$. For $j\leq k\leq\infty$, define
$\cL B^{k+\all,j}(X)=\cap_{i\leq l<k+1, i<j+1}^\infty\cL B^{l+\all,i}(X)$.
 Define analogous spaces by
replacing $\cL B$ and $\cL B_*$
with  $\cL C$ and $\cL C_*$, respectively.

Having defined the spaces, we now   briefly discuss how they depend on the embeddings.
 We first need a fact to change the order of differentiation. Let $\pd_i=\pd_{x_i}$ be
 derivatives on $\rr^n$.
\le{comcs} Let $f$ be a continuous function defined on an open subset $\Om$ of $\rr^n$. Assume
that on $\Om$, $\pd_{i_1}\cdots\pd_{i_k}f=g$ is continuous and
 $\pd_{i_{j_1}}\cdots\pd_{i_{j_l}}f$
are continuous for all $1\leq j_1<\cdots<j_l\leq k$. Then $\pd_{i'_1}\cdots\pd_{i'_k}f$
exists
and equals $g$, where $\pd_{i'_1}\cdots\pd_{i'_k}$ is a change of order of
 $\pd_{i_1}\cdots\pd_{i_k}$.
\ele
\begin{proof} Let $\chi$ be any smooth function with compact support in $\Om$.
Replace $f$ by  $\chi f$. Then $f$
satisfies the same hypotheses  and it suffices to verify the assertion for
  the new $f$. Assume that $\supp f\subset
(a,\infty)^n$ for a finite $a$. Let $X$ be the set
of continuous functions on $\rr^n$ with support in $(a,\infty)^n$.
Define $\cL I_i\colon X\to X$ by
\gan
\cL I_i\phi(x)=\int_{a}^{x_i}\phi(x_1,\ldots,x_{i-1}, t,x_{i+1},\ldots, x_n)\, dt.
\end{gather*}
Then  $\cL I_i\cL I_j=\cL I_j\cL I_i$ on $X$. Also, $\pd_i\cL I_if=f=\cL I_i\pd_if$
if $f$ and $\pd_if$ are in $X$.
Now
$
f=\cL I_{i_k}\cdots\cL I_{i_1}g=\cL I_{i_k'}\cdots\cL I_{i_1'}g,
$
which yields $\pd_{i_1'}\cdots\pd_{i_k'}f=g$.
\end{proof}
The above lemma shows that $\pd_{x,y,\la}^L=\pd_{x,y}^I\pd_\la^j$ holds on
 $\cL C_*^{k,j}(\ov\Om)$,
if
$|I|=k$ and $\pd_{x,y,\la}^L$ is obtained from $\pd_{x,y}^I\pd_\la^j$
by   changing
the order of differentiation. Also $\pd_{\ta,\la}^L=\pd_\ta^k\pd_\la^j$
on $\cL C_*^{k,j}(\pd\Om)$, if $\pd_{\ta,\la}^L$ is a change of order of  $\pd_\ta^k\pd_\la^j$.
 \le{chainp} Let $\ov\Om$ be a bounded domain with $\pd\Om\in\cL C^{k+\all}\cap \cL C^1$.
 Let $\Gamma_1^\lambda$ and $
\Gamma_2^\lambda$ embed  $\ov\Om$   onto
$\ov{\Om^\lambda}$. Let $\gamma_1^\lambda$ and $
\gamma_2^\lambda$ embed $\pd\Om$   onto
$\pd{\Om^\lambda}$.
\bppp
\item
A $\cL C^1$ mapping from $\ov{\Om}$ into $\ov{\Om_1}$  pulls back $\cL C^{\all,0}(\ov{\Om_1})$
and $\cL B^{\all,0}(\ov{\Om_1})$ into $\cL C^{\all,0}(\ov\Om)$
and
 $\cL B^{\all,0}(\ov\Om)$, respectively.
\item Let $\var^\la$
map   $\ov\Om$ into an open subset $D$ of $\rr^n$.
If $F$ is a function in $\cL C^1(D)$ and
$\var\in\cL B^{\all,0}(\ov\Om)$, then
 $\{F\circ\var^\la\}\in\cL B^{\all,0}(\ov\Om)$.
 If $F\in\cL C^2(D)$ and $\var\in\cL C^{\all,0}(\ov\Om)$, then
  $\{F\circ\var^\la\}\in\cL C^{\all,0}(\ov\Om)$.
\item If $\Gaa_i\in\cL B^{k+\all,j}(\ov\Om)\cap\cL C^{1,0}(\ov\Om)$, then $
{\cL B}^{k+\all,j}(\ov\Om_{\Gamma_1})= {\cL
B}^{k+\all,j}(\ov\Om_{\Gamma_2})$.
\item Let $\all>0$.
If $(\Gaa_1^\lambda)^{-1}\Gaa_2^\lambda$ are independent of $\lambda$
and $\Gaa_i$ are in $\cL C^{k+\all,j}(\ov\Om)\cap\cL C^{1,0}(\ov\Om)$, then $
{\cL C}^{k+\all,j}(\ov\Om_{\Gamma_1})= {\cL
C}^{k+\all,j}(\ov\Om_{\Gamma_2})$.
\eppp
The   assertions in $(i)$-- $(iv)$ remain true   if $\pd\Om$,
$\pd\Om_1$, $\pd\Om^\la$,   and $\gaa_i$ substitute for $\ov\Om$,
$\ov{\Om_1}$, $\ov{\Om^\la}$,   and
$\Gaa_i$, respectively.
The  identical spaces in $(iii)$ and $(iv)$  have equivalent norms.
\ele
\begin{proof} (i). Since $\pd\Om\in\cL C^1$, then $|\var(z_2)-\var(z_1)|
\leq C|z_2-z_1|$ if $\var\in\cL C^1(\ov\Om)$ or
 $\cL C^1(\pd\Om)$.
The assertions  follow immediately from the definition of the spaces.

(ii). We take a bounded open subset $D'$ of $D$ such that
$\ov{D'}$ has piecewise smooth boundary and contains  ranges of all $\var^\la$. Since $F$
is $\cL C^1$, then $F$ is Lipschitz on $\ov{D'}$. It is easy to check
that $\{F\circ\var^\la\}$ is in $\cL B^{\all,0}$. Assume now that $F\in\cL C^2$. We already
know that $\{F\circ\var^\la\}$ is in $\cL B^{\all,0}$.
Without loss of generality, we may assume that $|\la_2-\la_1|$ is so small   that the range
of  $t\var^{\la_2}+(1-t)\var^{\la_1}$ for $0\leq t\leq1$ is contained in
 $D'$. Then $\nabla F$ is Lipschitz on $\ov{D'}$.
Write
$$
(F(\var^{\la_2})-F(\var^{\la_1}))(x)=(\var^{\la_2}-\var^{\la_1})(x)
\cdot\int_0^1(\nabla F)((t\var^{\la_2}+(1-t)\var^{\la_1})(x))
\, dt.
$$
We obtain $|F(\var^{\la_2})-F(\var^{\la_1})|_\all\leq C|\var^{\la_2}-\var^{\la_1}|_\all
(1+|\var^{\la_2}|_\all+|\var^{\la_1}|_\all)$. Hence, $\{F\circ\var^\la\}\in C^{\all,0}$.

(iii).  Let
$\Gaa_1^{\lambda}\circ \Gaa_{12}^\lambda=\Gaa_2^\lambda$.
Since $\Gamma^\lambda_i$ are   embeddings with
$\Gamma_i \in\cL C^{1,0}(\ov\Om)$, we have
\eq{cta}\nonumber
|\zeta-z|/{C}\leq|\Gamma_{12}^{\lambda}(\zeta)- \Gamma_{12}^{\lambda}(z)|\leq C|\zeta-z|.
\end{equation}
Note that on $\cL C^{k,j}(\ov\Om)$
 all mixed derivatives of order $k-j$ in $x,y$ and of order $j$ in $\la$
can be written as $\pd^{k-j-a}_x\pd_{y}^a\pd_\la^j$. Abbreviate   the latter derivatives
as a set
by $\pd^{k-j}\pd_\la^j$. For a later purpose of expressing a commutator, it will be convenient to
write first-order derivatives of a function in a column vector. So let us form
the Jacobean matrix $(\Gamma_i^\lambda)'$ of the (real) map $\Gamma_i^\lambda$ in such a way.
Then the chain rule takes the form
  $$(\Gamma_2^\lambda)'=(\Gamma_{12}^{\lambda})'(\Gamma_1^\lambda)'
  \circ\Gamma_{12}^{\lambda},\quad
  \pd_\lambda
\Gamma_2^\lambda=(\pd_\lambda\Gamma_1^{\lambda})\circ\Gamma_{12}^\lambda
  +\pd_\lambda\Gamma_{12}^{\lambda}(\Gamma_1^\lambda)'\circ\Gamma_{12}^{\lambda}.
  $$
We solve for $\pd\Gaa_{12}^\la$ and $\pd_\la\Gaa_{12}^\la$;  in general, for $k\geq j$,
we express $\pd^{k-j}\pd_\lambda^j\Gaa_{12}^\lambda$ as a
polynomial in 
\ga  \label{det1}  [\det (\Gamma_1^\lambda)'\circ\Gamma_{12}^{\lambda}]^{-1},\quad
\pd^{a} \pd_\lambda^{b}\Gamma_2^\lambda,\quad
(\pd^{a} \pd_\lambda^{b}\Gamma_1^\lambda)\circ\Gamma_{12}^{\lambda},   \quad a+b\leq k, b\leq j.
\end{gather}
To repeat the above computation for $\gaa_i^\la$, let $(\gaa_1^\la)
\circ\gaa_{12}^\la=\gaa_2^\la$. Then
\eq{det12}\begin{gathered}
\pd_\ta\gaa^\la_2(z)=\pd_\ta\gaa^\la_{12}(\pd_\ta\gaa^\la_1)\circ\gaa^\la_{12},\\
\pd_\la\gaa^\la_2(z)=(\pd_\la\gaa^\la_1)\circ\gaa^\la_{12}
+\pd_\la\gaa^\la_{12}(\pd_\ta\gaa_1)\circ
\gaa_{12}^\la(z).
\end{gathered}
\eeq
Hence
 $\pd_\ta^{k-j}\pd_\la^j\gaa_{12}^\la$
is a polynomial in
$$
[(\pd_\ta\gaa_1^\la)\circ\gaa_{12}^\la]^{-1},\quad\pd_\ta^b\pd_\la^b\gaa_2^\la,
\quad (\pd_\ta^a\pd_\la^b\gaa_1^\la)\circ\gaa_{12}^\la, \quad a+b\leq k,\ b\leq j.
$$

Assume that $\Gaa_1,\Gaa_2$ are in $\cL B^{k+\all,j}$. Then
functions in \re{det1}   are in $\cL B^{\all,0}$, so $\Gaa_{12}\in\cL B^{k+\all,j}(\ov\Om)$.
For  $u\in \cL B^{k+\all,j}(\ov\Om_{\Gaa_1})$, we express
$\pd^{k-j}\pd_\lambda^j(u^\lambda\circ\Gaa_2^\lambda)$ as a linear combination of
$(\pd^{a_1}\pd_\lambda^{b_1}(u^\lambda\circ\Gaa_1^\lambda))
\circ\Gaa_{12}^\la\in\cL B^{\all,0}$, whose coefficients are polynomials in entries of
\re{det1}. Here we replace the $(a,b)$ in \re{det1}   by $(a_2,b_2)$; also
$a_i+b_i\leq k$ and $ b_1+b_2\leq j$. Therefore, $u\in\cL B^{k+\all,j}(\ov\Om_{\Gaa_2})$.
Assume now that $u\in \cL B^{k+\all,j}(\pd\Om_{\gaa_1})$. Then
 $\pd_\ta^{k-j}\pd_\la^j(u^\la\circ\gaa_2^\la)$
is a linear combination in $\pd_\ta^{a_1}\pd_\la^{b_1}(u^\la\circ\gaa_2^\la)$, whose
coefficients are polynomials in \re{det12} with  $(a,b)$ being replaced by $(a_2,b_2)$. Here
$a_i+b_i\leq k$ and $b_1+b_2\leq j$. Thus, we get $u\in \cL B^{k+\all,j}(\pd\Om_{\gaa_2})$.

(iv). Assume that $\Gaa_1,\Gaa_2$ are in $\cL C^{k+\all,j}(\ov\Om)$.
By the independence of $\Gaa_{12}^\lambda\equiv \Gaa_{12}$ in $\lambda$
and (i), we know that all functions in \re{det1} are in $\cL C^{\all,0}(\ov\Om)$.
Furthermore,
\aln|(\pd^{a_1}\pd_{\mu}^{b_1}(u^{\mu}\circ\Gaa_1^{\mu}))\circ\Gaa_{12}&-
(\pd^{a_1}\pd_\lambda^{b_1}(u^\lambda\circ\Gaa_1^\lambda))\circ\Gaa_{12}|_\all\\
& \leq
C|\pd^{a_1}\pd_{\mu}^{b_1}(u^{\mu}\circ\Gaa_1^{\mu})-
\pd^{a_1}\pd_\lambda^{b_1}(u^\lambda\circ\Gaa_1^\lambda)|_\all.
\end{align*}
Let $u\in\cL C^{k+\all,j}(\ov\Om_{\Gaa_1})$. The above inequality shows
that  $(\pd^{a_1}\pd_\lambda^{b_1}(u^\lambda\circ\Gaa_1^\lambda))\circ\Gaa_{12}^\la$ are
in $\cL C^{\all,0}(\ov\Om)$. By (ii), the latter is closed under addition,
 multiplication, and division
(for non-vanishing denominator); hence, $u$ is
in $\cL C^{k+\all,j}(\ov\Om_{\Gaa_2})$. By analogy, we can verify that
$\cL C^{k+\all,j}(\pd\Om_{\gaa_1})=\cL C^{k+\all,j}(\pd\Om_{\gaa_2})$.
For (iii) and (iv), the equivalence of norms  is easy to verify, too.
\end{proof}

We now set up some notation to be used throughout the paper.

We   assume that $\Om$ and $\Om^\la$ are bounded domains of at least   $\cL C^1$ boundaries.
We denote by $\gamma_0$ the outer boundary of $\Om$ and by $\gamma_1,\ldots,
\gamma_m$ the connected components of its inner boundary.
 Without loss of generality, we  choose the
 standard orientation for $\pd\Om$
 and $\pd\Om^\la$ and  assume that $\cL C^1$ embeddings $\gaa^\la\colon\pd\Om\to\pd\Om^\la$
preserve the orientation and send outer boundary to outer boundary.
Denote by $\tau$ and $\tau^\la$  the unit tangential vectors of $\pd\Om$ and $
 \pd\Om^\la$ that agree with  the orientation, and by $\nu$ and $\nu^\la$
 the outer unit normal vectors of $\pd\Om$
 and $\pd\Om^\la$. The arc-length elements on $\pd\Om$ and $\pd\Om^\la$
 are denoted by $d\sigma$ and $d\sigma^\la$, respectively.
 Sometimes,
 we parameterize $\pd\Om$ by $\gaa(t)$
 in arc-length such that $dt$
 agrees with the orientation of $\pd\Om$, and we
  regard $\ta_z$ and $\gaa'(t)$ as   complex numbers instead of vectors.
With the above notation, on
$\pd\Om$ we have
\ga
\label{dsot} df=\pd_{\tau}f\, d\sigma,\quad d\sigma(\zeta)=\ov\ta_\zeta\, d\zeta,\quad
d\sigma^\la=|\pd_\tau\gaa^\la|\, d\sigma.
\end{gather}

To simply the use of the chain rule, we need to compute   derivatives in $\pd\Om^\la$
or $\overline{\Omega^\lambda}$.   At $z^\la=\gaa^\la(z)$, we have
\eq{tazf}
\ta^\la_z=|\pd_\ta\gaa^\la|^{-1}\pd_\ta\gaa^\la(z),\quad
(\pd_{\ta^\la}u^\la)(z^\la)=|\pd_\ta\gaa^\la(z)|^{-1}\pd_\ta(u^\la(\gaa^\la))(z).
\eeq
Combining with
$[\pd_\la,\pd_\ta]=0$, on
 $\cL C^{1,1}_*(\pd\Om_\gaa)$ and with $\gaa\in\cL C_*^{1,1}(\pd\Om)$
 we define and compute the  following commutator:
\ga
[\pd_\la,\pd_{\ta_z^\la}](f^\la(z^\la))=\pd_\la[\pd_{\ta_z^\la}(f^\la(z^\la))]
-\pd_{\ta_z^\la}[\pd_\la((f^\la(z^\la)))],\\
\label{comt}
[\pd_\la,\pd_{\ta_z^\la}]=|\pd_{\ta_z}\gaa^\la|\pd_\la|\pd_{\ta_z}\gaa^\la|^{-1}\pd_{\ta^\la_z}
=-(\pd_{\la}\log|\pd_{\ta_z}\gaa^\la|)\pd_{\ta^\la_z}.
\end{gather}
Therefore, for $ \gaa\in\cL C_*^{i,j}(\pd\Om)
\cap\cL C_*^{1,0}(\pd\Om)$, we have
\gan
\pd_{\tau^\la}\colon\cL C_*^{i,j}(\pd\Om_\gaa)\to\cL C_*^{i-1,j}(\pd\Om_\gaa), \quad
\pd_{\la}\colon\cL C_*^{i,j}(\pd\Om_\gaa)\to\cL C_*^{i,j-1}(\pd\Om_\gaa),\\
[\pd_\la,\pd_{\ta^\la}]\colon
\pd_{\tau^\la}\colon\cL C_*^{i,j}(\pd\Om_\gaa)\to\cL C_*^{i-1,j-1}(\pd\Om_\gaa),
\end{gather*}
whenever the exponents are non negative. For $ \gaa\in\cL C^{k,j}(\pd\Om)$, we have
\gan
\pd_{\tau^\la}\colon\cL C^{k,j}(\pd\Om_\gaa)\to\cL C^{k-1,j}(\pd\Om_\gaa),\ k-1\geq j;\\
\pd_{\la}\colon\cL C^{k,j}(\pd\Om_\gaa)\to\cL C^{k-1,j-1}(\pd\Om_\gaa);\\
[\pd_{\la},\pd_{\tau^\la}]\colon\cL C^{k,j}(\pd\Om_\gaa)\to\cL C^{k-2,j-1}(\pd\Om_\gaa),
\quad k\geq j+1\geq2;\\
[\pd_{\la},\pd_{\tau^\la}]\colon\cL C^{k,j}(\pd\Om_\gaa)\to\cL C^{k-1,j-1}(\pd\Om_\gaa), \quad
 \text{for $\gaa\in\cL C^{k+1,j}(\pd\Om)$ and $k\geq2$}.
\end{gather*}

Throughout the paper, we   denote by $C_{k+\all,j}$, or $C$,  a constant which depends on
\eq{const}
\sup_\la|\det(\Gaa^\la)'|_0^{-1},\ |(\Gaa^\la)'|_0,\
 \|\Gaa\|_{k+\all,j},\ ||\hat\gaa'|^{-1}|_0,\ |\hat\gaa'|_0,\  |\hat\gaa|_{k+\all},
 \eeq
  where $\hat\gaa$ is a parameterization
 for $\pd\Om$ of class $\cL C^{k+\all}\cap\cL C^1$. We
 also denote by $C_{k+\all}$ or $C$
  a constant which depends on the
 last three quantities.
The constants $\cL C_{l+\beta,j}^*$ will depend only on quantities in \re{const},
 where $\|\Gaa\|_{k+\all,j}$
is replaced by $|\Gaa|_{l+\all,j}$.

 A  consequence of  \re{tazf}-\re{comt} is the following.
\le{chainp++}  Let $\pd\Om\in\cL C^{k+\all}\cap\cL C^{1}$.
 Let $\gaa^\lambda$ embed $\pd\Om$ onto $\pd\Om^\lambda$ with
  $\gaa\in{\cL B}^{k +\all,j}(\pd\Om)\cap\cL C^{1,0}(\pd\Om)$.
Then $\{u^\lambda\}\in\cL B^{k+\all,j}(\pd\Om_\gaa)$ if and only if
 $\{\partial_{\tau^\la}^a\pd_\lambda^bu^\la\}$ or $\{\pd_\lambda^b\partial_{\tau^\la}^a u^\la\}$
 is in $\cL B^{\all,0}(\pd\Om_\gaa)$
for every $(a,b)$ with $a+b\leq k$ and $b\leq j$.  Moreover,
$$
C_{k+\all,j}^{-1}\|u\|_{k+\all,j}\leq
\sum_{a+b\leq k,b\leq j}\sup_\la\|\partial_{\tau^\la}^a\pd_\lambda^bu^\la\|_{\all,0}
\leq C_{k+\all,j}\|u\|_{k+\all,j}.$$
These  conclusions remain true  if $\cL C^{k+\all,j}$ and $\cL C^{\all,0}$ substitute  for
$\cL B^{k+\all,j}$ and $\cL B^{\all,0}$, respectively.
\ele

 We distinguish the first-order  derivatives on $\Om^\la$ by $\pd_{x^\la}$ in real
 variables $x^\la$
 and denote the first-order derivatives on $\Om$   by $\pd_x$.  Then for $x^\la=\Gaa^\la(x)$
\eq{tazf+}
\pd_{x^\la}u^\la=(\pd_x\Gaa^\la)^{-1}\pd_x(u^\la\circ\Gaa^\la).
\eeq
Combining with
$[\pd_\la,\pd_x]=0$, we define and compute on
 $\cL C^{1,1}_*(\ov\Om_\Gaa)$ with $\Gaa\in\cL C_*^{1,1}(\ov\Om)$
  the following commutator:
\ga
[\pd_\la,\pd_{x^\la}](f^\la(x^\la))=\pd_\la[\pd_{x^\la}(f^\la(x^\la))]
-\pd_{x^\la}[\pd_\la((f^\la(x^\la)))],\\
\label{comt+}
[\pd_\la,\pd_{x^\la}]=\pd_{\la}((\pd_{x}\Gaa^\la)^{-1})\pd_x\Gaa^\la\pd_{x^\la}.
\end{gather}
We denote by $\pd_{x^\la}^a$ the derivatives of order $a$ in $x^\la$.
The following can be verified easily.
\le{chainp+}
Let $\Gamma^\lambda$ embed  $\ov\Om$ onto $\ov{\Om^\lambda}$ with
  $\Gaa\in{\cL B}^{k+\all,j}(\ov\Om)\cap\cL C^{1,0}(\ov\Om)$.
Then $\{u^\lambda\}\in\cL B^{k+\all,j}(\ov\Om_\Gaa)$ if and only if
$\{\pd_{x^\la}^a\pd_\lambda^bu^\la\}$ or $\{\pd_\lambda^b\pd_{x^\la}^a u^\la\}$ is in
$\cL B^{\all,0}(\ov\Om_\Gaa)$
for every $(a,b)$ with $a+b\leq k$ and $b\leq j$.  Moreover,
$$
C_{k+\all,j}^{-1}\|u\|_{k+\all,j}\leq
\sum_{a+b\leq k,b\leq j}\|\pd_{x^\la}^a\pd_\lambda^bu^\la\|_{\all,0}
\leq C_{k+\all,j}\|u\|_{k+\all,j}.$$
The  conclusions remain true  if $\cL C^{k+\all,j}$ and $\cL C^{\all,0}$ substitute  for
$\cL B^{k+\all,j}$ and $\cL B^{\all,0}$, respectively.
\ele

We have seen the dependence of spaces $\cL C^{k+\all,j}$ in parameterizations through \rl{chainp}.
Throughout the paper, we   assume that $\gaa^\la$ is the restriction of $\Gaa^\la$ on $\pd\Om$.
We will return in   section~\ref{sec8} to further
discuss the spaces $\cL C^{k+1+\all,j}$ and $\cL B^{k+1+\all,j}$ and define H\"older
spaces for exterior domains.

We conclude the section with   further notation. Recall that $\Om\in\cL C^1$ is bounded and has
the standard orientation.
On $\pd\Om\times\pd\Om$ and off its diagonal, define
$
K(z,\zeta)=\f{1}{\pi}\pd_{\tau_\zeta}\arg(z-\zeta)$.
 By \re{dsot}, we have $K(z,\zeta)\,d\sigma(\zeta)=\f{1}{\pi}d_{\zeta}\arg(z-\zeta)$ and hence
\ga\nonumber
 \int_{\pd\Om}K(z,\zeta)\, d\sigma(\zeta)=
1,\quad z\in\pd\Om.
\end{gather}
A basic property of kernel $K$ is that $|K(z,\zeta)|\leq C|\zeta-z|^{\all-1}$ for
 $\zeta,z\in\pd\Om$, when
$\pd\Om\in\cL C^{1+\all}$ with $0<\all<1$.
By Fubini's theorem
 and H\"older inequalities (or Young's inequality), we have two bounded operators
 on $L^p(\pd\Om)$ ($p\geq1$)
\gan
\cL Kf(z)=\int_{\pd\Om}f(\zeta)K(z,\zeta)\, d\sigma(\zeta),
\quad\cL K^*f(z)= \int_{\pd\Om}f(\zeta)K(\zeta,z)\, d\sigma(\zeta).
\end{gather*}
These two operators  play   important roles in solving the Dirichlet and Neumann problems.
We will regard
$\cL K$ and $\cL K^*$ as operators on $L^1(\pd\Om)$, unless otherwise specified.
%

\setcounter{thm}{0}\setcounter{equation}{0}
\section{Integral equations for Dirichlet and Neumann problems}
\label{sec3}

Let $\Om$ be a bounded domain in $\cc$ with $\cL C^1$ boundary and let $f\in L^1(\pd\Om)$.
On $\Om$ and $\Om'=\cc\setminus\ov\Om$,   the double and simple
potentials with moment $f$ are respectively
\ga
\label{Ufz}
Uf(z)=  \f{1}{\pi}
\int_{\pd\Om}f(\zeta)\pd_{\tau_\zeta}\arg {(z-\zeta )}\, d\sigma(\zeta),\\
\label{Ufz+}
Wf(z)=
\f{1}{\pi } \int_{\pd\Om} f(\zeta)\log|z-\zeta |\, d\sigma(\zeta).
\end{gather}
 The following formulae lead the   solutions of the Dirichlet and Neumann problems
 via the Fredholm theory.
\pr{pjumpf}
Let $\pd \Om\in \cL C^{1+\alpha }$ with $0<\all<1$.
Suppose that $f$ is a continuous function on $\pd\Om$. Then $Wf$ is continuous on $\cc$ and
 $Uf$ extends to  functions
$U^+f\in\cL C^0(\ov\Om)$ and $U^-f\in\cL C^0(\ov{\Om'})$.  On $\pd\Om$
\ga
\label{u+-f}
U^+f=f+\cL Kf,\quad
 U^-f=-f+\cL Kf;\\
\partial_{\nu}Wf =f +\cL K^*f,
\quad-\pd_{-\nu}Wf= - f+\cL K^*f.\label{dnwf}
\end{gather}
\end{prop}
\begin{proof}  Recall that we   parameterize $\pd\Om$ by $\gaa(t)$  such that
 $dt$ is the arc-length element agreeing with the standard orientation of $\pd \Om$.
Let $l$ be the arc-length of $\pd\Om$.
We abbreviate $f(\gaa(t))$, $\tau(\gaa(t)),$ and $ \nu(\gaa(t))$ by $f(t),\tau(t),$ and
$\nu(t)$, respectively.

Write  $\gaa(t)=\gaa(s)+\ta(s)(t-s)+R(t,s)$ with $|R(t,s)|\leq |t-s|/4$ for $|t-s|<1/C$. Then
\eq{dtsc}
 \sqrt{h^2+|t-s|^2}/2\leq |\gaa(s)+h\nu( s )-\gamma(t)|\leq 2\sqrt{h^2+|t-s|^2}.
\eeq
For a later purpose we remark that the above merely needs $\gaa\in\cL C^1$.
Note that
$
\nu(t)\cdot (\gamma(t)-\gamma(s ))
=\nu(t)\cdot\int^{t}_{s  }(\gamma'(r)-\gamma'(t))\, dr.
$  Returning to condition $\gaa\in\cL C^{1+\all}$, we have, for $|t-s|<1/C$,
\eq{ntgt}
\f{|\nu(t)\cdot(\gaa(t)-\gaa(s))|}
{|\gaa(s)+ h\nu(s)-\gaa(t)|^2}\leq \f{C|s-t|^{1+\all}}{|t-s|^2+h^2}\leq
C|s-t|^{\all-1}.
\eeq
In particular,
 $$ 
 k(s,t)\df\pd_t\arg(\gamma(s)-\gamma(t) )=\f{\nu(t)\cdot (\gamma (t)
 -\gamma (s))}
{|\gamma(s)-\gamma(t)|^2}
$$ 
satisfies
$
 |k(s,t)|\leq C|s-t|^{\alpha-1},
 $
  and $k(s,\cdot)$ is integrable.

Recall that
$$
Uf(z)=  \f{1}{\pi}
\int_{0}^{l}f(t)\pd_t\arg {(z-\gaa(t) )}\, dt.
$$ Fix a small  $\e>0$ and $\gaa(s)\in\pd\Om$.
Let  $\delta=\dist(z,\pd\Om)$.
Choose $s_*$ such that $|\gaa(s_*)-z|=\del$. Note that as $\gaa\in\cL C^{1+\all}$
with $\all<1$, $s_*$ may not be unique even if $\delta$ is sufficiently small.
Nevertheless, $z=\gaa(s_*)+\del\nu(s_*)$.
Let $|z-\gaa(s)|$ be so small that $|s_*-s|<\e/2$.
 We have
\aln
&\pd_{t}\arg(z-\gaa(t)) =\f{\nu(t)\cdot(\gaa(t)-\gaa(s_*))
 }{|\gaa(t)-z|^2}+
\f{\nu(t)\cdot(\gaa(s_*)-z)
}{|\gaa(t)-z|^2}.
\end{align*}
 By \re{dtsc}-\re{ntgt}, we get
 \eq{pdta}\nonumber
 |\pd_{t}\arg(z-\gaa(t))|= \f{|\nu(t)\cdot(\gaa(t)-z)|
}{|\gaa(t)-z|^2}
 \leq C\f{|t-s_* |^{1+\alpha}+\delta}{\delta^2+|t-s_* |^2}.
 \eeq
 Since $s_*$ depends only on $z$, this shows that
 \eq{nlog}
\int_{0}^{l}\Bigl|\pd_{t}\arg(z-\gaa(t))\Bigr|\, dt
<C_0, \quad z\in\cc.
 \end{equation}
Here $C_0$ is independent of $s_*,z$ and $\delta$.
We have
\aln
&\bigl|\int_{0}^{l}
(f(t)-f(s))(\pd_{t}\arg {(z-\gaa(t) )}-\pd_{t}\arg {(\gaa(s)-\gaa(t) )})\, dt\bigr|\\
&\ \leq 2\|f\|_0\sup_{|\zeta-\gaa(s)|>\e}|\pd_t\arg(z-\gaa(t))-\pd_t\arg(\gaa(s)-\gaa(t))|
+2C_0\sup_{|t-s|<\e}|f(t)-f(s)|.
\end{align*}
By \re{nlog} and the continuity of $f$ at $s$, we conclude
\al\label{s3s5}
&\lim_{\pd\Om\not\ni z\to \gaa(s)}\int_{0}^{l}
(f(t)-f(s))\pd_{t}\arg {(z-\gaa(t) )}\, dt\\
&\qquad\qquad = \int_{0}^{l}
(f(t)-f(s))\pd_{t}\arg {(\gaa(s)-\gaa(t) )}\, dt.\nonumber
\end{align}
Expand both sides. By the values of $\int_{0}^l\pd_t\arg(z-\gaa(t))\, dt$ on $\cc$, we get \re{u+-f}.

For \re{dnwf},   recall that
$$
Wf(z)=
\f{1}{\pi } \int_{0}^{l} f(t)\log|z-\gaa(t)|\, dt.
$$
  We want to show that
 the interior and exterior    normal derivatives  of $Wf$ exist at $\gaa(s)$.
 Let  $|h|> 0$ be small. By the fundamental theorem of calculus, we have
\aln
&\f{W\!f(\gaa(s)+h\nu(s))-Wf(\gaa(s))}{h/\pi}
=\int_0^l\int_0^1f(t)\f{\nu(s)\cdot(\gaa(s)-\gaa(t))\, dr dt}
{|\gaa(s)+rh\nu(s)-\gaa(t)|^2}\\
& \hspace{5em}+\int_0^l\int_0^1\f{f(t)rh\, dr dt}
{|\gaa(s)+rh\nu(s)-\gaa(t)|^2}\df R_1(s,h)+R_2(s,h).
\end{align*}
We see that $R_1(s,h)$ tends to $\int_0^lf(t)\f{\nu(s)\cdot
(\gaa(s)-\gaa(t))}{|\gaa(s)-\gaa(t)|^2}\,
dt$ as  $h\to0$, by \re{ntgt} and  the dominated convergence theorem.

  Decompose $R_2$ into   integrals $R_\e', R_\e''$ in $(t,r)$ with $|t-s|<\e$
and $|t-s|>\e$, respectively. It is immediate that, for fixed $\e>0$, $R_\e''(s,h)$
tends to $0$ as
$h\to0$. Note that the integrand in $R_2$ does not change the sign
when $f\geq0$.  By the continuity of $f$, it remains to show that when $f\equiv 1$
\eq{lelh}
\lim_{\e\to0}\lim_{h\to0^+}R_\e'(s,h)=\pi, \quad \lim_{\e\to0}\lim_{h\to0^-}R_\e'(s,h)=-\pi.
\eeq
Let $E(s,t)=\gaa(s)-\gaa(t)+\gaa'(s)(t-s)$. Then $|E(s,t)|\leq C|s-t|^{1+\all}$ and for $|h|<1$
\al
\nonumber
&|\gaa(s)+ h\nu(s)-\gaa(t)|^2=|-\ta(s)(t-s)+ h\nu(s)+E(s,t)|^2\\
\nonumber &\hspace{21.5ex}=
(s-t)^2+h^2+\tilde E(s,t,h),\\
\nonumber\label{tEcr}
&|\tilde E(s,t,h)|\leq C(| h||t-s|^{1+\all}+|t-s|^{2+\all})\leq2 C\e^\all(h^2+|s-t|^2).
\end{align}
Let $h$ tend to $0^+$ and then let $\e$ tend to $0^+$. We get
\aln
R_\e'(s,h)&=(1+C\e^\all)\int_{|t-s|<\e} \int_0^1\f{rh\, dr dt }
{(s-t)^2+(rh)^2}\to\pi.
\end{align*}
This yields the first identity in \re{lelh}. The second is obtained by analogy.
\end{proof}

Let $\Om$ be a bounded domain in $\cc$ with $\cL C^1$ boundary.
Recall the Cauchy transform
\eq{cfz1}
\cL Cf(z)=\f{1}{2\pi i}\int_{\pd\Om}\f{f(\zeta)}{\zeta-z}\, d\zeta
\eeq
on $\cc\setminus\pd\Om$ for   $f\in L^1(\pd\Om)$.
Away from $\pd\Om$,
$$ 
 Uf=2\RE\cL Cf,\quad \text{for $ f=\ov f$};\quad
\pd_zWf=-i\cL C[ \ov\ta f].
$$ 
We will  derive   estimates of $Uf,Wf$ via $\cL Cf$, when $f$ is in H\"older spaces.

\le{jumpf} Let $0<\all<1$ and let $k,l\geq0$ be integers.
Let $\Om$ be a bounded domain in $\cc$ with $\pd \Om\in \cL C^{1 }$ and let
 $\Om'=\cc\setminus\ov\Om$.
\bppp \item
 Let $f$ be  a function  in $C^\all(\pd \Om)$. Then $\cL Cf$
extends
to functions $\cL C^+f\in\cL C^{\alpha}(\ov\Om)$ and $\cL C^-f\in\cL C^{\alpha}(\ov{\Om'})$.
Moreover, on $\pd\Om$
\ga\label{c+fs}
\cL C^-f(z)=
\f{1}{2\pi i}\int_{\pd\Om}\f{ f(z)-f(\zeta)}{z-\zeta} \, d\zeta, \quad
\cL C^+f(z)-\cL C^-f(z)=f(z).
\end{gather}
\item Let $f\in \cL C^{l+\alpha}(\pd \Om)$ and $\pd \Om\in \cL
C^{k+1+\alpha}$ with $k+1\geq l$. Then  $\cL C^+f\in \cL C^{l+\alpha}(\ov\Om)$
and $\cL C^-f\in \cL C^{l+\alpha}(\ov{\Om'})$.  If $f$ and
$\pd\Om$ are real analytic, then $\cL C^+f\in\cL C^\om(\ov\Om)$.
\item
If $f\in L^\infty(\pd\Om)$, then $Wf$ extends to a
continuous function   on $\cc$.
\eppp
\end{lemma}
\begin{proof}
(i). Let $z\not\in\pd\Om$ and
let $z_*=\gamma(s)$ satisfy $|z-z_*|=\dist(z,\pd \Om)=\delta$.  Assume that
 $\delta$ is small. We have
$$ 
(\cL Cf)'(z) =
\f{1}{2\pi i}\int_{\pd\Om}\f{ f(\zeta)-f(z_*)}{(\zeta-z)^{2}} \, d\zeta.
$$ 
By \re{dtsc},
$
|(\cL Cf)'(z)|\leq C\int_{0}^{\infty} (r+\delta)^{\alpha-2}\, dr
\leq C'\delta^{-1+\alpha}=C'\dist(z,\pd \Om)^{-1+\alpha}.
$
By the Hardy-Littlewood lemma, $\cL Cf$ is of class
$\cL C^{\alpha}$ on $\ov\Om$ and $\ov{\Om'}$.

To find  the boundary values of $\cL C^+f$ and $\cL C^-f$, it suffices to compute  limits
  of $\cL Cf$
in the normal directions. Let $z=\gaa(s)+\del \nu(s)\in\Om'$ and $z_*=\gaa(s)$. Write
$$
\cL C^-f(z)= \f{1}{2\pi i}\int_{\pd\Om}\f{ f(\zeta)-f(z_*)}
{\zeta-z} \, d\zeta.
$$
By \re{dtsc}, we obtain
$$
 \f{ |f(\gaa(t))-f(\gaa(s))|}
{|\gaa(t)-z|}\leq C|t-s|^{\alpha-1}.
$$
By the dominated convergence theorem, we find on $\pd\Om$
$$
\cL C^-f(z_*)= \f{1}{2\pi i}\int_{\pd\Om}\f{ f(\zeta)-f(z_*)}
{\zeta-z_*} \, d\zeta.
$$
Analogously, we can verify the formula for $
\cL C^+f.$

(ii).
For higher order derivatives,   for $l\leq k+1$  we get from \re{dsot}
\ga
\nonumber
\pd_z \cL Cf(z) =\f{1}{2\pi i}\int_{\pd\Om}\f{f(\zeta)\, d\zeta}{(\zeta-z)^{2}}
=\f{1}{2\pi i}\int_{\pd\Om}\ov\tau\pd_\tau f(\zeta)\f{
d\zeta}{\zeta-z},\\
\pd_z^{l} \cL Cf(z)
\label{dzcf}
= \f{1}{2\pi i}\int_{\pd\Om}\f{(\ov\tau\pd_\tau)^{l}f(\zeta)
}{\zeta-z}\, d\zeta,\\
\nonumber
\pd_z^{l+1} \cL Cf(z) 
= \f{1}{2\pi i}\int_{\pd\Om}\f{(\ov\tau\pd_\tau)^{l}f(\zeta)
-(\ov\tau\pd_\tau)^{l}f(z_*  )}{(\zeta-z)^2}\, d\zeta.
\end{gather}
By \re{dtsc} again, we obtain
\ga
\label{pdzl1-}
|\pd_z^l\cL Cf(z)|\leq |(\ov\ta\pd_\ta)^lf|_0+C_1|(\ov\tau\pd_\ta)^lf|_{\all},\\
\nonumber 
|\pd_z^{l+1} \cL Cf(z)|\leq C_{1}|(\ov\ta\pd_\ta)^{l}f|_{\all}\dist(z,\pd\Om)^{-1+\alpha}.
\end{gather}
Therefore,  $\cL Cf$ is of class $\cL C^{l+\alpha}$ on $\ov\Om$ and $\ov{\Om'}$.

For the real analytic case, we note that the constant $C_{1}$ in \re{pdzl1-}
is independent of $l$.
By Taylor's theorem,
 a function $f$ on $\ov\Om$ with $\pd\Om\in\cL C^\om$ is real analytic if and
 only if $$
 |\pdoz^i\pd_z^jf(z)|\leq i!j!C^{i+j+1}
 $$
for some $C>1$ independent of $z$.  Note that
$|(\ov\ta\pd_{\ta})^lf|_\all\leq C|(\ov\ta\pd_{\ta})^{l+1}f(z)|_0$.
 By \re{pdzl1-} it suffices to show that  near each point $z_0 \in\pd\Om$, we have
\eq{tapdk}
|(\ov\ta\pd_{\ta})^lf(z)|\leq C^{l+1}l!.
\eeq
Let $x\to \var(x)$ be a local real analytic parameterization of $\pd\Om$
with $\var(0)=z_0$.  Then $(d\var^{-1})(\ov\ta\pd_{\ta})$
is given by $A(x)\pd_x$ with $A\neq0$. Extend $\var(x)$, $A(x)$ and $f(\var(x))$
as holomorphic functions
and denote them by the same symbols.
We find   local holomorphic coordinates
$z=\psi(w)$ such that $(d_w\psi)^{-1}(A\pd_z)
=\pd_w$. Then $(\ov\ta\pd_\ta)^lf(\zeta)=\pd_w^l(f\circ\var\circ\psi)(w)$
 with $\zeta=\var\circ\psi(w)$.
Since $f\circ\var\circ\psi$ is holomorphic, we get \re{tapdk} easily.

(iii). One can verify  the
 continuity of $Wf$ via \re{dtsc}.
\end{proof}

\setcounter{thm}{0}\setcounter{equation}{0}
\section{Derivatives of $\cL Kf$ and $\cL  K^*f$}
\label{sec4}
In this section, we recall some calculation on kernels by Kellogg~\ci{Kezeeib}, \ci{Keontw}
and express $\cL K$ and $\cL K^*$ via the Cauchy transform.
We   write $\nu_{\gaa(t)}=\nu(t), \tau_{\gaa(t)}=\tau(t)$ and $f(\gaa(t))=f(t)$.
Let $l_j$ be  the $j$-th component $\gaa_j$ of $\pd\Om$. Recall that
  $\gaa_0$ is the outer boundary of $\pd\Om$.
Set $l_{-1}=0$.
\le{kta} Let $\pd\Om\in\cL C^{k+1+\all}$ with $k\geq0$ and $0<\all<1$. On $\pd\Om\times\pd\Om$,
 we have
\ga
\label{kta1}
|\pd_{\tau_z}^kK(\zeta,z)|\leq C_{k+1+\all}|z-\zeta|^{\all-1},\\
\label{kta2}
|K(z_2,\zeta)-K(z_1,\zeta)|\leq C_{1+\all}\frac{|z_2-z_1|}{|\zeta-z_1|^{2-\all}},\\
\quad |\pd_{\tau_{z_2}}^kK(\zeta, z_2)-\pd_{\tau_{z_1}}^kK(\zeta, z_1)|
\leq C_{k+1+\all}\frac{|z_2-z_1|^\all}{|\zeta-z_1|},
\label{kta3}
\end{gather}
where the last two inequalities require $|\zeta-z_1|>2|z_2-z_1|$.
\ele
\begin{proof}
We first verify \re{kta2}. We have
$$
\pd_t\arg(\gamma(s)-\gamma(t))
=\f{N(s,t)}{|\gamma(s)-\gamma(t)|^2}, \quad N(s,t)=\nu(t)\cdot\int_{s}^t
(\gamma'(r)-\gamma'(t))\, dr.
$$
First, we obtain $
 |N(s_1,t)|\leq C|t-s_1|^{1+\alpha}$ and
\ga
\left||\gamma(s_2)-\gamma(t)|^2-|\gamma(s_1)-\gamma(t)|^2\right|\leq
C |s_2-s_1|(|t-s_1|+|t-s_2|).
\label{gs2s}\nonumber
\end{gather}
Note that
$$
N(s_2,t)-N(s_1,t)=-\nu(t)\cdot
\int_{s_1}^{s_2}(\gamma'(r)-\gamma'(t))\, dr.
$$
Using $|\gamma'(r)-\gamma'(t)|\leq |\gaa'(r)-\gaa'(s_1)|+|\gaa'(t)-\gaa'(s_1)|$,
we obtain
$$
|N(s_2,t)-N(s_1,t)|\leq C(|s_2-s_1|^{1+\alpha}+
 |s_2-s_1||t-s_1|^\alpha ).
$$
Combining the above, we get for $|t-s_1|\geq2|s_2-s_1|$,
\al\label{ks2s}\nonumber
|K(s_2,t)-K(s_1,t)|\leq C  \f{
|s_2-s_1| }{|t-s_1|^{2-\all}}.
\end{align}

To verify \re{kta3},   we may assume that $x'(t)\neq 0$ for $t$ near $s$.
For a later use, we remark that the rest of computation does not need
 $dt$ to be the arc-length element. The condition $x'(t)\neq 0$ is
only
to ensure that  $C^{-1}|t-s|\leq |x(t)-x(s)|\leq C|t-s|$.
Following \ci{Keontw}, let
\eq{xsxt}\nonumber
(x(t)-x(s))q(s,t)=y(t)-y(s).
\end{equation}
 By \re{tazf}, we have
$\pd_{\ta}u(\gaa(t))=|\pd_t \gaa|^{-1}\pd_t(u(\gaa(t)))$. By $\arg(x+iy)=\arctan(y/x)\mod\pi$, we get
$$
\pd_\ta^kK(\gaa(s),\gaa(t))=\sum_{j\leq k}Q_{j}^\la(t)\pd_t^{j}q(s,t).
$$
where $Q_{j}$  are $C^\infty$ functions in  $|\pd_t\gaa|^{-1},
\pd_t\gaa,\ldots, \pd_t^{k+1}\gaa$, and $q(s,t)$. Hence \re{kta3} follows from
\eq{kta3t}
|\pd_{t_2}t^kq(s, t_2)-\pd_{t_1}^kq(s, t_1)|\leq C_{k+1+\all}\frac{|t_2-t_1|^\all}{|s-t_1|}.
\eeq
Differentiate the equation  and solve for $\pd_t^jq$.  Then $(x(t)-x(s))^{k+2}\pd_t^{k+1}q$
equals the determinant
\begin{equation*}
\begin{vmatrix}
  x(t)-x(s) & 0 & 0&\cdots & 0 & y(t)-y(s) \\
  x'(t) & x(t)-x(s) & 0 & \cdots & 0 & y'(t)\vspace{.75ex}\\
  x''(t) & \binom{2}{1}x'(t) & x(t)-x(s) &\cdots & 0 & y''(t)\\
  \cdot &\cdot &\cdot &\cdot&\cdot&\cdot\\
  \cdot &\cdot &\cdot &\cdot&\cdot&\cdot \\
  \cdot &\cdot &\cdot &\cdot&\cdot&\cdot\\
  x^{(k)}(t) & \binom{k}{1}x^{(k-1)}(t) & \binom{k}{2}x^{(k-2)}(t) & \cdots & x(t)-x(s)
  &y^{(k)}(t)
  \vspace{1ex} \\
x^{(k+1)}(t) & \binom{k+1}{1}x^{(k)}(t) & \binom{k+1}{2}x^{(k-1)}(t) & \cdots &
\binom{k+1}{k}x'(t) & y^{(k+1)}(t)
\end{vmatrix}.
\end{equation*}
Multiply  the $i$-th row by $\f{1}{i!}(s-t)^i$ and add  it to
the first row. The entries in the first row become
$$
\frac{1}{j!}(s-t)^{j}(\cL P_{k+1-j}x(s,t)-x(s)), \quad 0\leq j\leq k,
\quad \cL P_{{k+1}}y(s,t)-y(s),
$$
where   $\cL P_kf(s,t)$ denotes the Taylor polynomial
of degree $k$ of $f$ about $s=t$. Then
the remainder $\cL R_{k}f(s,t)=f(s)-\cL P_{k}f(s,t)$ can be written as
$$
\cL R_{k}f(s,t)=\frac{(s-t)^{k}}{k!}\hat R_kf(s,t), \quad  \hat R_kf(s,t)
=\int_0^1\Bigl\{f^{(k)}(t+r(s-t))
-f^{(k)}(t)\Bigr\}\, dr.
$$
Therefore,
\eq{keli}
(x(s)-x(t))^{k+2}\pd_t^{k+1}q(s,t)=(s-t)^{k+1}\Bigl\{
P_0\hat R_{k+1}y(s,t)+\sum_{i=1}^{k+1}
P_i
\hat R_ix(s,t)\Bigr\},
\eeq
where $P_i(s,t)$ are polynomials in $\pd_t\gaa, \ldots,\pd_t^{k+1}\gaa,
x(s)-x(t)$.
Then \re{kta1} follows from
$|\hat R_{k+1}\gaa(s,t)|\leq C|\gaa|_{k+1+\all}
|s-t|^\all$. Assume that $|s-t_2|\geq2|t_2-t_1|$. We have
$|\hat R_i\gaa(s,t_2)|\leq C|s-t_1|^\all$ and
\gan
 |(P_i,\hat R_{k+1}\gaa)(s,t_2)-(P_i,\hat R_{k+1}\gaa)(s,t_1)|
 \leq C|\gaa|_{k+1+\all}|t_2-t_1|^\all.
\end{gather*}
Using $|t_2-t_1|\leq |s-t_1|^{1-\all}|t_2-t_1|^\all$, we get
 \gan
 |(s-t_2)^{k+1}-(s-t_1)^{k+1}|\leq C|s-t_1|^{k+1-\all}|t_2-t_1|^\all,\\
|(x(s)-x(t_2))^{-k-2}-(x(s)-x(t_1))^{-k-2}|\leq C |s-t_1|^{-k-2-\all}|t_2-t_1|^\all.
\end{gather*}
By the above inequalities, we get \re{kta3t} and hence \re{kta3}.
\end{proof}

We   need a function $\Theta$, which   plays an important role in Kellogg's
first-order derivative estimate.
Define a single-valued continuous
function $\pi\Theta(t,t)$ on $[0,l]$, which measures the angle from the $x$-axis
 to the tangent
line of $\pd\Om$ at $\gaa(t)$.
Set
$$
\Theta(s,t)=\Theta(s,s)+\f{1}{\pi}\int_s^t\pd_r\arg(\gaa(s)-\gaa(r))\, dr,\quad s,t\in[0,l].
$$
Then
$
\partial_t\Theta(s,t)=K(s,t),
\Theta(s,t)=\Theta(t,s)$, and $
 \Theta(s,l)-\Theta(s,0)=1.
$
\le{dsar}   Let $\pd\Om\in \cL C^{k+1+\alpha}$ with $k\geq0$ and $0<\alpha<1$.
Let $I_i=\{s\colon 0<s-(l_0+\ldots+l_{i-1})<l_i\}$.
\bppp
\item Let   $\varphi\in L^1(\pd\Om)$.
 In the sense of distributions,
\ga\label{psar}
\pd_s\int_0^l\var(\gaa(t))\Theta(s,t)\, dt=
\int_0^l\var(\gaa(t))K(\gaa(t),\gaa(s))\, dt,\quad  s\in I_i,\\
\label{psar+} \pd_{\tau}^{k}\int_{\pd\Om}\var(\zeta) K(\zeta,z)\, d\sigma(\zeta)=
 \int_{\pd\Om}\var(\zeta)\pd_{\tau_z}^{k}K(\zeta,z)\, d\sigma(\zeta),\quad
  \text{  $z\in\pd\Om$}.
\end{gather}
\item
 If   $\var$ and $\pd_{\tau}\varphi$ are in $ L^1(\pd\Om)$, then on $\pd\Om$ and
 in the sense of distributions
\ga
 \pd_{\tau}\int_{\pd\Om}\varphi(\zeta)K(z,\zeta)\, d\sigma(\zeta)=
  -\int_{\pd\Om}\pd_{\tau}\varphi(\zeta)K(\zeta,z)\, d\sigma(\zeta).
\label{psar++}\end{gather}
\eppp
\end{lemma}
\begin{proof}
(i). Note  that $\Theta(s,t)$
is a continuous branch of $\f{1}{\pi}\arctan\f{y(s)-y(t)}{x(s)-x(t)}$ on $[0,l]\times[0,l]$.
Then
$$
\pd_s\int_{I_j}\var(t)\Theta(s,t)\, dt=
\int_{I_j}\var(t)\pd_s\arg(\gamma(s)-\gamma(t))\, dt
$$
holds on $I_i$ when $j\neq i$.
It suffices to  verify that on $I_i$
$$
\pd_s\int_{I_i}\var(t)\Theta(s,t)\, dt=
\int_{I_i}\var(t)\pd_s\arg(\gamma(s)-\gamma(t))\, dt.
$$
Thus we may assume that $\pd\Om=\gaa_i$.
We have
$$
 \int_0^l\int_{0}^{l} \f{| \var(t)|}{|t-s|^{1-\alpha}} \, dt  ds
\leq C|\var|_{L^1}.
$$
Hence,
\gan   \int_0^l \int_{0}^{l}| \var(t)\pd_s\Theta(s,t)|\, dt     ds
\leq C\int_{0}^{l}\int_{0}^{l}\f{|\var(t)|}{|t-s|^{1-\alpha}}\, dtds\leq C'|\var|_{L^1}.
\end{gather*}
Therefore, $\int_{0}^{l}| \var(s)\pd_s\Theta(s,t)|\, ds$ is in $L^1(\pd\Om)$.
For a test function $\phi$ on $(0,l)$,
\aln
\int_{0}^{l}\phi'(s)\int_{0}^{l}\var(t)\Theta(s,t)\, dt  ds
 &=\int_{0}^{l}\var(t) \int_{0}^{l} \phi'(s)\Theta(s,t)\,  dsdt
\\
&=-\int_{0}^{l}\phi(s)
\int_0^l\var(t)\pd_s\Theta(s,t)\, dt   ds.
\end{align*}
which gives us \re{psar}.

To verify \re{psar+}, we let $k\geq1$ and  use
\gan
\pd_\ta^{k-1}\int\var(\zeta)K(\zeta,z)\, d\sigma(\zeta)=
\int\var(\zeta)\pd_{\ta_z}^{k-1}K(\zeta,z)\, d\sigma(\zeta).
\end{gather*}
Let $\phi$ be a $\cL C^1$ function on $\pd\Om$.
Let $\chi_\e(\zeta,z)-1$ be a $\cL C^1$ functions on $\pd\Om\times\pd\Om$ which has support
in $|\zeta-z|<\epsilon$ such that $|\chi_\e|\leq 1$
and  $|\pd_{\ta_z}\chi_\epsilon(\zeta,z)|<C\e^{-1}$.
 Now
\aln
I&=\int\pd_{\ta_z}\phi(z)\int\var(\zeta)\pd_{\ta_z}^{k-1}K(\zeta,z)\,
d\sigma(\zeta)\, d\sigma(z)\\
&=\lim_{\e\to0}
\int\var(\zeta)\int\chi_\e(\zeta,z)\pd_{\ta_z}[\phi(z)-\phi(\zeta)]\pd_{\ta_z}^{k-1}K(\zeta,z)
\, d\sigma(z)\, d\sigma(\zeta).
\end{align*}
Write the last double integral as $-I'_\e-I''_\e$ with
\gan
I_\e'=\int\var(\zeta)\int(\pd_{\ta_z}\chi_\e(\zeta,z))[\phi(z)-\phi(\zeta)]
\pd_{\ta_z}^{k-1}K(\zeta,z)\, d\sigma(z)\, d\sigma(\zeta),\\
I_e''=\int\var(\zeta)\int\chi_\e(\zeta,z) [\phi(z)-\phi(\zeta)]
\pd_{\ta_z}^{k}K(\zeta,z)\, d\sigma(z)\, d\sigma(\zeta).
\end{gather*}
Then $I_\e'$ tends to $0$ uniformly in $\zeta$
as $\e\to0$, because as $\e$ tends to $0$
\aln
&\Bigl|\int(\pd_{\ta_z}\chi_\e(\zeta,z))[\phi(z)-\phi(\zeta)]
\pd_{\ta_z}^{k-1}K(\zeta,z)\, d\sigma(z)\Bigr|\\
&\qquad \leq C\e^{-1}\sup_{|z-\zeta|<\e}|\phi(z)-\phi(\zeta)|
\int_{z\in\pd\Om,|z-\zeta|<\e}|\pd_{\ta_z}^{k-1}K(\zeta,z)|\, d\sigma(z)\to0.
\end{align*}
 Since $| \chi_\e|\leq1$ and $| (\phi(z)-\phi(\zeta))
\pd_{\ta_z}^{k}K(\zeta,z)|\leq C$ by \re{kta1}, then $\lim_{\e\to0}I_\e''$ equals
$$
I''=\int\var(\zeta)\int [\phi(z)-\phi(\zeta)]
\pd_{\ta_z}^{k}K(\zeta,z)\, d\sigma(z)\, d\sigma(\zeta).
$$
Since $k\geq1$, then
\aln
\int\pd_{\ta_z}^{k}K(\zeta,z)\,   d\sigma(z)&=\lim_{\e\to0}\int_{|z-\zeta|>\e}
\pd_{\ta_z}^{k}K(\zeta,z)\,   d\sigma(z)\\
&=
\lim_{\e\to0} \{\pd_{\ta_z}^{k-1}K(\zeta,\zeta''_\e)-\pd_{\ta_z}^{k-1}K(\zeta,\zeta'_\e)\}=0.
\end{align*} Here we have
used the continuity of $\pd_{\ta_z}^{k-1}K(\zeta,z)$, and identities $\{\zeta'_\e,\zeta''_\e\}=\pd\Om\cap
\{z\colon|z-\zeta|=\e\}$ for small $\e$  and $\lim_{\e\to0}\zeta'_\e=z=\lim_{\e\to0}
\zeta''_\e$. Now \re{psar+} follows from
$$I=-\int\var(\zeta)\int \phi(z)
\pd_{\ta_z}^{k}K(\zeta,z)\, d\sigma(z)\, d\sigma(\zeta).
$$

(ii). When $\pd\Om$ is parameterized $\gaa(t)$,
  at $z=\gaa(t)$, we have $\pd_\ta f(z)\,d\sigma(z)=df(z)=\pd_t(f(\gaa(t)))\, dt$.
Then  \re{psar++} is obtained by
  integration by parts and \re{psar}.
\end{proof}

 We have seen from \re{psar++}
that differentiating integral operator $\cL K$ inevitably leads
to the   kernel $\cL K^*$. To recover a loss of
regularity in Kellogg's arguments. We will need to combine  \re{kta1} and \re{kta3}
with estimates on   $\cL K,\cL K^*$ from the Cauchy transform.
\le{k1rg}
Let $\pd\Om\in\cL C^{k+1+\alpha}$ with  $0<\alpha<1$. Then for a real
 function $\psi\in\cL C^\all(\pd\Om)$,  $$
\cL  K\psi=  2\RE\{   \cL C^+\psi\}-  \psi,\quad
\cL  K^*\psi=\psi-2 \RE\{ \ta  \cL C^+(\ov\ta\psi)\}.
$$
 In particular,   $
\cL K^*(\cL C^{k+\all}(\pd\Om))\subset\cL C^{k+\all}(\pd\Om)$;
and  for $l\leq k+1$, $
\cL K(\cL C^{l+\all}(\pd\Om))\subset\cL C^{l+\all}(\pd\Om)
$.
\ele
\begin{proof}
The first formula follows from \re{c+fs} immediately.
 Parameterize $\pd\Om$
by $\gaa$ in  arc-length.  By a simple computation we obtain
\al
\label{pdsta}
\pd_s\arg(\gamma(s)-\gamma(t))&=-\RE(\gamma'(s)\ov{\gamma'(t)})\pd_t\arg(\gamma(s)-\gamma(t))\\
&\quad -
\IM(\gamma'(s)\ov{\gamma'(t)})\pd_t\log|\gamma(s)-\gamma(t)|.\nonumber
\end{align}
To verify the second one, we use \re{pdsta} to decompose
\gan
 \cL  K^*\psi(z)=\f{1}{\pi}\int_{\pd\Om} \psi(\zeta)
\pd_{\tau_z}\arg(z- \zeta )\, d\sigma(\zeta)=J_1(z)+J_2(z)
\end{gather*}
with
\gan
 J_1(z)=-\psi(z)+ \f{1}{\pi}\RE\Bigl\{ \tau_z\int_{\pd\Om} (\psi(z)
 \ov{\tau_z}-\psi(\zeta)\ov{\tau_\zeta})
\pd_{\tau_\zeta}\arg(z-\zeta)\,
 d\sigma(\zeta)\Bigr\},\\
J_2(z)= \f{1}{\pi}\IM\Bigl\{ \tau_z  \int_{\pd\Om} (\psi(z)
\ov{\tau_z}-\psi(\zeta)\ov{\tau_\zeta})\pd_{\tau_\zeta}\log|z-\zeta|\,
d\sigma(\zeta)\Bigr\}.
\end{gather*}
By a simple computation,
$$
\cL  K^*\psi(z)=-\psi(z)+\f{1}{\pi}\IM\Bigl\{ \tau_z  \int_{\pd\Om}
(\psi(z)\ov{\tau_z}-\psi(\zeta)\ov{\tau_\zeta}
)\f{d\zeta}{\zeta-z}\Bigr\}.
$$
Therefore,
$
\cL  K^*\psi=-\psi-2 \RE\{ \ta  \cL C^-(\ov\ta\psi)\}.
$
The assertions follow from \rl{jumpf}.
\end{proof}

\setcounter{thm}{0}\setcounter{equation}{0}
\section{Kernels with parameter}
\label{sec5}
We have derived estimates for $K, K^*$ and $
\cL Cf$. In this section we   modify the arguments for the parameter case.
The   requirement that $k\geq j$
in the H\"older spaces $\cL C^{k+\all,j}(\pd\Om)$
will be evident in  identity \re{coun1} below
and in the proof of \rl{ktap} for the Cauchy transform
with parameter.

\le{logd} Let $\Om$ be a bounded domain in $\cc$ with
$\pd\Om\in\cL C^1$.
 Let $\Gamma^\lambda$ embed  $\ov\Om$ onto $\ov{\Om^\lambda}$ with $\Gaa\in\cL C_*^{1,j}(\ov\Om)$.
Let $z^\la=\Gaa^\la(z)$ and $k>0$.
For $z,\zeta\in\ov\Om$ with $\zeta\neq z$,
\ga
\label{flze}
\Bigl|\f{1}{(\zeta^\mu - z^\mu )^k }-
\f{1}{(\zeta^{\lambda} - z^{\lambda}  )^k}\Bigr|\leq
 C_{1,0}\f{|\Gamma^{\mu}-\Gamma^{\lambda}|_1}{|\zeta-z|^k},\\
\label{flxe}
\Bigl|\f{1}{(x^\mu(\zeta)-x^\mu(z))^{k}}-\f{1}{(x^\la(\zeta)-x^\la(z))^{k}}\Bigr|
\leq C_* C_{1,0}\f{|x^\mu-x^\la|_1}{|\zeta-z|^k},
\\
 \bigl|\pd_\la^j
\log|\zeta^{\lambda} - z^{\lambda}|\bigr|\leq
 C_{1,j}^*,\quad   j\geq1,
\label{f1ze+-}
\\
\bigl|\pd_\mu^j\log|\zeta^\mu - z^\mu  |-\pd_\la^j
\log|\zeta^{\lambda} - z^{\lambda}|\bigr|\leq
C_{1,j}^*  |\Gamma^{\mu}-\Gamma^{\lambda}|_{1,j},\quad j\geq0,
\label{f1ze+}
\end{gather}
where \rea{flxe} is for $ \zeta,z\in\pd\Om$
and   under the assumptions that $|\pd_{\ta_z} x^\la|\geq
1/{C_*}$ and
$|x^\mu-x^\la|_1, |\zeta-z|$ are sufficiently small.
Assume further that $\Gaa^\la\in\cL B_*^{1+\all,j}(\pd\Om)$. Then for $\zeta\in\pd\Om$
\al
\label{flze2}&\bigl| \pd_\la^j\pd_{\ta_\zeta^\la}\arg(\zeta^\la-z^\la)\bigr|\leq
 C_{1+\alpha,j}^*
\f{|\zeta-z|^{1+\alpha}+\dist(z,\pd\Om)}{|\zeta-z|^2}. 
\end{align}
\end{lemma}
\begin{proof}
Since $\Gamma^\lambda$ are embeddings with $\Gaa\in\cL C^{1,0}(\ov\Om)$, we have
\eq{zzeg}
|\zeta-z|/{C}\leq|\zeta^{\lambda} - z^{\lambda}|\leq C|\zeta-z|.
\eeq
 Take a path $\rho$ in $\ov\Om$ such that $\rho(0)=z,
\rho(1)=\zeta$ and $|\rho'|\leq C|\zeta-z|$.  When $j=0$, \re{f1ze+} follows from
$|\log(1+x)|\leq 2|x|$ for $|x|<1/2$ and
\al\label{zmzl}
|\pd_\mu^j(\zeta^\mu - z^\mu)-\pd_\la^j
(\zeta^{\lambda} - z^{\lambda})|&=\Bigl|\int_0^1\nabla(\pd_\mu^j\Gamma^{\mu} -\pd_\la^j
 \Gamma^{\lambda})(\rho(t))\cdot\rho'(t)\, dt\Bigr|\\
&
 \leq C|\Gamma^{\mu}-\Gamma^{\lambda}|_{1,j}
|\zeta  - z |.
\nonumber\end{align}
And \re{flze} follows from \re{zzeg}-\re{zmzl}  too.  By analogy, one can verify \re{flxe}.
For $j>0$, $\pd_\la^j
\log|\zeta^{\lambda} - z^{\lambda}|$ is a linear combination of
$$
Q^\la(\zeta,z)= \frac{\pd_\la^{i_1}(\zeta^\la-z^\la)\cdots\pd_\la^{i_a}(\zeta^\la-z^\la)
}
{(\zeta^\la-z^\la)^{a}}
$$
and their conjugates, where  $ a>1$ and $  i_l\leq j$. Using
$|\pd_\la^{j}(\zeta^\la-z^\la)|\leq C_{1,j}|\zeta-z|$,
we obtain  \re{f1ze+} by \re{zmzl} and  \re{f1ze+-}.  Note that $Q^\la$ may not extend
continuously  to $z=\zeta$.

To verify \re{flze2}, we choose   local $\cL C^{k+1+\all}$
   coordinates such that $\Om$
contains $[-1,1]\times(0,1]$ and $\pd\Om$ contains $[-1,1]\times\{0\}$.
Assume that $\zeta=\xi+i0$, $z\in\Om$  and $|\zeta|+|z|<1/2$.
Then $\gaa_1^\la(x,y)=\pd_x\gaa^\la(x,y)$ is tangent to  $\pd\Om^\la$ and
$$
\pd_{\ta^\la_\zeta}\arg(\zeta^\la-z^\la)=|\gaa_1^\la(\xi,0)|^{-1}
\IM\left\{\f{\gaa_1^\la(\xi,0)\ov{(\gaa^\la(\xi,0)-\gaa^\la(x,y))}}
{|\gaa^\la(\xi,0)-\gaa^\la(x,y)|^2}\right\}.
$$
Set
$\gaa_2^\la(x,y)=\pd_y\gaa^\la(x,y)$.  We have
\al
\nonumber
&\gaa^\la(x,y)-\gaa^\la(\xi,0)=\int_0^1\left\{(x-\xi)\gaa_1^\la(\xi+r(x-\xi),0)
 +y\gaa_2^\la(x,ry)\right\}\, dr,\\
\nonumber&\IM\left\{\gaa_1^\la(\xi,0)\ov{(\gaa^\la(\xi,0)-\gaa^\la(x,y))}\right\}\\
\nonumber&\hspace{5.6em}=
\IM\left\{\ov{\gaa_1^\la(\xi,0)}[(x-\xi)\hat R_1\gaa_1^\la(x,y,\xi)+y
R_1^*\gaa_2^\la(x,y,\xi)]\right\},\\
\label{hr1}
&\hat R_1\gaa_1^\la(x,y,\xi)=\int_0^1[\gaa_1^\la(\xi+r(x-\xi),0)-\gaa_1^\la(\xi,0)]\, dr, \\
\label{r1s}
&  R_1^*\gaa_2^\la(x,y,\xi)=\int_0^1 \gaa_2^\la (x,ry)\, dr.\end{align}
Therefore, $\pd_\la^j\pd_{\ta^\la_\zeta}\arg(\zeta^\la-z^\la)$ is a linear combination of
\begin{align}\label{i0i1}
\pd_\la^{i_0}|\gaa_1^\la(\xi,0)|^{-1}\IM&\Bigl\{\ov{\pd_\la^{i_1}
\gaa_1^\la(\xi,0)}[(x-\xi)\hat R_1\pd_\la^{i_2}
\gaa_1^\la(x,y,\xi)
\\  &\quad +y  R_1^*\pd_\la^{i_2}\gaa_2^\la(x,y,\xi)]\Bigr\}q_{i_3}^\la(\xi,x,y).
\nonumber\end{align}
Here $i_0+i_1+i_2+i_3=j$ and $
q_{i_3}^\la(\xi,x,y)=\pd_\la^{i_3}|\gaa^\la(\xi,0)-\gaa^\la(x,y)|^{-2}.
$
We can verify that
\gan
|\pd_\la^{i_0}|\gaa_1^\la(\xi,0)|^{-1}|\leq C_{1,j}, \quad
|\pd_\la^{i_1}\gaa_1^\la(\xi,0)|\leq C_{1,j}.
\end{gather*}
By the arguments for \re{f1ze+-}-\re{f1ze+}, we obtain
\gan
|q_{i_3}^\la(\xi,x,y)|\leq C|z-\zeta|^{-2}.
\end{gather*}
By \re{hr1}-\re{r1s}, we get $|\hat R_1\pd_\la^{i_2}\gaa_1^\la(x,y,\xi)|
\leq C_{1+\all,j}|x-\xi|^\all$ and
 $| R_1^*\pd_\la^{i_2}\gaa_2^\la(x,y,\xi)|\leq C_{1,j}$.
In \re{i0i1}, we have $y=\dist(z,\pd\Om)$ and $|x-\xi|\leq|z-\zeta|$.
Combining the above estimates, we get \re{flze2}.
\end{proof}

Given a family of continuous function $f^\lambda$ on $\pd\Om^\la$, let
$
\cL C^\lambda  f
$
be the Cauchy transform defined off $\pd\Om^\la$ by \re{cfz1}.
Let  $\cL C_+^\la f$ be its restrictions on   $
 \Om^\la$.
Denote by $W^\la f$ and $U^\la f$ the single and double layer potentials with moment $f^\la$
on $\pd\Om^\la$. Denote by $W_+^\la f$ and $U_+^\la f$  their restrictions on $\Om^\la$
and extensions to $\ov{\Om^\la}$ if continuous extensions exist.

It will be convenient to use notation
\ga
\label{norm3-} |u^\la|_{i+ \all,j}=\max_{l\leq j }|\pd_{\la}^{l}  (u^\la\circ\Gaa^\la)|_{i+\all},
\quad
  \|u^\la\|_{k+ \all,j}=\max_{i\leq j }|u^\la|_{k+\all-i,i},
  \\
\label{norm3} |u^{\mu}-u^\la|_{i+ \all,j}=\max_{l\leq j }| \pd_{\mu}^l    (u^{\mu}\circ\Gaa^\mu)
 - \pd_{\la}^l  (u^\la\circ\Gaa^\la)|_{i+\all},
 \\  \|u^{\mu}-u^\la\|_{k+ \all,j}=\max_{i\leq j }|u^{\mu}-u^\la|_{k-i+ \all,i},\quad j\leq k.
\label{norm4} \end{gather}
Define analogous norms by replacing $\Gaa^\la$ with $\gaa^\la$.

\begin{prop}\label{c0pa}
Let $\Gamma^{\lambda}$ embed $\ov\Om$ onto $\overline{\Omega^\lambda}$
 with $\Gaa\in\cL B^{k+1+\all,j}(\ov\Om)$. Let $j\leq k$ and $j\leq l\leq k+1$. With the norms
 defined by \rea{ukaj}-\rea{norm2} and \rea{norm3-}-\rea{norm4}, we have
\ga
\|\cL C_{+} f\|_{0,0}\leq
C_{1,0}|f|_{\all,0},\quad \|\cL C_{+} f\|_{l+\all,j}\leq
C_{k+1+\all,j}\|f\|_{l+\all,j},
\label{clc+}\\
\|\cL C_{+}^{\mu} f-\cL C_{+}^\lambda f\|_{0,0}\leq
C_{1,0}(|\Gamma^{\mu}-\Gamma^{\lambda}|_{1}|f^{\lambda}|_\all
+  \|f^\mu-f^\la\|_{\all,0}),
\label{clc+1}\\
\|\cL C_{+}^{\mu} f-\cL C_{+}^\lambda f\|_{l+\all,j}\leq
C_{k+1+\all,j}(\|\Gamma^{\mu}-\Gamma^{\lambda}\|_{k+1+\all,j}|f^{\lambda}|_{l+\all}
+  \|f^{\mu}-f^{\lambda}\|_{l+\all,j}).
\label{clc+2}\end{gather}
If $\pd\Om$ is real analytic, $\Gaa^\la(z)$ and
$f^\la\circ\Gaa^\la(z)$ are real analytic on $\ov\Om\times[0,1]$, then
$\cL C_+^\la f\circ\Gaa^\la(z)$ and $W^\la_+f\circ\Gaa^\la(z)$ are
real analytic on $\ov\Om\times
[0,1]$ too.
\end{prop}
\begin{proof}    Let $z\in \Om$. Take $z_*\in\pd\Om$ such that
 $|z_*-z|=\dist(z,\pd\Om)$. We have
$$
\cL C^\lambda f(z^\la)=f^{\lambda}(z_*^\la)+
\f{1}{2\pi i}\int_{\pd\Om}\f{f^{\lambda}( \zeta^\la)-f^{\lambda}(z_*^\la)}
{  \zeta^\la - z^\la }\, d\zeta^\la.
$$
Denote the last integral by $A^\la(z)$. By \re{flze} it is easy to see that
\aln
&|A^{\mu} -A^{\lambda}|_0\leq
 C_{1,0}(\|f^{\mu}-f^{\lambda}\|_{\all,0}+
|f^\lambda|_\all|\Gamma^{\mu}-\Gamma^{\lambda}|_1)\int_{\pd\Om}
\f{|\zeta-z_*|^{\all}}{|\zeta-z|}\, |d\zeta|.
\end{align*}
The last integral is bounded by a constant; indeed
when $\del=\dist(z,\pd\Om)$ is sufficiently small, for $z_*=\gaa(s)$ and $\zeta=\gaa(t)$
we have $|\zeta-z_*|\leq C|s-t|$ and $|z-\zeta|\geq(\delta+|t-s|)/C$.
This verifies \re{clc+1}. By \rl{jumpf}, $\cL C_+^\la f$ is continuous when $\la$
 is fixed. Then \re{clc+1} also implies that $\cL C_+f$
 is in $\cL C^{0,0}(\ov\Om_\Gamma)$. One can also verify the first inequality in \re{clc+}.
 Notice that the proof merely needs $\Gaa\in\cL C^{1,0}(\ov\Om)$. Next, we will verify
 \re{clc+2} and the second inequality in \re{clc+}. Note that \re{clc+2} and \rl{jumpf}
 implies that $\cL C_+f$ is in $\cL C^{l,j}(\ov\Om_\Gaa)$.

Denote by  $\pd_{z^\la}$ the derivative in $z$ on $\Om^\la$. By analoge of \rl{chainp+},
 it suffices
to estimate norms for    $\pd_{z^\la}\cL C_+^\la f$. We first consider the case where $j=0$.
Differentiate $\cL C^{\la}f$ and then
apply Stokes to transport the derivatives
to $f^\la(\gamma^{\lambda}(\zeta))$.   By \re{dzcf} we have  \aln
g^\la( z^\la)&\df\pd_{z^\la}^{l} \cL C^\lambda f(z^\la) =
\f{1}{2\pi i}  \int_{\zeta\in\pd\Om}\f{\tilde\pd_{\ta^\la}^{l}f^\lambda(\zeta^\la)
}{\zeta^\la-z^\la}d\zeta^\la
\df  \int_{\zeta\in\pd\Om}\f{h^\lambda(\zeta^\la)\,d\zeta^\la
}{\zeta^\la-z^\la},
\end{align*}
where $\tilde \pd_{\tau^\la}=
\ov{\ta^\la }\pd_{\ta^\la}$.
 We have
$$
\pd_z g^\lambda( z^\la)=
\pd_z\Gaa^\lambda(z)  I^\la(z), \quad I^\la(z)=\int_{\pd\Om}
\f{h^\lambda(\zeta^\la)-h^\lambda(z_*^\la)
}{(\zeta^\la-z^\la)^2}d\zeta^\la.
$$
Note that  by   the product role,
$\tilde\pd_{\ta^\la}^lf^\la(\zeta^\la)$ involves derivatives of order at most
$l$ in $\zeta^\la$
and $f^\la(\zeta^\la)$. Thus $|h^\la|_\all\leq
C_{k+1+\all}|f^\la|_{l+\all}$.  We can verify that
\aln
|I^\la(z)|&\leq C|h^\la|_\all
\int_{\pd\Om}|\zeta^\la-z^\la|^{\all-2}\, d\sigma
\leq C'|h^\la|_\all\int_{\pd\Om}|\zeta -z |^{\all-2}\, d\sigma
\\ &\leq C_{k+1+\all}|f^\la|_{l+\all}\dist(z,\pd\Om)^{\all-1}.
\end{align*}
Combining with the first inequality
 in \re{clc+}, we get its   second inequality   for $j=0$. Also, by \re{flze}
\aln
&|I^{\mu}(z)-I^\lambda(z)|
\leq C(\|h^\mu-h^\la\|_{\all,0}+|h^\lambda|_{\all}|\Gamma^{\mu}-\Gamma^\lambda|_{1})
\int_{\pd\Om}|\zeta-z|^{\all-2}\, d\sigma(\zeta).
\end{align*}
Thus, $|I^{\mu}(z)-I^\lambda(z)|\leq C\dist(z,\pd\Om)^{\all-1}$. For $l\leq k+1$ we obtain
\aln
\|g^{\mu}-g^{\lambda}\|_{\all,0}
\leq C(\|h^\mu-h^\la\|_{\all,0}+|h^\lambda|_{\all}|\Gamma^{\mu}-\Gamma^\lambda|_{k+1+\all}).
\end{align*}
Note that $\|h^\mu-h^\la\|_{\all,0}\leq C_{k+1+\all}(
\|f^{\mu}-f^\lambda\|_{k+1+\all,j}+|f^\la|_{k+1+\all}|\gaa^{\mu}-\gaa^\lambda|_{l+\all})$
for $l\leq k+1$.
This gives us \re{clc+2}   for $j=0$.

Assume that $j>0$ and \re{clc+2} is valid when $j$ is replaced by $j-1$.
 Here we need a crucial cancellation.
By \re{dsot},  we have  $d\gaa^\la=
\pd_{\tau}\gaa^\la \, d\sigma$, i.e., $d\zeta^\la=\pd_{\ta_\zeta}\zeta^\la\, d\sigma(\zeta)$.
Thus
\aln
\pd_{\la} (\cL C^\lambda f(z^\la))
&=\f{1}{2\pi i}\int_{\zeta\in\pd\Om} \f{\pd_\la[\pd_{\ta_\zeta} \zeta^{\lambda}
 f^\lambda(\zeta^\la)]
}{\zeta^\la-z^\la}d\sigma(\zeta)\\
&\quad -
\f{1}{2\pi i}   \int_{\zeta\in\pd\Om}\f{( \pd_{\la}\zeta^\la-\pd_\la z^\la )
f^\la(\zeta^\la)}
{(\zeta^\la-z^\la)^2}d\zeta^\la.\nonumber
\end{align*}
We apply  integration by parts to   the second term and write the above as
\aln
\f{1}{2\pi i}\int_{\pd\Om} \f{\pd_\la[\pd_{\ta_\zeta} \zeta^{\lambda} f^\lambda(\zeta^\la)]
}{\zeta^\la-z^\la}d\sigma(\zeta) -
\f{1}{2\pi i}   \int_{\pd\Om}\f{\pd_{\tau_\zeta}[( \pd_{\la}\zeta^\la-\pd_\la z^\la )
f^\la(\zeta^\la)]}
{\zeta^\la-z^\la}d\sigma(\zeta).
\end{align*}
Cancelling two second-order derivatives,  we arrive at
\al\label{coun1}
\pd_{\la} (\cL C^\lambda f(z^\la))
&=\f{1}{2\pi i}\int_{\zeta\in\pd\Om} \f{\pd_\la( f^\lambda(\zeta^\la))-
\pd_\la\zeta^\la(\pd_{\tau_\zeta}\zeta^\la)^{-1} \pd_{\tau_\zeta} ( f^\lambda(\zeta^\la))
}{\zeta^\la-z^\la}d\zeta^\la\\
\nonumber
&\quad +
\f{\pd_\la z^\la}{2\pi i}   \int_{\zeta\in\pd\Om}
\f{(\pd_{\tau_\zeta}\zeta^\la)^{-1}\pd_{\tau_\zeta}
(f^\la(\zeta^\la))}
{\zeta^\la-z^\la}d\zeta^\la.
\end{align}
Now \re{clc+2} follows from the induction hypothesis
 where $(j,l)$ is replaced with $(j-1,l-1)$.
By a simpler computation,  estimates \re{clc+} for $j>0$ follow from \re{coun1} too.

The   proof for real analyticity in
 \rl{jumpf} cannot be applied to the parameter case,  as generally we cannot normalize two
 differential operators simultaneously.  Instead, we will prove it by estimating
 Taylor coefficients.
We start with
\al\label{coun2}
\pd_{z} (\cL C^\lambda f(z^\la)) &=\f{\pd_z z^\la}{2\pi i}  \int_{\zeta\in\pd\Om}
\f{\ov{\tau_{\zeta}}\pd_\tau(f^\lambda(\zeta^\la))
}{\zeta^\la-z^\la}d\zeta^\la.
\end{align}
Analogous formula holds for $\pdoz(\cL C^\lambda f(z^\la))$.  By \re{coun1}-\re{coun2},
we express
\eq{coun3}
\pd_{z}^{i}\pdoz^j\pd_\la^{k-i-j}\cL C_+^\la f^\la(z^\la)=
\sum_{l=1}^{N_{i,j;k}} P_{i,j,k,l}^\la(z^\la)\cL C_+^\la\{ Q_{i,j,k ,l}^\la(\zeta^\la)\}(z^\la).
\eeq
Here $P_{i,j,k,l}^\la(z^\la)Q_{i,j,k,l}^\la(\zeta^\la)$ is the product of elements of the form
\ga
\nonumber\pd_z^{a_1 }\pdoz^{b_1}\pd_\la^{c_1+1} z^\la, \quad\pd_z^{a_2+1 }
\pdoz^{b_2}\pd_\la^{c_2}{z^\la},
\quad\pd_z^{a_3}\pdoz^{b_3+1}\pd_\la^{c_3}  {z^\la}; \\
\nonumber 
\pd_{\ta }^{a_4}\pd_\la^{b_4+1}\zeta^\la,\quad
 \pd_{\ta }^{a_5}\pd_\la^{b_5}(\pd_\ta\zeta^\la)^{-1},\quad\pd_{\ta }^{a_6}\ov{\tau_\zeta};
 \\
\nonumber\pd_{\ta }^{a_7}\pd_\la^{b_7} ( f^\la(\zeta^\la));\quad
  L(\zeta,z,\la)=(\pd_\la z^\la,\pd_{ z}{z^\la},\pd_{\ov z}{z^\la},
 \pd_\la\zeta^\la,(\pd_\ta\zeta^\la)^{-1}, \ov{\tau_\zeta}).
\end{gather}
Let us explain how the above terms are used.
We introduce $L$ that includes all first-order derivatives   appearing
in \re{coun1}-\re{coun2} and $\pd_{\ov z}z^\la$ in the formulae
analogous to $\re{coun2}$, except those of $f$. To count   the total of the orders of derivatives
 efficiently, we will   count
 the first-order derivatives  appearing  in $L$  separately,
when \re{coun1} or \re{coun2} is used each time.  Set $b_6=c_4=\cdots=c_7=0$. For the purpose
of  counting, we duplicate
  the above
terms associated to $(a_n,b_n,c_n)$ for $n<7$ and denote by $m_n$
the   number of the  copies associated to $(a_n,b_n,c_n)$
that appear in $P_{i,j,k,l}^\la Q_{i,j,k,l}^\la$. Since  $\pd_{\ta }^{a_7}\pd_\la^{b_7}
( f^\la(\zeta^\la))$ appears   once
in  $P_{i,j,k,l}^\la Q_{i,j,k,l}^\la$, we set
  $m_7=1$.
  By an abuse of notation, we have
not expressed the dependence of $m_n$   on $i,j,
k$, $l$, $ a_{n'}$, $b_{n'}$ and $c_{n'}$. Nevertheless, we have
$$ 
d_{k}=\max_{i,j,l}\sum_{n=1}^7m_n(a_n+b_n+c_n)\leq k,\quad
\prod_{n=1}^7 (a_n!b_n!c_n!)^{m_n}\leq k!.
$$ 
Since $z^\la$ and $ f^\la(\zeta^\la)$ are real analytic, we have
\gan
|L|\leq C_0, \quad |\pd_z^{a}\pdoz^{b}\pd_\la^{c} z^\la|\leq (a+b+c-1)! C_0^{a+b+c}, \quad
a+b+c>0,\\
|\pd_{\ta }^{a}\pd_\la^{b}(\pd_\la \zeta^\la,(\pd_\ta\zeta^\la)^{-1},
\ov{\tau_\zeta}, f^\la(\zeta^\la))|\leq (a+b)!C_0^{a+b+1}.
\end{gather*}
Here the last inequality is obtained by using real analytic parameterization in arc-length.
Thus, the product of the terms
 in $P_{i,j,k,l}^\la Q_{i,j,k,l}^\la$, excluding these in $L$, is bounded in sup norm by
\eq{prod17}
\prod_{n=1}^7(a_n!b_n!c_n!)^{m_n}C_0^{m_n(a_n+b_n+c_n)}\leq C_0^{k}k!.
\eeq
Next, we count $l_k$, the maximum number of first-order   derivatives in $L$ which
appear in each $P_{i,j,k,l}^\la Q_{i,j,k,l}^\la$ as $i,j$ and $l$ vary. From \re{coun1},
taking one derivative in $\pd_\la$ produces
at most two terms in $L$; from \re{coun2}, taking one derivative in $z$ or
$\ov z$ produces two terms in $L$. Therefore,
$l_k\leq 2k+1.$  Thus, the product of the terms  in $L$ that appear
 in $P_{i,j,k,l}^\la Q_{i,j,k,l}^\la$ is bounded in sup norm by $C_0^{l_k}\leq C_0^{2k+1}$.
 Combining with \re{prod17}, we get
\eq{plql}
|P_{i,j,k,l}^\la Q_{i,j,k,l}^\la|_0\leq C_0^{l_k}
\prod_{n=1}^7(a_n!b_n!c_n!)^{m_n}C_0^{m_n(a_n+b_n+c_n)}\leq C_0^{3k+1}k!.
\end{equation}
 Finally, we count
the maximum number of terms in \re{coun3}.
  When we take one derivative in $\la$ on $\cL C^\la_+f$, we get   $3$ terms by using
   \re{coun1};
when we take one derivative in $z$ or $\ov z$ on $\cL C^\la_+f$, we   have just one term
in \re{coun2}. Therefore,
\eq{nkdf}
  N_k\df \max_{i,j} N_{i,j;k}\leq 3^k.
\eeq

We have
$|Q_{i,j,k-1,l}(\cdot,\cdot)|_{\all}\leq C_1|Q_{i,j,k-1,l}(\cdot,\cdot)|_{1}$. Taking a
 $\zeta$-derivative
on $P_{i,j,k-1,l}^\la\cdot Q_{i,j,k-1,l}^\la$ introduces at most $N_{k}$ terms of
the form $P_{i',j',k,l'}^\la\times Q_{i',j',k,l'}^\la$.    This shows that
\aln
|P_{i,j,k-1,l}(\cdot)|_{0}
|Q_{i,j,k-1,l}(\cdot)|_{\all}&\leq N_{k}C_1\max_{i',j',l'}|P_{i',j',k,l'}(\cdot)|_{0}
|Q_{i',j',k,l'}(\cdot)|_{0}
\\  &\leq C_13^{k}{k}!C_0^{3k+1}, \quad (\text{by \re{plql}, \re{nkdf}}).
\end{align*}
By \re{coun3},\re{clc+} and the above inequality, we obtain
\aln
|\pd_{z}^{i}\pdoz^j\pd_\la^{k-1-i-j}\{\cL C_+^\la f^\la(z^\la)\}|&
\leq N_{k-1}C_{1,0}|P_{i,k-1-i,l}(\cdot,\cdot)|_{0}
|Q_{i,k-1-i,l}(\cdot,\cdot)|_{\all} \\
&\leq
C_1C_{1,0}3^{2k-1}k!C_0^{3k+1}.
\end{align*}
Using
 $k!\leq i!j!(k-1-i-j)!3^{k-1}k$, we obtain
 the desired estimate on Taylor coefficients to show
 that $\cL C^\la_+ f^\la(z^\la)$ is real analytic on $\ov\Om\times[0,1]$.

It is clear that $W^\la f(z)$ is real analytic on $\Om\times [0,1]$.
 We need to show that it is real
analytic near  $(z_1,\la_1)\in  \ov\Om\times[0,1]$ for $z_1\in\pd\Om$.
 We use a local real analytic coordinates to find a real analytic function $\var(z_0,z,t)$
 defined on $U\times U\times[0,1]$
  such
  that $\var(z_0,z,0)=z_0$ and $\var(z_0,z,1)=z$, where $U$ is an open set containing $z_1$.
   Moreover, $\var(z_0,z,t)$ is
  in $\Om$ when $t\in(0,1)$
  and $z_0,z$ are in $U\cap\ov\Om$. Fix $z_0\in U\cap\Om$ and vary $z\in U\cap\ov\Om$. We have
 \aln
 W^\la f(\Gaa^\la(z))=W^\la f(\Gaa^\la(z_0))
   +2\RE\int_0^1 \{\pd_z((W^\la f)\circ\Gaa^\la)\}
 (\var(z_0,z,t))\pd_t\var(z_0,z,t)\, dt.
 \end{align*}
Since $\pd_{z^\la}W^\la f=-2iC^\la_+(\ov\ta^\la f^\la)$ and
$\pd_{\ov z^\la}W^\la f^\la$ are  real analytic in $(z,\la)\in\ov\Om
\times[0,1]$, then $\pd_zW^\la f$ is real analytic in $z$ and $\la$
 by the chain rule.
Thus, the   integrand in the above integral is   real analytic in $(z_0,z,\la,t)
\in (U\cap\ov\Om)^2\times[0,1]^2$. This shows that $W^\la f^\la(z^\la)$ is real analytic
in $(z,\la)\in\ov\Om
\times[0,1]$.
\end{proof}

We have seen that the kernels of $Uf$ and $Wf$
 behave better than that of $\cL Cf$ for spaces
of continuous functions. In the parameter case, we have the following analogue of \re{u+-f} and
\rl{jumpf} (iii).
\begin{prop}\label{c0paU}
Let $\Gamma^{\lambda}$ embed $\ov\Om$ onto $\ov{\Om^\la}$ with $\Gaa\in\cL C_*^{1,j}$. Suppose
that $\Gaa^\la$ preserve the orientation.
 Assume that $f\in\cL C_*^{0,j}(\pd\Om_\gaa)$. Then $Wf\in\cL C_*^{0,j}(\ov\Om_\Gaa)$
 and
\aln
\pd_\la^jW^\la f(z)&=\f{1}{\pi}\sum_{i=0}^j\binom{j}{i}\int_{\pd\Om}\pd_\la^i
(|\pd_\ta\gaa^\la(\zeta)|f^\la(\zeta^\la))\pd_\la^{j-i}
\{\log|\zeta^\la-z^\la|\}\, d\sigma(\zeta), \  z\in\ov\Om.
\end{align*}
 Assume further that $\pd\Om\in\cL C^{1+\all}$ and $\Gaa\in\cL B_*^{1+\all,j}$.
 Then $Uf\in\cL C_*^{0,j}(\ov\Om_\Gaa)$ and  on $\pd\Om$
\aln 
\pd_\la^jU_+^\la f(z)&=\f{1}{\pi}\sum_{i=0}^{j}\binom{j}{i}\int_{\pd\Om}
\pd_\la^i\{f^\la(\zeta^\la)\}\pd_\la^{j-i}
\{|\pd_\ta\gaa^\la(\zeta)|\pd_{\ta^\la}\arg(\zeta^\la-z^\la)\}\, d\sigma(\zeta)\\
\nonumber&\quad +
\pd_\la^j (f^\la (z^\la)).
\end{align*}
 Under
 the norms defined by \rea{ukaj}-\rea{norm2} and \rea{norm3-}-\rea{norm4},
\gan 
 |W_+ f|_{0,j}\leq
C_{1,j}^*|f|_{0,j},\
|W_+^{\mu} f-W_+^\lambda f|_{0,j}\leq
C_{1,j}^*(|\Gamma^{\mu}-\Gamma^{\lambda}|_{1,j}|f^\la|_{0,j}
+  |f^\mu-f^\la|_{0,j}).
\end{gather*}
In particular, if $\pd\Om\in\cL C^{k+1+\all}$,
 $\Gaa\in\cL B^{k+1+\all,j}(\ov\Om)$ and
$f\in\cL B^{k+\all,j}(\pd\Om_\gaa)$
with $k\geq j$, then $W_+f\in\cL B^{k+1+\all,j}(\ov{\Om}_\Gaa)$; the same assertion holds
if $\cL C$ substitutes for $\cL B$.
\end{prop}
\begin{proof} We write $d\sigma^\la=a^\la\, d\sigma$ on $\pd\Om$. Recall from \re{dsot}
that $a^\la(z)=|\pd_\ta\gaa^\la(z)|$.
Fix $z_0\in\pd\Om$.
For $z\in\Om$ we have $ \pd_\la^jU^\la f(z)= \sum\binom{j}{k}I_{k}(z)$ with
\aln
&  I_{k}^\la(z)=\f{1}{\pi}\int_{\pd\Om}\pd_\la^{k}\{f^\la(\zeta^\la)\}
\pd_\la^{j-k}\{a^\la(\zeta)\pd_{\ta_\zeta^\la} \arg(z^\la-\zeta^\la)\}\, d\sigma.
\end{align*}
 Using
\re{flze2} for $\gaa\in\cL B^{1+\all,j}_*$, we get for $l\leq j$
\ga\label{1ad2-}
\int_{\zeta\in\pd\Om,|\zeta-z_0|<\e}|\pd_\la^{l}\pd_{\ta_\zeta^\la}\arg(z^\la-\zeta^\la)|\, d\sigma^\la\leq C\e^\all,\quad z_0\in\pd\Om,\\
\int_{\zeta\in\pd\Om}|\pd_\la^{l}\pd_{\ta_\zeta^\la}\arg(z^\la-\zeta^\la)|\, d\sigma^\la<C,\quad
z\in\cc.
\label{1ad2}\end{gather}
Then, by analogue of \re{s3s5},  from $f\in\cL C^{0,j}_*$, $\gaa\in\cL B^{1,j}_*$
and \re{1ad2} we get
\al\label{dlac}
&\lim_{\pd\Om\not\ni z\to z_0}\int_{\pd\Om} \{\pd_\la^{k}\{f^\la(\zeta^\la)\}-\pd_\la^{k}\{f^\la(z_0^\la)\}\}
\pd_\la^{j-k}\{a^\la(\zeta)\pd_{\ta_\zeta^\la} \arg(z^\la-\zeta^\la)\}\, d\sigma(\zeta)\\
&\qquad =\int_{\pd\Om} \{\pd_\la^{k}\{f^\la(\zeta^\la)\}-\pd_\la^{k}\{f^\la(z_0^\la)\}\}
\pd_\la^{j-k}\{a^\la(\zeta)\pd_{\ta_\zeta^\la} \arg(z_0^\la-\zeta^\la)\}\, d\sigma(\zeta),
\nonumber
\end{align}
where the convergence is uniform in $\la$.
Let $0<l\leq j$.
Note that $$\int_{\pd\Om}
\pd_\la^{l}\{a^\la(\zeta)\pd_{\ta_\zeta^\la} \arg(z^\la-\zeta^\la)\}\, d\sigma(\zeta)=0,\quad
z\not\in\pd\Om.
$$
 For $z=z_0\in\pd\Om$ and $\gaa\in\cL C_*^{1,j}$, the last integral equals
\aln
&\lim_{\e\to0}
\int_{\zeta\in\pd\Om,|\zeta-z_0|>
\e}
\pd_\la^{l}\{a^\la(\zeta)\pd_{\ta_\zeta^\la} \arg(z_0^\la-\zeta^\la)\}\, d\sigma\\
&\qquad\quad =\lim_{\e\to0}\pd_\la^{l}\{\pi-\arg(\gaa^\la\circ\hat\gaa(\e_2)-\gaa^\la\circ\hat\gaa(0))-
\arg(\gaa^\la\circ\hat\gaa(0)-\gaa^\la\circ\hat\gaa(\e_1 ))\}\\
&\qquad\quad =\lim_{\e\to0}\pd_\la^{l}\left\{\pi-\arctan
\f{\int_{0}^1(y^\la\circ\hat\gaa)'( r\e_2)\, dr}{\int_{0}^1(x^\la)'(r\e_2)\, dr}+
\arctan\f{\int_{0}^1(y^\la)'(r\e_1)\, dr}{\int_{0}^1(x^\la)'( r\e_1)\, dr}\right\}=0.
\end{align*}
 Here we have used
 $\arg(x+iy)=\arctan(y/x)\mod\pi$ and a local $\cL C^1$
 parameterization $\hat\gaa$ of $\pd\Om$ near $z_0$
 with $\hat\gaa(0)=z_0$. Also,
  $\pd\Om$ intersects $\{|\zeta-z_0|=\e\}$ at $\hat\gaa(\e_1)$,
   $\hat\gaa(\e_2)$ for $\e$ sufficiently small. We have also assumed without loss of generality that $\pd_\ta x^\la(z_0)\neq0$.
Expanding both sides of \re{dlac} we get the formula for $\pd_\la^jU^\la_+f$. Combining
the formula with \re{1ad2-}, we see that $U^\la_+f(z_0^\la)$ is continuous in $\la$.
 Then, the uniform
convergence of $U^\la f(z^\la)$ as $z\to z_0$ yields   $U_+f\in\cL C_*^{0,j}(\ov\Om_\Gaa)$.

Write $\pd_\la^jW^\la f$ as a linear combination of
$$
h_{j_1j_2}^\la(z)=\int_{\pd\Om}\pd_\la^{j_1}\{a^\la(\zeta) f^\la(\zeta^\la)\}
\pd_\la^{j_2}\log|z^\la-\zeta^\la|\, d\sigma, \quad j_1+j_2=j.
$$
Using \re{f1ze+-} and \re{dtsc}, we obtain
$$
|\pd_\la^{j_2}\log|\zeta^\la-z^\la||\leq C, \  j_2>0; \quad
|\log|\zeta^\la-z^\la||\leq C(|\log|t-s||+1),
$$
where $z=\hat\gaa(s)+h\nu(s)$, $\zeta=\hat\gaa(t)$, and $\hat\gaa$ is a parameterization
of $\pd\Om$. We conclude easily that $h_{j_1j_2}^\la(z)$ are continuous on $\ov\Om\times[0,1]$.
This verifies the   formula for $\pd_\la^jW^\la_+f$.

By the  formulae of $\pd_\la^jW^\la f$, we obtain    $W_+f\in
\cL C_*^{0,j}(\ov\Om_\Gaa)$ and the desired estimate for $W_+f$  by \re{f1ze+-}-\re{f1ze+},
  $\int_{\pd\Om^\la}|\log|\zeta^\la-z^\la||\, d\sigma^\la<C$ and $\dist(z,\pd\Om)
  \int_{\pd\Om}|\zeta-z|^{-2}\,
d\sigma<C$ for $z\in\Om$. The assertion on higher order derivatives following from
 $\pd_zWf=-i\cL C[ \ov\ta f]$
and \rp{c0pa}.
 \end{proof}

To prepare our estimates in section~\ref{sec7}
for integral equations with parameter, we use the rest of the section to
extend \rl{kta} to the parameter case.

For convenience, we will  use the following  difference operators
$$
\delta^{\la\mu}f=f^{\mu}-f^\la, \quad
\del_{t_1t_2}g=g(t_2)-g(t_1).
$$
For clarity, we will also write the above as $\del^{\la\mu}f^\bullet$
and $\del_{t_1t_2}g(\cdot)$ where $\cdot$ and $\bullet$ indicate the  variables used in the operators.
  Both satisfy the product rule to the extent that
\gan
\delta^{\la\mu}(f g)=\delta^{\la\mu}f g^\nu+
 f^\nu\delta^{\la\mu}g,\quad
\del_{t_1t_2}(fg)
=\del_{t_1t_2}fg(t_3)+f(t_3)\del_{t_1t_2}g,
\end{gather*}
where $\nu=\la$ or $\mu$ (and two $\nu$'s are different),
and  $t_3=t_1$ or $t_2$.
 For $\gaa\in\cL B^{k+1+\all,j}(\pd\Om)$, with  the above notation we have
\gan
K_{k-j,j}^\la(\zeta^\la,z^\la)=
\pd_\la^j\pd_{\tau^\la_{z}}^{k+1-j}\{\arg(\zeta^\la-z^\la)\},\qquad
\zeta^\la,z^\la\in\pd\Om^\la,\\
 \delta^{\la\mu}K_{k-j,j}(\zeta,z)=
K_{k-j,j}^\mu(\zeta^\mu,z^\mu)-K_{k-j,j}^\la(\zeta^\la,z^\la), \quad \zeta,z\in\pd\Om.
\end{gather*}
\le{ktap} Let $\gaa^\la$ embed  $\pd\Om$ onto $\pd\Om^\la$ with
 $\gaa\in\cL C^{k+1+\all,j}$ and $k\geq j$. Then on $\pd\Om^\la\times\pd\Om^\la$ and off its diagonal,
\ga
\bigl|K_{k-j,j}^\la(\zeta^\la,z^\la)
\bigr|
\leq C_{k+1+\all,j}|\zeta-z|^{\all-1},
\label{kta3p--}\\
\bigl|K_{k-j,j}^\la(\zeta^\la,z_2^\la)
-K_{k-j,j}^\la(\zeta^\la,z_1^\la)
\bigr|
\leq C_{k+1+\all,j}\frac{|z_2-z_1|^\all}{|\zeta-z_1|},\label{kta3np}\\
\bigl|\delta^{\la\mu}K_{k-j,j}(\zeta,z)
\bigr|
\leq C_{k+1+\all,j}\|\gaa^\mu-\gaa^\la\|_{k+1+\all,j}|\zeta-z|^{\all-1},
\label{kta3p-}\\
\bigl|\delta_{z_1z_2}\delta^{\la\mu}K_{k-j,j}(\zeta,\cdot)
\bigr|
\leq C_{k+1+\all,j}\|\gaa^\mu-\gaa^\la\|_{k+1+\all,j}\frac{|z_2-z_1|^\all}{|\zeta-z_1|},
\label{kta3p}\\
\bigl|\del^{\la\mu}K_{0,j}(z_2,\zeta)-\del^{\la\mu}K_{0,j}(z_1,\zeta)
\bigr|
\leq C_{1+\all,j}^*\|\gaa^\mu-\gaa^\la\|_{1+\all,j}\frac{|z_2-z_1|}{|\zeta-z_1|^{2-\all}},
\label{nkta3p}
\end{gather}
where
 \rea{kta3np}, \rea{kta3p} and \rea{nkta3p} are for  $|\zeta-z_1|>2|z_2-z_1|$.
\ele
\begin{proof}   It suffices to verify \re{kta3p--}-\re{kta3p}
for $\zeta, z_1, z_2,z$ near a point $w\in\pd\Om$ at which $|\pd_{\ta_{w}}x^\la|\neq 0$.
We may assume that $w=0$ and $\pd\Om$ contains $(-1,1)+i0$.
We may  assume that $\|\gaa^\mu-\gaa^\la\|_{1,0}$ is sufficiently small; otherwise,
 \re{kta3p-}-\re{kta3p}
follow  from \re{kta3p--}-\re{kta3np}. We may further assume  that $\pd\Om^\la$
is embedded through $\gaa^\la(t)$ such that,
for $|t|\leq 1$, $(x^\la)'>1/C$ and $|(x^\la(s)-x^\la(t))^{-1}|\leq C|s-t|^{-1}$.
In the following, we assume that $s,t,t_1$
and $t_2$ are in $(-1,1)$.
Define $q^\la(s,t)=(y^\la(s)-y^\la(t))/{(x^\la(s)-x^\la(t))}$ and \gan
\hat K_{k-j,j}^\la(s,t)=
\pd_\la^j\pd_{t}^{k+1-j}q^\la(s,t),\quad
 \delta^{\la\mu}\hat K_{k-j,j}(s,t)=
\hat K_{k-j,j}^\mu(s,t)-\hat K_{k-j,j}^\la(s,t).
\end{gather*}
 By \re{tazf}, we have
$(\pd_{\ta^\la}u^\la)(\gaa^\la(t))=|\pd_t \gaa^\la|^{-1}\pd_t(u^\la(\gaa^\la))$. Hence
$$ 
K_{k-j,j}^\la(\gaa^\la(s),\gaa^\la(t))=\sum_{j'\leq j,k'\leq k}Q_{j'k'}^\la(t)
\hat K_{k'-j',j'}^\la(s,t).
$$ 
Here $Q_{j'k'}^\la$  are $C^\infty$ functions in  $|\pd_t\gaa^\la|^{-1},
\pd_t\gaa^\la,\ldots, \pd_t^{k+1}\gaa^\la$, and
$$
q^\la(s,t)=\f{\int_0^1  (y^\la)'(s+r(t-s))\, dr}{\int_0^1  (x^\la)'(s+r(t-s))\, dr}.
$$   We have for $Q^\la=Q^\la_{j'k'}$
\gan
|Q^\la(t)|\leq C_{k+1,j}, \quad |Q^\la(t_2)-Q^\la(t_1)|\leq C_{k+1+\all,j}|t_2-t_1|^\la,\\
|(Q^\mu-Q^\la)(t_2)-(Q^\mu-Q^\la)(t_1)|\leq C_{k+1+\all,j}
\|\gaa^\mu-\gaa^\la\|_{k+1+\all,j}|t_2-t_1|^\la.
\end{gather*}
Therefore, to show \rea{kta3p--}-\rea{nkta3p}, it suffice to verify them when
 $K^\la(\zeta^\la,z^\la),\zeta$, and $z$
are replaced by $\hat K^\la(s,t),s$, and $t$, respectively.

Recall that $\hat R_kf(s,t)=\int_0^1\{f^{(k)}(t+r(s-t))
-f^{(k)}(t)\}\, dr.$
We apply  formula \re{keli} and obtain for $(x^\la(s)-x^\la(t))q^\la(s,t)=y^\la(s)-y^\la(t)$,
\aln
(x^\la(s)&-x^\la(t))^{k+2-j}\pd_t^{k+2-j}q^\la(s,t)\\
&=(s-t)^{k+1-j}\Bigl\{
P_0^\la\hat R_{k+1-j}y^\la(s,t)+\sum_{i=1}^{k+1-j}
P_i^\la
\hat R_ix^\la(s,t)\Bigr\},
\end{align*}
where $P_i^\la(s,t)$ are polynomials in $s-t$,
$\pd_t(x^\la,y^\la)$, $\ldots$, $\pd_t^{k+1-j}(x^\la,y^\la)$,
$\pd_s^{k+1-j}\gaa^\la$, and $ x^\la(s)-x^\la(t)$.
Hence $\pd_\la^j\pd_t^{k+1-j}q^\la(s,t)$ is a linear combination of
$L^\la(s,t)$ of the form
$$
(s-t)^{k+1-j}\frac{\pd_\la^{j_1}(x^\la(s)-x^\la(t))\cdots\pd_\la^{j_a}(x^\la(s)-x^\la(t))}
{(x^\la(s)-x^\la(t))^{k+2-j+a}}\pd_\la^{j_b}P_i^\la\pd_\la^{j_c}\hat R_i(x^\la,y^\la)(s,t).
$$
Here $j_1+\cdots+j_a+j_b+j_c=j$ and $i+j\leq k+1$. Assume  that $j'\leq j$.
 We first bound each term in $L^\la(s,t)$
as follows:
\ga\label{dljp1}
|(\pd_\la^{j}\gaa^\la(s)-\pd_\la^{j}\gaa^\la(t))|\leq C|s-t|,\quad
|(x^\la(s)-x^\la(t))^{-1}|\leq C|s-t|^{-1}, \\
\label{dljp2}|\pd_\la^{j}\hat R_i\gaa^\la(s,t)|\leq
C|s-t|^{\all},\quad  |\pd_\la^{j}
P_i^\la(s,t)|\leq C, \quad i+j\leq k+1.
\end{gather}
By \re{dljp1}-\re{dljp2} we get $|L^\la(s,t)|\leq C|s-t|^{\all-1}$, which gives us \re{kta3p--}.
We now assume that $|s-t_2|\geq2|t_2-t_1|$. Then
\ga
\label{dljp3}
|(s-t_2)-(x-t_1)|\leq |s-t_1|^{1-\all}|t_2-t_1|^\all,\\
\label{dljp3+}|\pd_\la^{j}x^\la(t_2)-\pd_\la^{j}x^\la(t_1)|\leq
C |s-t_1|^{1-\all}|t_2-t_1|^\all,\\
|(x^\la(s,t_2))^{-1}-(x^\la(s,t_1))^{-1}|\leq C|s-t_1|^{-1-\all}|t_2-t_1|^\all,\\
\label{dljp4}|\pd_\la^{j}(P_i^\la,\hat R_i \gaa^\la)(s,t_2)-\pd_\la^{j}
(P_i^\la,\hat R_i \gaa^\la)(s,t_1)|
\leq C\|\gaa\|_{k+1+\all,j}|t_2-t_1|^\all.
\end{gather}
Here $i+j\leq k+1$. Comparing the exponents in \re{dljp1}-\re{dljp2} with the ones in
 \re{dljp3}-\re{dljp4}, and using the exponent in \re{kta3p--},
we obtain   \re{kta3np} by a simple computation.

 Applying $\del^{\la\mu}$ to each term in $L^\la(s,t)$, we get
\ga
\label{dljp5}|\delta^{\la\mu}( \pd_\bullet^{j}x^\bullet (s)-
\pd_\bullet^{j}x^\bullet (t))|\leq C\|\gaa^\mu-\gaa^\la\|_{k+1,j}|s-t|,\\
\label{dljp5+}|\delta^{\la\mu}(( x (s)- x (t))^{-1})|
\leq C\|\gaa^\mu-\gaa^\la\|_{k+1,j}|s-t|,\\
\label{dljp6-}|\delta^{\la\mu}(\pd_\bullet^{j}P_i^\bullet(s,t))| \leq
C\|\gaa^\mu-\gaa^\la\|_{k+1,j}, \\
\label{dljp6}
\ |\delta^{\la\mu}((\pd_\bullet^{j}\hat R_i^\bullet\gaa) (s,t))|\leq
C\|\gaa^\mu-\gaa^\la\|_{k+1+\all,j}|s-t|^{\all},\quad i+j\leq k+1.
\end{gather}
We see that \re{dljp5}-\re{dljp6} and \re{dljp1}-\re{dljp2} differ by
a factor $\|\gaa^\mu-\gaa^\la\|_{k+1+\all,j}$, as    \re{kta3p--} and \re{kta3p-} do. A simple
computation gives us
 \re{kta3p-}.
We have
\aln
\del^{\la\mu}(x(s)-x(t))^{-1}=(x^\mu(t)-x^\mu(s))^{-1}(x^\la(s)-x^\la(t))^{-1}
\del^{\la\mu}(x(s)-x(t)).
\end{align*}
Note that $\del_{t_1t_2}$ and $\del^{\la\mu}$ commute. Assume that $|s-t_1|>2|t_2-t_1|$. Then
\aln
&|\del_{t_1t_2}\del^{\la\mu}(x(s)-x(\cdot))^{-1}|\leq 2C|t_2-t_1||s-t_1|^{-2}|
\gaa^\mu-\gaa^\la|_{0}|s-t_1|\\
&\quad\quad +C|s-t_1|^{-2}|\gaa^\mu-\gaa^\la|_{1}|t_2-t_1|\leq 3C
|\gaa^\mu-\gaa^\la|_{1}|s-t_1|^{-1-\all}|t_2-t_1|^\all.
\end{align*}
Therefore,
 \ga
\label{dljp7} |
\delta^{\la\mu}\delta_{t_1t_2} (\pd_\bullet^{j}\gaa^\bullet(s)-
\pd_\bullet^{j}\gaa^\bullet (\cdot)) |
\leq C|\gaa^\mu-\gaa^\la|_{1,j}|s-t_1|^{1-\all}|t_2-t_1|^\all,\\
\label{dljp7+}
|\delta^{\la\mu}\delta_{t_1t_2}(x(s,\cdot))^{-1}|\leq C|\gaa^\mu-\gaa^\la|_{1,0}
|s-t_1|^{-1-\all}|t_2-t_1|^\all,\\
\label{dljp8} |
\delta^{\la\mu}\delta_{t_1t_2}\pd_\bullet^{j}( P_i^\bullet,\hat R_i \gaa^\bullet)(s,\cdot)|
\leq C\|\gaa^\mu-\gaa^\la\|_{k+1+\all,j}|t_2-t_1|^\all.
\end{gather}
Here $ i+j\leq k+1$. Comparing
\re{dljp3+}-\re{dljp4} with \re{dljp1}-\re{dljp2}
and \re{dljp7}-\re{dljp8} with \re{dljp3}-\re{dljp4},
we see that applying $\del^{\la\mu}$   introduces
a factor $|\gaa^\mu-\gaa^\la|_{k+1+\all,j}$, as shown in \re{kta3np} and \re{kta3p}.
A simple computation gives us   \re{kta3p}.

To verify \re{nkta3p}, we   start  with
$$
\hat K^\la(t,s)=\IM\f{\ov{\pd_s\gaa^\la}\int_t^s(\pd_r\gaa^\la -\pd_s\gaa^\la)\, dr}
{|\pd_s\gaa^\la||\gaa^\la(s)-\gaa^\la(t)|^2}.
$$
Then $\pd_\la^j\hat K^\la(t,s)$ is a linear combination of $J^\la(s,t)=
A_1^\la(s)A_2^\la(s,t)A_3^\la(s,t)$  with
\gan
A_1^\la(s)=\pd_\la^{j_1}(|\pd_s\gaa^\la|^{-1}\ov{\pd_s\gaa^\la}),\quad
A_2^\la(s,t)=\pd_\la^{j_2}(|\gaa^\la(s)-\gaa^\la(t)|^{-2}),\\
A_3(s,t)=\pd_\la^{j_3}\int_t^s(\pd_r\gaa^\la -\pd_s\gaa^\la)\, dr
\end{gather*}
and $j_1+j_2+j_3=j$.  Then
$$
|A_1^\la|\leq C, \  |A_2^\la(s,t)|\leq C|s-t|^{-2}, \  |A_3^\la(s,t)|\leq C|s-t|^{1+\all},
\ |J^\la(s,t)|\leq C|s-t|^{\all-1}.
$$
Assume that $|s-t_1|\geq 2|t_2-t_1|$.  As in the proof of \re{kta2}, we write
\aln
& \int_{t_2}^{s}(\pd_r\gaa^\la -\pd_{s}\gaa^\la)\, dr
-\int_{t_1}^{s}(\pd_r\gaa^\la -\pd_{s}\gaa^\la)\, dr\\
&\quad\quad=
 \int_{t_2}^{t_1}(\pd_r\gaa -\pd_{t_1}\gaa)\, dr
- \int_{t_2}^{t_1}(\pd_{t_1}\gaa^\la -\pd_{s}\gaa^\la)\, dr.
\end{align*}
Applying $\pd_\la^j$ and then $\del^{\la\mu}$ to the above, we get
\gan
| A_3(s,t_2)-A_3(s,t_1)|\leq C|\gaa^\la|_{1+\all,j}|t_2-t_1||s-t_1|^{\all},
\\ |\delta^{\la\mu}\del_{t_1t_2} A_3(s,\cdot)|
\leq C|\gaa^\mu-\gaa^\la|_{1+\all,j}|t_2-t_1||s-t_1|^{\all}.
\end{gather*}
We can also verify that
\gan
|A_2^\la(s,t)|\leq C|s-t|^{-2}, \quad|\del^{\la\mu}A_2(s,t)|\leq C|\gaa^\mu-\gaa^\la|_{1,j}
|s-t|^{-2},\\
|\del_{t_1t_2}A_2^\la(s,\cdot)|\leq C\f{|t_2-t_1|^\all}{|s-t|^{1+\all}},
\quad|\del_{t_1t_2}\del^{\la\mu}A_2(s,\cdot)|\leq C|\gaa^\mu-\gaa^\la|_{1,j}
\f{|t_2-t_1|^\all}{|s-t|^{1+\all}}.
\end{gather*}
By a simple computation, we get \re{nkta3p}. \end{proof}

Define $
\cL K_{k-j,j}^{\la}\var(z^\la)=\int_{\pd\Om^\la}\var^\la(\zeta^\la)
K_{k-j,j}^\la(z^\la,\zeta^\la)\, d\sigma^\la$ and
\gan
\cL K_{k-j,j}^{\la*}\var(z^\la)=\int_{\pd\Om^\la}\var^\la(\zeta^\la)
K^\la_{k-j,j}(\zeta^\la,z^\la)\, d\sigma^\la.
\end{gather*}
\pr{ktapi} Let   $k\geq j$ and $0\leq\beta\leq\all$. Let $\all'=\all$
for $\beta>0$ and $\all'<\all$ for $\beta=0$.
Let $\Om$ be a bounded domain in $\cc$ with $\pd\Om\in\cL C^{k+1+\all}$.
Let $\gaa^\la$ embed  $\pd\Om$ onto $\pd\Om^\la$ with
 $\gaa\in\cL B^{k+1+\all,j}$. Suppose that $k\geq j$ and $k+1\geq l\geq j$.
 Then
\ga\label{ktapi1}
|\cL K_{0j'}\var|_{\all,j-j'}\leq C_{1+\all,j}^*|\var|_{0,j-j'},\quad
|\cL K_{0j'}^*\var|_{\all',j-j'}\leq C_{1+\all,j}^*|\var|_{0,j-j'},\\
\label{ktapi2}
|\del^{\la\mu} K_{0j'}\var|_{\all,j-j'}\leq C_{j+1+\all,j}
(|\del^{\la\mu}\var|_{0,j-j'}+|\del^{\la\mu}\gaa|_{j+1+\all,j}),\\
\label{ktapi4}|\cL K_{ij'}^*\var|_{k-i+\all',j-j'}\leq C_{k+1+\all,j}|
\var|_{j-j'+\beta,j-j'}, \\
\label{ktapi6}|\del^{\la\mu}K_{ij'}^{*}\var|_{k-i+\all',j-j'}
\leq C_{k+1+\all,j}(|\del^{\la\mu}\var |_{j-j'+\beta,j-j'}
+\|\del^{\la\mu}\gaa \|_{k+1+\all',j}),\\
\label{ktapi5}|\cL K_{ij'}\var|_{l-i+\all,j-j'}\leq C_{k+1+\all,j}
|\var|_{j-j'+1+\beta,j-j'},
\\
\label{ktapi7}|\del^{\la\mu}K_{ij'}\var|_{l-i+\all',j-j'}
\leq C_{k+1+\all,j}(|\del^{\la\mu}\var|_{j-j'+1+\beta,j-j'}
+\|\del^{\la\mu}\gaa\|_{k+1+\all,j}).
\end{gather}
\epr
\begin{proof}
Recall from \re{dsot} that $a^\la(\zeta)=|\pd_{\ta_\zeta}\gaa^\la|$
 and $d\sigma^\la(\zeta^\la)=a^\la(\zeta)\, d\sigma(\zeta)$
on $\pd\Om$.
Since $|\pd_\la^jK^\la(\zeta^\la,z^\la)|\leq C|\zeta-z|^{\all-1}$, by the mean value theorem
we can change the order of differentiation and integration
in $\pd_\la^{j-j'}\cL K_{0j'}^{\la}\var$. The latter
is then a linear combination of
$$
\int_{\pd\Om}(a^\la(\zeta))^{-1}\pd_\la^{j_1}(a^\la(\zeta)\var^\la
(\zeta^\la))\pd_\la^{j_2}K_{0j'}^\la(z^\la,\zeta^\la)\, d\sigma^\la, \quad j_1+j_2=j-j'.
$$
By replacing $(a^\la(\zeta))^{-1}\pd_\la^{j_1}(a^\la(\zeta)\var^\la
(\zeta^\la))$ with $ \var^\la
(\zeta^\la)$,
it suffices to verify \re{ktapi1}  when $j'=j$; analogously, we only need to verify
\re{ktapi1}-\re{ktapi7} for $j'=j$.

We have $|\cL K^\la_{0,j}\var(z^\la)|\leq C|\var^\la|_0
\int_{\pd\Om}|\zeta-z|^{\all-1}\, d\sigma\leq C'|\var^\la|_0$ and by
\re{kta1}-\re{kta2}
\al
\label{dlmk}
&\Bigl|\int_{\pd\Om}\varphi^\la(\zeta^\la)(K_{0,j}^\la(z_2^\la,\zeta^\la)
-K_{0,j}^\la(z_1^\la,\zeta^\la))\, d\sigma^\la\Bigr|
\\
\nonumber
&\quad\leq
C|a^\la\var^\la|_0\Bigl\{\int_{|\zeta-z_1|<3|z_2-z_1|}
 2 |\zeta-z_1|^{\all-1}\, d\sigma(\zeta)\\
\nonumber
&\qquad +\int_{|\zeta-z_1|>|z_2-z_1|} |z_2-z_1|
 |\zeta-z_1|^{2-\all}\, d\sigma(\zeta)\Bigr\}\leq C'|\var^\la|_0|z_2-z_1|^\all,
\end{align}
which gives us \re{ktapi1}. We have
\aln
\del^{\la\mu}\cL K_{0,j}\var(z)&=\int_{\pd\Om}\del^{\la\mu}(a(\zeta)
\var(\zeta))K_{0,j}^\mu(z^\mu,\zeta^\mu)\, d\sigma
\\
&\quad +\int_{\pd\Om}a^\la(\zeta^\la)\var^\la(\zeta^\la)\del^{\la\mu}
K_{0,j}(z,\zeta)\, d\sigma
\df I_1(z)+I_2(z).\end{align*}
And $|I_1(z)|\leq C(|\var^\mu-\var^\la|_0+|\gaa^\mu-\gaa^\la|_1)$ by \re{kta3p--}.
By analogue of \re{dlmk},
we get $
|I_1|_\all\leq C(|\var^\mu-\var^\la|_0+|\gaa^\mu-\gaa^\la|_1)$.
Using \re{kta3p-} and \re{nkta3p}  we get $|I_2|_{\all,0}
\leq C|\var^\la|_0|\del^{\la\mu}\gaa|_{1+\all,0}$.
This shows  \re{ktapi2}.

 By
the chain rule \re{tazf}, that $K_{i,j}^{\la*}\var$ satisfies \re{ktapi4}-\re{ktapi6}
if and only if $\pd_{\ta^\la}^{k-i-j}K_{i,j}^{\la*}\var$ satisfies  estimates
\re{ktapi4}-\re{ktapi6} (with $i=k,j'=j$).
By \re{psar+}, $K_{k-j,j}^{\la*}\var$ equals $\pd_{\ta^\la}^{k-i-j}(K_{i,j}^{\la*}\var)$.
(The proof of \re{psar+} is still valid when   \re{kta3p--} substitutes for
 \re{kta1}.) Hence, we have reduced   \re{ktapi4}-\re{ktapi6}  to the case
 where $i=k-j$ and $j'=j$.
Using $|K_{k-j,j}^\la(\zeta^\la,z^\la)|\leq C|\zeta-z|^{\all-1}$, we obtain
$|\cL K_{k-j,j}^{\la*}\var|_0\leq C|\var|_0$.
For the H\"older norm,
we recover a loss of regularity in Kellogg's arguments~\ci{Kezeeib} by decomposing
\aln \cL K_{k-j,j}^{\la*}\var(z_2)&-\cL K_{k-j,j}^{\la*}\var(z_1)=
 \varphi^\la(z_1^\la)\int_{\pd\Om}
\{ K_{k-j,j}^{\la}(\zeta^{\la},z_2^{\la})- K_{k-j,j}^{\la}
(\zeta^{\la},z_1^{\la})\}\, d\sigma^\la(\zeta^\la)
\\
&+\int_{\pd\Om}(\varphi^\la(\zeta^\la)-\varphi^\la(z_1^\la))
\{ K_{k-j,j}^{\la}(\zeta^{\la},z_2^{\la})- K_{k-j,j}^{\la}
(\zeta^{\la},z_1^{\la})\}\, d\sigma^\la(\zeta^\la).
\end{align*}
The first integral equals $\cL K_{k-j,j}^{\la*}1(z_2)-\cL K_{k-j,j}^{\la*}1(z_1)$.
The second integral is bounded by  $|\var^\la|_\beta$ times
$$\int_{|\zeta-z_1|\leq 3|z_2-z_1|}|\zeta-z_1|^{\alpha-1}\, d\sigma(\zeta)+
\int_{|\zeta-z_1|>2|z_2-z_1|}\f{|z_2-z_1|^\alpha}{|\zeta-z_1|^{1-\beta}}
\, d\sigma(\zeta).$$
Here the sum does not exceed $C|z_2-z_1|^\all$ when $\beta>0$. If $\beta=0$,
it does not exceed
$C_{\all'}|z_2-z_1|^{\all'}$ for any $\all'<\all$. We have
\aln
\cL K_{k-j,j}^{\la*}1(z)
&=\int_{\pd\Om}a^\la(\zeta)\pd_{\ta_z^\la}^{k-j}\pd_\la^j\bigl\{\cL K^{\la}(\zeta^\la,z^\la)\bigr\}\, d\sigma(\zeta)\\
&=\pd_{\ta_z^\la}^{k-j}\int_{\pd\Om}a^\la(\zeta)\pd_\la^j\bigl\{\cL K^{\la}(\zeta^\la,z^\la)\bigr\}\, d\sigma(\zeta)\\
&=\sum_{l\leq j}C_{jl}\pd_{\ta_z^\la}^{k-j}\pd_\la^{j-l}\int_{\pd\Om}\pd_\la^la^\la(\zeta) \cL K^{\la}(\zeta^\la,z^\la) \, d\sigma(\zeta).
\end{align*}
By \rl{k1rg} and \rl{c0pa}, we get for $b_l^\la=(a^\la)^{-1}\pd_\la^la^\la$
$$|\cL K^{\la*}b_l|_{k-j+\all,j-l}\leq C,\quad |\cL K^{\la*}b_l-\cL K^{\la*}b_l|_{k-j+\all,j-l}\leq C\|\gaa^\mu
-\gaa^\la\|_{k+1+\all,j}.$$
We have verified \re{ktapi4}.
We have
\aln
\del^{\la\mu}\cL K_{k-j,j}^*\var(z)&=\int_{\pd\Om}\del^{\la\mu}(a(\zeta)\var(\zeta))
K_{k-j,j}^\mu(\zeta^\mu,z^\mu)\, d\sigma
\\
&\quad +\int_{\pd\Om}a^\la(\zeta)\var^\la(\zeta^\la))\del^{\la\mu}
K_{k-j,j}^*(\zeta,z)\, d\sigma.\end{align*}
By analogue of estimation for $\cL K_{k-j,j}^{\la*}\var$, we obtain \re{ktapi6}
by   \re{kta3np} and \re{kta3p}.
Finally, we obtain
\re{ktapi5}-\re{ktapi7}  by \re{ktapi1}-\re{ktapi6}, and
 $\pd_{\ta^\la}\cL K^\la\varphi=-\cL K^{\la*}(\pd_{\ta^\la} \var^\la)$.
\end{proof}

\setcounter{thm}{0}\setcounter{equation}{0}
\section{Null spaces of $I\pm\cL K$ and $I\pm\cL  K^*$}\label{sec6}
In this section, we   describe     results on integral equations for
the Dirichlet and Neumann problems.
Lacking a
  reference to the precise regularity in derivatives  on solutions to
   the two problems, we   derive  some details. The estimates will be used in
   arguments for the parameter case in section~\ref{sec7}.
As mentioned in section~\ref{sec1}, we     reduce the $\cL C^1$ regularity of solutions,
which is an important step in Kellogg's proof~\ci{Kezeeib},
to the integral equations for the Dirichlet problem to $\cL C^0$ regularity  of the integral
equations for the Neumann problem.

\pr{bt}  Let $\pd\Om\in\cL C^{k+1+\all}$  with $k\geq0$ and $0<\all<1$.
Let $0\leq\beta\leq\all$.
\bppp
\item Let  $p>1/\all$.  Then
$$
|\cL K\var|_{\all-1/p}\leq C_{1+\all}|\var|_{L^p}, \quad |\cL K^*\var|_{\all'}
\leq  C_{1+\all}C_{\all'}|\var|_{L^p}
$$
for any $\alpha'\leq\all-\f{1}{p}$ with $\all'<\alpha$.
\item Let $\cL L$ be one of $\cL K,-\cL K,\cL K^*$, and $-\cL K^*$. Then
\gan
|\var|_{L^p}\leq C_{1+\all}C_p(|\var|_{L^1}+|\var+\cL L\var|_{L^p}),\quad
p>1,\\
|\var|_{\beta}\leq C_{1+\all}(|\var|_{L^1}+|\var+\cL L\var|_{\beta}).
\end{gather*}
\item Let $\cL L=\cL K^*$ or $-\cL K^*$. Then  $
|\var|_{k+\beta}\leq C_{k+1+\all}(|\var|_{L^1}+|\var+\cL L\var|_{k+\beta}).
$
 \item  Let $\cL L=\cL K$ or $-\cL K$ and  $1\leq l\leq k+1$. Assume that $l\geq2$ or
 $\pd_\ta\var\in L^1$. Then
  $
|\var|_{l+\beta}\leq C_{k+1+\all}(|\var|_{L^1}+|\pd_\ta\var|_{L^1}+|\var+\cL L\var|_{l+\beta}).
$
 \eppp
\epr
\begin{proof}
(i). We adapt Kellogg's arguments  in the proof of \rl{ktapi}.
Let $1/p+1/q=1$. Decompose $\cL K\varphi(z_2)-\cL K\varphi(z_1)$ as
$$
\Bigl\{\int_{|\zeta-z_1|<2|z_2-z_1|}+\int_{|\zeta-z_1|>2|z_2-z_1|}\Bigr\}
\varphi(\zeta)(K(z_2,\zeta)-K(z_1,\zeta))\, d\sigma(\zeta).
$$
We estimate   first integral   by  $|K(z_j,\zeta)|\leq C|\zeta-z_j|^{\alpha-1}$ and get
\al\label{kz2z}
&\left\{\int_{|\zeta-z_1|<2|z_2-z_1|}|K(z_2,\zeta)|^q+|K(z_1,\zeta)|^q\,
d\sigma(\zeta)\right\}^{1/q}
\leq C |z_2-z_1|^{\alpha-1+\f{1}{q}}.
\end{align}
We estimate the second integral by \re{kta2}, i.e., $|K(z_2,\zeta)-K(z_1,\zeta)|
\leq C|z_2-z_1||\zeta-z_1|^{\all-2}$ for $|\zeta-z_1|>2|z_2-z_1|$.
Thus
\eq{logt}
\left\{\int_{|\zeta-z_1|>2|z_2-z_1|}|K(z_2,\zeta)-K(z_1,\zeta)|^q\,
d\sigma(\zeta)\right\}^{1/q}\leq C
|z_2-z_1|^{\alpha-\f{1}{p}}.
\eeq
Therefore,   $\varphi\in   L^p$ implies $\cL K\varphi\in \cL C^{\alpha-\f{1}{p}}$.

We now  estimate $\cL K^*f$, for which we use \rl{kta}. Thus, when $K$ is replaced by $K^*$ we still
have \re{logt} for $1\leq p<\infty$ and  \re{kz2z}. However,  for $p=\infty$,
$$
\int_{|\zeta-z_1|>2|z_2-z_1|}|K(\zeta,z_2)-K(\zeta,z_1)|\, d\sigma(\zeta)\leq C
|z_2-z_1|^\alpha(1+|\log|z_2-z_1||),
$$
which results in $\cL K^*\var\in\cL C^{\all'}$.

(ii). We follow some standard   estimates for compact integral operators
(\cite{Fonifi}, p.~120; \ci{Miseze},
p.~178).
 Define $T_L\var=\int\var(\zeta)L(z,\zeta)\, d\sigma(\zeta)$.
Let $\chi(z,\zeta)=1$ for $|z-\zeta|<\e/2$ and $\chi(z,\zeta)=0$ for $|z-\zeta|>\e$.
Let $p>1$ and $1/p+1/{q}=1$.  We have
$$
\int_{\pd\Om}\int_{\pd\Om}|\varphi(\zeta)|^p\f{d\sigma(\zeta)\,
d\sigma(z)}{|z-\zeta|^{1-\all}}\leq C_0\all^{-1}|\varphi|_{L^p}^p.
$$
Thus, we obtain $|T_{\chi L}\varphi|_{L^p}\leq C_0\all^{-1}
\e^{\all(1-\f{1}{p})}|\varphi|_{L^p}$ from
$$
\left|\int_{\pd\Om} \varphi(\zeta)\chi L(z,\zeta)\, d\sigma(\zeta)\right|\leq
\left(\int_{\pd\Om} |\varphi(\zeta)|^p\f{d\sigma(\zeta)}{|\zeta-z|^{1-\all}}\right)^{1/p}
\left(\int_{|z-\zeta|<\e} \f{d\sigma(\zeta)}{|\zeta-z|^{1-\all}}\right)^{1/{q}}.
$$
Therefore, $I+T_{\chi L}\colon L^p\to L^p$ has an inverse with norm   $<2$
 when $C_0\all^{-1}\e^{\all(1-\f{1}{p})}<1/2$.
Since $(1-\chi)L$ is continuous, it is easy to obtain
$$
|T_{(1-\chi)L}\varphi|_{L^\infty}\leq C_\e|\varphi|_{L^1}.
$$
Using
$
\varphi=(I+T_{\chi L})^{-1} (I+T_L)\varphi -(I+T_{\chi L})^{-1}  T_{(1-\chi)L}\varphi,
$
we estimate two inverses and obtain
$$
|\varphi|_{L^p}\leq C|\varphi+\cL L\varphi|_{L^p}+C_\e|\varphi|_{L^1}.
$$
When $\beta=0$, we  get $\var\in  L^\infty$ and hence $\var\in\cL C^0$  by (i). Assume that
$\beta>0$. Using $\varphi=(\varphi+\cL L\varphi)-\cL L\varphi$, we obtain
$|\varphi|_{\beta/2}\leq C|\varphi+\cL L\varphi|_{\beta/2}+C'|\varphi|_{L^1}$,
from which we get   $|\varphi|_\beta\leq C_1|\varphi+\cL L\varphi|_{\beta}+
C_2|\varphi|_{\beta/2}
\leq C_3|\varphi+\cL L\varphi|_{\beta}+C_4|\varphi|_{L^1}$.

(iii). It follows from \re{ktapi1} with $j=0$ and (ii).

(v).  When $k\geq l\geq2$, we know that $K$ is of class $\cL C^1$. Hence $\var\in\cL C^1$ if
$\var+\cL K\var$ is additionally of class $\cL C^1$.
   Since $\pd_\ta\var\in L^1$, by \rl{dsar} (ii) we get
$
\pd_\ta \varphi (z)\mp\cL K^* \pd_\ta \varphi=\pd_\ta (\varphi\pm\cL K\varphi).
$
  The rest follows from (ii)-(iii).
 \end{proof}
For applications to integral equations with parameter (\rl{kl}), we emphasize that
the constants $C_{k+1+\all}$ in \rp{bt}
depend only on $|(\hat\gaa')^{-1}|_0$ and $|\hat\gaa|_{k+1+\all}$ if $\pd\Om$ is
parameterized by $\hat\gaa$.
\pr{kern} Let $\pd\Om\in\cL C^{1+\all}$.  Let $e_0=1$ on $\pd\Om$.
For $i>0$, let $e_i=1$ on $\gaa_i$ and $e_i=0$ on $\pd\Om\setminus\gaa_i$. Let
 $0\leq\beta\leq\all$.
\bppp
\item
 Let $L$ be one of $K,-K,K^*,$ and $-K^*$. Then $\var+\cL L\var=\psi\in\cL C^\beta$
 admits an $L^1$
solution $\var$ if and only if
$\psi\perp\ker(I+\cL L^*)$.
All  $L^1$ solutions $\var$ are in $\cL C^\beta$.
\item
 $\{e_1,\ldots, e_m\}$ spans   $\ker(I+\cL K)$.
And $e_0$ spans $\ker(I-\cL K)$.
\item  $\ker(I+\cL K^*)\cap\operatorname{ker}(I+\cL K)^\perp=\{0\}$ and
$\ker(I-\cL K^*)\cap\operatorname{ker}(I-\cL K)^\perp=\{0\}$.
 \item $\ker(I+\cL K^*)$ is spanned by   $\{\phi_1,\ldots, \phi_m\}$,
  where $\phi_i$ satisfy $\int_{\gaa_i}\phi_j\, d\sigma=\del_{ij}$
for $i,j>0$ and $\int_{\pd\Om}\phi_i\, d\sigma=0$ for $i>0$. Moreover, $W{\phi_1},\ldots,
W\phi_m$
are locally constant on $\pd\Om$ and vanish on outer boundary of $\pd\Om$,  and
 $(W{\phi_i}|_{\gaa_j})_{1\leq i,j\leq m}$ is non-singular when $m>0$.
 \item
 $\ker(I-\cL K^*)$ is spanned by $\phi_0$ and $W\phi_0$ is constant on $\pd\Om$.
  Moreover, $\phi_{0}$ vanishes on $\pd\Om\setminus \gaa_0$,
  $\int_{\pd\Om}\phi_0\, d\sigma=1$,
  and $\phi_0$ depends only on $\gaa_0$.
\eppp
\epr
\begin{proof}
(i-iii).
   The first assertion follows from
  the compactness of $\cL L$   on $L^2$ (\cite{Miseze}, p.~162, p.~167). That
   $\var\in\cL C^\beta$
  follows from \rp{bt} (ii).
The proof of (ii) is in\ci{Fonifi} (p.~135).
For (iii), assume that $\psi\in\ker(I+\cL K^*)\cap\operatorname{ker}(I+\cL K)$.
We have $\psi+\cL K^*\psi=0$
and by (i) $\psi=\var+\cL K^*\var$. For $\psi\in\ker(I-\cL K^*)\cap\operatorname{ker}(I-\cL K)$,
we have $\psi-\cL K^*\psi=0$
and $\psi=\var-\cL K^*\var$.
In both cases, we have $\var,\psi\in\cL C^\all$. Then $W\var$ and $W\psi$
 are in $\cL C^{1+\all}$ by \rl{jumpf} (iii). One can
 show that $\psi=0$; see~\cite{Fonifi} (p.~137), where
  the use of Green's identities merely requires that $\pd\Om,$
   $W\var,W\psi$ be of class $\cL C^{1+\all}$.

 (iv). By compactness of $\cL L$, we have $ \dim\ker(I+\cL L^*)=\dim
 \ker(I+\cL L)=1$ (\cite{Fonifi}, p.~24). Note that if $\phi_1,\ldots, \phi_m$
 span $\ker(I+\cL K^*)$,
 the matrix $A=(\int_{\gaa_i}\phi_j\, d\sigma)_{1\leq i,j\leq m}$ must have
 rank $m$. Indeed if $\var=c_1\phi_1+\cdots+c_m\phi_m$
  is orthogonal to $e_1,\ldots, e_m$, then by (i) and (iii), $\var=0$.
 With $A$ being non-singular, we can normalize $\phi_i$ such that $A$ is the identity matrix.
 This verifies the first assertion.
 To show that $W\phi_i$ are locally constant on $\pd\Om$,
we integrate $\phi_i+\cL K^*\phi_i=0$ and get $\int_{\pd\Om}\phi_i\, d\sigma=0$ for
 $i>0$.
 This shows that
 $\phi_i\in\ker(I+K^*)$. Hence, $W{\phi_i}$ is locally constant on $\Om'$
and vanishes on the unbounded component of $\Om'$. By the continuity of $W\phi_i$,
it is constant on the inner boundary of $\pd\Om$ and vanishes on the outer boundary.
Assume for the sake of contradiction that
 $(W{\phi_i}|_{\gaa_j})_{1\leq i,j\leq m}$ is singular. Since $W\phi_i$ are
 constants on $\gaa_i$ and vanish on $\gaa_0$. Then $W\phi_1,\ldots, W\phi_m$
 are linearly dependent on $\pd\Om$. Therefore, for some $c_i$ which are not
 all zero, we have
  $W(c_1\phi_1+\cdots+c_m\phi_m)=0$ on $\pd\Om$. This implies that
  $c_1\phi_1+\cdots+c_m\phi_m$ is in $\ker(I+K^*)\cap\ker(I-K^*)$. Since $\phi_1,\ldots,\phi_m$
  form a basis, we get   $c_i=0$ for all $i$, a contradiction.

 (v). By (iii), we know that if $\phi_0$ spans  $\ker(I-\cL K^*)$
 then $\int_{\pd\Om}\phi_0\, d\sigma\neq0$. Let $\phi_0$ be the unique element
 in $\ker(I-\cL K^*)$
 such that $\int_{\pd\Om}\phi_0\, d\sigma=1$.  We want to
 show that $\phi_0=0$ on $\gaa_j$ for $j>0$ and that
  $\phi_0$ depends only on $\gaa_0$.

Let $\Om_0$ be
the bounded domain bounded by
outer boundary  $\gaa_0$ of $\Om$. Let $\phi$, with $\int_{\gaa_0}\phi\, d\sigma=1$,
span
$\ker(I-\cL K_{0}^*)\subset L^2(\pd\Om_0)$. Here
$\cL K_0^*(\zeta,z)=\f{1}{\pi}\pd_{\ta_\zeta}\arg(\zeta-z)$  for $\zeta,z\in\pd\Om_0$.
 Let $\hat W{\phi }$ be the simple-layer distribution with
density $\phi $ on $\gaa_0$.
Since $\hat W{\phi }$ is constant on $\Om_0$,
then $\pd_{\nu}\hat W{\phi}=0$ for the normal vector $\nu$ of
any $\cL C^1$ curve $\gaa$ in $\Om_0$.
This shows that $\ker(I-\cL K^*)$ is spanned by $\tilde\phi$, if
$\tilde\phi$  equals  $\phi$  on $\gaa_0$ and is zero on $\pd\Om\setminus\gaa_0$.
The  condition $\int_{\pd\Om}\phi_0\, d\sigma=1$
implies that $\phi_0=\tilde\phi$.
\end{proof}
For convenience, we will use
$ \{e_1,\ldots, e_m\},\{e_0\}, \{\phi_1,\ldots, \phi_m\}$, and $\{\phi_0\}$
for bases of $\ker(I+\cL K)$, $\ker(I-\cL K)$, $\ker(I+\cL K^*)$, and
$\ker(I-\cL K^*)$, respectively.
\le{klem} Let $\pd\Om\in\cL C^{1+\all}$ and    $0\leq\beta\leq\all$. Let $L=K$ or $-K$.
 If $\var+\cL {L}\var =g $ is in   $\cL C^{1+\beta}$ and $\var\perp\ker(I+\cL L)$,
 then $\var\in\cL C^{1+\beta}$ and it is determined by
 \gan
 \var=\hat\var +c_0e_0+c_1e_1+\cdots+c_me_m,\\
\var_1-\cL L^*\var_1=\pd_\ta g,\quad\var_1\in\ker(I-\cL L^*)^\perp\cap
  \ker(I+\cL L)^\perp,\\
\pd_\ta\hat\var=\var_1-d_0\phi_0-d_1\phi_1-\cdots-d_m\phi_m, \quad
 \int_{\gaa_i}\hat\var  \, d\sigma=0, \quad   i\geq0.
 \end{gather*}
Moreover,   $c_i$ and $d_i$ are determined as follows:\bppp
 \item If $\cL L=\cL K$, then
 \gan
 c_0=\f{1}{2l^2}\int_{\pd\Om}(g- \cL L\hat\var) \, d\sigma,
\quad  c_1=\cdots=c_m=-c_0, \\
d_0=\int_{\pd\Om}\var_1 e_0\, d\sigma,\quad d_1=\cdots=d_m=0.
\nonumber \end{gather*}
 \item If $\cL L=-\cL K$, then
 \gan
 c_i=\f{1}{2l_i^2}\int_{\gaa_i}(g- \cL L\hat\var) \, d\sigma, \quad   i\geq1,\qquad
 c_0=-\f{1}{2l^2}\int_{\pd\Om\setminus\gaa_0}(g- \cL L\hat\var) \, d\sigma,\\
d_i=\int_{\pd\Om}\var_1 e_i\, d\sigma,\quad i\geq1, \qquad d_0=0.
 \end{gather*}
 \eppp
\ele
\begin{proof} (i).
Assume that $\var+\cL K\var=g\in\cL C^{1+\beta}$. Recall that $e_1,\ldots, e_m$
 span $\ker(I+\cL K)$.
Since $\pd_\ta g\perp\ker(I-\cL K)$, there exists $ \var_1\in\cL C^\beta\cap
\ker(I-\cL K^*)^\perp$
such that $ \var_1-\cL {K^*} \var_1=\pd_\ta g$.
Let $j>0$. Since $\int_{\gaa_j}\cL K^* \var_1\, d\sigma=-\int_{\gaa_j} \var_1\, d\sigma$,
then $\int_{\gaa_j} \var_1\, d\sigma=0$.
Recall that $\phi_0=0$ on $\gaa_j$ and $\int_{\gaa_0}\phi_0\, d\sigma=1$.
Let $d_0=\int_{\pd\Om}
\var_1\, d\sigma$.
 Then $\tilde\var_1=\var_1-d_0\phi_0$ is orthogonal to $e_0,\ldots, e_m$ and hence
there is a unique
$\hat\var\in\cL C^{1+\beta}(\pd\Om)$ such that  $\tilde\var_1=\pd_\ta\hat\var
$ and $\int_{\gaa}\hat\var e_i\, d\sigma=0$ for all $i\geq0$.
Thus, we obtain
\aln
\pd_\ta(\var+K\var)&=
\pd_\ta g=\tilde\var_1-\cL K^*\tilde\var_1=\pd_\tau\hat\var+d_0\phi_0-
\cL K^*(\pd_\tau\hat\var+d_0\phi_0)\\
&=\pd_\ta(\hat\var+K\hat\var).
\end{align*}
Hence, $\var-\hat\var+\cL K(\var-\hat\var)=2c_0+\sum_{i>0} \tilde c_ie_i$. We rewrite it as
$$
(\var-\hat\var-c_0)+\cL K(\var-\hat\var-c_0)=\sum_{i>0} \tilde c_ie_i.
$$
Being in the range of $I+\cL K$,  the right-hand side must be orthogonal to
 $\ker(I+\cL K^*)$. Hence,  $\tilde c_i=0$
and consequently $\var-\hat\var-c_0=\sum_{j>0}  c_je_j$. This shows that
$\var\in\cL C^{1+\beta}$.
Since $\var$
and $\hat \var$ are orthogonal to $e_i$ for $i>0$,
then $c_i+c_0=0$. We  substitute
$\hat \var+c_0e_0+\cdots+c_me_m$ for $\var$ in $\var+\cL K\var=g$ to get
$g=\hat\var+\cL K\hat\var+2c_0e_0$.
Therefore, $2c_0l^2= \int_{\pd\Om}(g- \cL K\hat\var)\, d\sigma$.

(ii).  Assume that $\var-\cL K\var=g\in\cL C^{1+\beta}$. We find
$\var_1\in\cL C^\beta\cap\ker(I+\cL K^*)^\perp$ such that
$\var_1+\cL {K^*}\var_1=\pd_\ta g$. By $\int_{\pd\Om}\cL K^*\var_1\,
 d\sigma=\int_{\pd\Om}\var_1\, d\sigma$,
we get $\int_{\pd\Om}\var_1\, d\sigma=0$.
Since $\phi_j\in\ker(I+\cL {K^*})$ satisfy $\int_{\gaa_i}\phi_j\, d\sigma=
\del_{ij}$ for $i,j>0$, then
for $d_j=\int_{\gaa_j}\var_1\, d\sigma$,
 $\tilde\var_1=\var_1-d_1\phi_1-\cdots-d_m\phi_m$ is orthogonal to
$e_1,\ldots, e_m$. We still have $\tilde\var_1+\cL {K^*}\tilde\var_1=\pd_\ta g$;
 in particular,
$\int_{\gaa_0}\tilde\var_1\, d\sigma=\int_{\pd\Om}\tilde\var_1\, d\sigma=0$.
We   write  $\tilde\var_1=\pd_\ta\hat\var$
with
 $\int_{\gaa_j}\hat\var\, d\sigma=0$ for $j\geq 0$.  As in (ii), we get
 $
\pd_\ta(\var-\hat\var-\cL K(\var-\hat\var))=0$ and hence
$$
 (\var-\hat\var)-\cL K(\var-\hat\var)=\tilde c_0+2\sum_{i=1}^m c_ie_i.
$$
The right-hand side must be orthogonal to $\ker(I-\cL K^*)$, the span
 of $\phi_0$. As $\phi_0$ vanishes on $\gaa_1\cup\cdots\cup\gaa_m$ by \rp{kern} (v), we obtain
  $\tilde c_0=\tilde c_0\int_{\pd\Om}\phi_0 \, d\sigma=0$.
Then $\var-\hat\var-\sum_{j>0} c_je_j\in\ker(I-\cL {K})$, so it  is a constant $c_0$. Therefore,
$\var\in\cL C^{1+\beta}$. Also, $g=\var-\cL K\var=\hat\var-\cL K\hat\var+2(c_1e_1+\cdots+c_me_m)$.
We get $2c_il_i^2=\int_{\gaa_i}(g-\cL K\hat\var)\, d\sigma$ for $i>0$.
Using $0=\jq{\var,e_0}=\sum_{i>0}c_i|l_i|^2+c_0l^2$, we get the formula for $c_0$.
\end{proof}

The above lemma allows us to study integral equations for
the planar Dirichlet problem  via integral equations for
the Neumann problem.
We now  strengthens \rp{bt} (iv) as follows.
\co{ffdi} Let $\pd\Om\in\cL C^{l+\all}$ with $l\geq1$ and $0<\all<1$. Let
 $0\leq\beta\leq\all$ and let $L$ be $K$
or $-K$. If $\var+\cL L\var\in\cL C^{l+\beta}$, then $\var\in\cL C^{l+\beta}$.
\eco

\setcounter{thm}{0}\setcounter{equation}{0}
\section{Regularity of solutions for integral equations with parameter}
\label{sec7}

To motive our methods, we assume for simplicity  that $\pd\Om$ is $\cL C^2$,
and parameterize $\pd\Om$ by $\gaa(t)$ in arc-length.
The kernel $K(s,t)=\f{1}{\pi}\pd_t\arg(\gaa(s)-\gaa(t))$ is then continuous and
the resolvent $L(s,t,z)$ satisfies
\gan 
K(s,t)=L(s,t,z)+z\int_0^lL (s,r,z)K(r,t)\,d r.
\end{gather*}
 It is a basic result of Fredholm
that there exists
 $\del(z)$ with $\del(0)=1$ such that $\delta(z)$
 and $\delta(z)L(s,t,z)$ are entire functions  in $z$ (see, e.g., \ci{Kezeeib}).
It is known that  $L(s,t,z)$ is analytic
at $z=1$ when $\Om$ is simply connected (see~\cite{Kezeeib}, or~\ci{Fonifi}, p.~133);
by a theorem of Plemelj~\ci{Plzefo}, it has a simple pole at $z=1$ otherwise.
However, we do not know
 if the zeros of $\del$ accumulate at $1$ as $\del$ varies with $\Om$. One can verify that
 $\del(1)\neq0$ when $\Om$ is simply connected (\cite{Ketwni}, p.~294) and in this case
 the zero of $\del$ does not accumulate at $1$ as domains vary.
  Without resolving this issue,
we will  estimate solutions by taking   limit   and differentiating in $\la$
on the integral equations directly.

This section consists of three results.   \rl{kl} shows the uniform boundedness of solutions
of integral equations in $L^p$ and H\"older norms;   \rl{kl2} provides
 formulae  to differentiate
the integral equations;   \rp{kl3} contains the estimates   for the solutions of the integral
equations.

Recall that for a family of functions $f^\la$ on $\pd\Om\in\cL C^{k+\all}\cap\cL C^1$,
  we define for $k\geq j$
\gan
\|f\|_{k+ \all,j}=\max_{i\leq j,\la}| \pd_{\la}^i    f^\la  |_{k-i+\all},\quad
 \|f^{\mu}-f^\la\|_{k+ \all,j}=\max_{i\leq j }| \pd_{\mu}^i    f^{\mu}
 - \pd_{\la}^i  f^\la|_{k-i+\all}.
 \end{gather*}
For a family of embeddings $z\to\gaa^\la(z)$ from $\pd\Om$
onto $\pd\Om^\la$, we use notation $z^\la=\gaa^\la(z)$ and  $g(z,\la)=g^\la(\gaa^\la(z))$.
Let $\{\cL L^\la\}$ be one of $\{\cL K^{\la}\}, \{-\cL K^{\la}\},
 \{\cL K^{\la*}\}, \{-\cL K^{\la*}\}$, and let
 $\{\ell_1^\la,\ldots, \ell_{n}^\la\}$ be the canonical
  basis of $\ker(I+\cL L^\la)$, described after the proof of \rp{kern}. Define
  $$
  (<\!\var,\ell_i\!>)^\la=\,\jq{\var^\la,\ell_i^\la}\,=
  \int_{\pd\Om^\la}\var^\la\ell_i^\la\, d\sigma^\la.
  $$
\le{kl} Let $\gaa^\la$ embed  $\pd\Om$
onto $\pd\Om^\lambda$  with  $\gaa\in{\cL B}^{1+\all,0}(\pd\Om)$. Let $0<\all<1$ and
$0\leq\beta\leq\all$.
Let $\var^\la\in L^1(\pd\Om^\la)$ and define $\psi_i^\la$ according to the following two cases.
 \bpp
\item
Let  $\{\cL L^\la\}$ be $\{\cL K^\la\}$ or $\{-\cL K^\la\}$. And
$$
\varphi^\la+\cL L^{\la*}\varphi^{\la}=\psi_0^\la,
\quad \jq{\varphi^\la,\ell_i^\la}\,=\psi_i(\la),
\quad 1\leq i\leq n.
$$
\item
 Let $\{\cL L^\la\}$ be one of $\{\cL K^\la\},\{-\cL K^\la\}, \{\cL K^{\la*}\}$,
  and $\{-\cL K^{\la*}\}$. And
$$
\var^\la+\cL L^{\la}\var^{\la}=\psi_0^\la,
\quad \jq{\var^\la,\ell_i^\la}\,=\psi_i(\la),
\quad 1\leq i\leq n.
$$
\epp
Then the followings hold.
\bppp
\item
Let   $1/\all <p\leq\infty$. Suppose that
$\la \mapsto\psi_0^\la\circ\gaa^\la\in L^p(\pd\Om)$
and $\la\mapsto \psi_i(\la)\in\rr$ are bounded $($resp.\,continuous$)$ maps. Then
$\la\mapsto \var^\la\circ\gaa^\la\in {L^p(\pd\Om)}$ is bounded $($resp.\,continuous$)$.
\item If $\psi_0^\mu,\psi_0^\la$ are in $L^p(\pd\Om^\la)$ with $1<p\leq\infty$ then
\al
\label{ptata+2Lp}
&|\varphi(\cdot,\mu)-\varphi(\cdot,\lambda)|_{L^p}
 \leq C_{1+\all,0}C_p(|\varphi(\cdot,\mu)-\varphi(\cdot,\lambda)|_{L^1}+
 |\psi_0 (\cdot,\mu)-\psi_0 (\cdot,\lambda)|_{L^p}\\
 &\qquad+
(|(\psi_0(\cdot,\mu),\psi_0(\cdot,\lambda))|_{L^p}
 +
|(\varphi(\cdot,\mu),\varphi(\cdot,\lambda))|_{L^1})
|\gaa^{\mu}-\gaa^{\lambda}|_{1+\all}).
\nonumber
\end{align}
\item
If $\psi_0^\mu,\psi_0^\la$ are in $\cL C^\beta(\pd\Om^\la)$, then
\al
\label{ptata+2}
|\varphi(\cdot,\mu)-\varphi(\cdot,\lambda)|_{\beta}
& \leq C_{1+\all,0}(|\varphi(\cdot,\mu)-\varphi(\cdot,\lambda)|_{L^1}+
 |\psi_0 (\cdot,\mu)-\psi_0 (\cdot,\lambda)|_{\beta}\\
 &\quad+
(|\psi_0^\la|_{\beta}
 +
|\varphi^\la|_{L^1})
|\gaa^{\mu}-\gaa^{\lambda}|_{1+\all}).
\nonumber\end{align}
\eppp
\ele
\begin{proof} We first verify the assertions for case a).  The verification for b) will be
   simpler, after we establish $\phi_i\in\cL B^{\all,0}$ via (i) of case a).
The proof of (i) is given in steps 1 and 2. The proofs of (ii) and (iii) are in step 3.

   {\bf Step 1. Boundedness in $L^p$ norms}. Fix $1/\all <p\leq\infty$.
We are given
\ga\label{tphi-}
    \varphi(\zeta,{\lambda})+\int_{\eta\in\pd\Om}     \varphi(\eta,{\lambda})
  L^{\lambda}(\eta^{\lambda},\zeta^{\lambda})\,
   d\sigma^{\lambda}(\eta^{\lambda})=  \psi_0(\zeta,{\lambda}),\\
\nonumber 
 \int_{\eta\in\pd\Om}
  \varphi(\eta,{\lambda})\ell_i^{\la}\, d\sigma^{\la}(\eta^{\la})= \psi_i({\lambda}),\quad
  1\leq i\leq n.
\end{gather}
Assume for the sake of contradiction that $A_j=| \varphi^{\lambda_j}|_{L^p}\to\infty$
for some $\lambda_j\to0$.
Normalize in $L^p$ norm by letting
$\tilde \varphi^{\la_j}=A_j^{-1}\varphi^{\la_j}$ and $\tilde \psi_i^{\lambda_j}
=A_j^{-1}\psi_i^{\lambda_j}$.
   We get
\ga\label{tphi}
  \tilde \varphi(\zeta,{\lambda_j})+\int_{\pd\Om}    \tilde \varphi(\eta,{\lambda_j})
  L^{\lambda_j}(\eta^{\lambda_j},\zeta^{\lambda_j})\,
   d\sigma^{\lambda_j}(\eta^{\lambda_j})=\tilde \psi_0(\zeta,{\lambda_j}),\\
\label{tphi+}|\tilde \varphi(\cdot,{\lambda_j})|_{L^p}=1, \quad \int_{\pd\Om}
\tilde \varphi(\eta,{\lambda_j})\ell_i^{\la_j}\, d\sigma^{\la_j}(\eta^{\la_j})=
\tilde \psi_i({\lambda_j}).
\end{gather}
Since the $L^p$ norms of  $\tilde \varphi^{\lambda_j}$  are   bounded, by \rp{bt} (i)
the  $\cL C^{\all/2}$-norms of $\cL L^{\la_j*}\varphi^{\la_j} $ on $\pd\Om^{\gaa_j}$
are   bounded too. Thus,
$(\cL L^{\la_j*}\varphi^{\la_j})\circ\gaa^{\la_j}$ have bounded $\cL C^{\all/2}$-norms
 on $\pd\Om$.
Passing to a subsequence if necessary, $(\cL L^{\la_j*}\varphi^{\la_j})\circ\gaa^{\la_j}$
converges uniformly on $\pd\Om$. Since $\tilde\psi_0(\cdot,\la_j)$ converges to $0$
in $L^p$ norm,  \re{tphi} implies
that
$\tilde\varphi(\cdot,\la_j)$ converges  to  some $\varphi_*=\varphi^*\circ\gaa^0$
in $L^p$ norm.
Recall that $d\sigma^\la(z^\la)=a^\la(z)\, d\sigma(z)$ with $a^\la(z)=|\pd_{\ta}\gaa^\la(z)|$.
 Since $a^{\la_j}$
converges to $a^0$ in sup norm, then $\tilde\var(\cdot,\la_j)a^{\la_j}(\cdot)$ approaches
to $\var_*a^0(\cdot)$ in $L^p$ norm.  Decompose
\aln
&\Bigl|\int_{\pd\Om}   \tilde\varphi  (\eta,\la_j) L^{\la_j}(\eta^{\la_j},\zeta^{\la_j})\,
   d\sigma^{\la_j}(\eta^{\la_j})-\int_{\pd\Om}   \varphi_*  (\eta) L^{0}(\eta^{0},\zeta^{0})\,
   d\sigma^{0}(\eta^{0})\Bigr|\\
   &\quad\leq
 \Bigl|\int_{\pd\Om}   (\tilde\varphi  (\eta,\la_j)a^{\la_j}(\eta)
 -\varphi_*  (\eta)a^0(\eta) ) L^{\la_j}(\eta^{\la_j},\zeta^{\la_j})\,
   d\sigma(\eta)\Bigr|\\
   &\qquad +\Bigl|\int_{\pd\Om}   \varphi_*  (\eta)a^0(\eta)( L^{\la_j}(
   \eta^{\la_j},\zeta^{\la_j})
   - L^{0}(\eta^{0},\zeta^{0}))\,
   d\sigma(\eta)\Bigr| =I_j'(z)+I_j''(z).
\end{align*}
From $p>1/\all$, $|K^\la(\zeta^\la,z^\la)|\leq C|\zeta-z|^{\all-1}$, and
H\"older inequalities, we see
that $I_j\to0$ in $L^p$ as $\la_j\to0$. From H\"older inequality and the
dominated convergence theorem, we
see that $I_j''\to0$ in $L^p$ also for $\la_j\to0$.  Thus,
  letting $j$ tend to $\infty$ in \re{tphi}-\re{tphi+}, we get
 \gan
  \varphi^*(\zeta^0)+\int_{\pd\Om}   \varphi^*  (\eta^0) L^{0}(\eta^{0},\zeta^{0})\,
   d\sigma^{0}(\eta^{0})=0,\\
|\varphi^*|_{L^p}=1, \quad \int_{\pd\Om}\varphi^*(\eta^0)\ell_i^{0}\,
d\sigma^{0}(\eta^0)=0,\quad
i=1,\ldots, n.
\end{gather*}
By \rp{kern} (iii), the first and   last $n$
 identities imply that $\varphi^*=0$. The latter contradicts to the second
identity. Therefore $\{|\varphi^\la|_{L^p}\}$ is bounded.
By \rp{bt} (i) and (ii), we obtain
\ga\label{houb}
|\cL L^{\la*}\varphi^\la|_{\all/2}\leq C_{1+\all}|\var^\la|_0,\\
|\varphi^\la|_{\beta}\leq C_{1+\all}(|\varphi^\la|_{L^1}+|\psi_0^\la|_\beta).
\label{houb+}\end{gather}

{\bf Step 2. Continuity in $L^p$ norms}. Fix $1/\all <p\leq\infty$.
Assume for the sake of contradiction that
   $|\varphi(\cdot,\la_j)-\varphi(\cdot,0)|_{L^p}\geq\del>0$ for a sequence
$\la_j$ tending to zero. By \re{houb},
passing to
a subsequence if necessary, we may assume that the sequence of continuous functions
$(\cL L^{\la_j}\varphi^{\la_j*})\circ\gaa^{\la_j}$ converges uniformly as $\la_j\to0$.
Hence by \re{tphi-},
$\varphi(\cdot,\la_j)$
 converges to   $\varphi_*=\varphi^*\circ\gaa^0$ in $L^p$. We have
$|\varphi(\cdot,{0})-\varphi^{*}(\cdot)|_{L^p}\geq\del$. By the same arguments
in step 1 we know that $\varphi^*,\varphi^0$ satisfy the same equations
\gan
\varphi^0(\zeta^0)+ \int_{\pd\Om}   \varphi^0  (\eta^0) L^{0}(\eta^{0},\zeta^{0})\,
   d\sigma^{0}(\eta^{0})=\psi_0^0(\zeta^0),\\
   \int_{\pd\Om}\varphi^0(\eta^0)\ell_i^0\, d\sigma^0(\eta^0)=\psi_i^0,
   \quad 0\leq i\leq n.
\end{gather*}
By \rp{kern} (iii),  $\varphi^0=\varphi^{*}$, a contradiction.
This proves  the continuity of $\varphi(z,\la)$ in $L^p$ norm.

We  proceed to repeat steps $1$ and $2$ for case b).  From case a), we know
that $\phi_0,\ldots,\phi_m$ are of class $\cL B^{\alpha,0}$.
Thus the basis $\{\ell_1,\ldots,\ell_n\}$ of $\ker(I+\cL L)$ is of class
$\cL B^{\all,0}$ in all cases.

We are given
\ga\label{vaeq2}
    \varphi^\la+
  \cL L^{\lambda}\var^\la=  \psi_0^\la,\quad
   \int_{\eta\in\pd\Om}
  \varphi^\la\ell_i^{\la}\, d\sigma^{\la}(\eta^{\la})= \psi_i^\la,\quad
  1\leq i\leq n.
\end{gather}
We first repeat step 1, which is simpler now.
Assume for contradiction that there exists a sequence $\la_j$, approaching to $0$,
 such that $|\var^{\la_j}_j|_{L^p}=B_j$ tends to $\infty$. Then
 $\tilde\var_j=B_j^{-1}\var_j^{\la_j}$ has bounded $L^p$ norms, and $\cL L^{\la_j}\tilde\var_j$
 has bounded $\cL C^{\all/2}$ norms.   Passing to a subsequence if necessary,
  we may assume that $\cL L^{\la_j}\tilde\var_j$ converges uniformly on $\pd\Om$. Hence
    $\tilde\var_j=B_j^{-1}\psi_j-\cL L^{\la_j}
  \tilde\var_j$ converges to $\var_*=\var^*\circ\gaa^0$ in $L^p(\pd\Om)$. Reasoning
  as in step 1 shows that $\var^*$ satisfies
$$
\var^*+    \cL L^{0}\var^*  =0,\quad \var^*\perp\ker(I+\cL L^0),\quad |\var^*|_{L^p}=1.
   $$
The first two expressions imply that $\var^*=0$, a contradiction. This shows that $\var^\la$
  have bounded $L^p$ norms.
   Thus the $\cL C^{\all/2}$ norms of $\cL L^\la\var^\la$ on $\pd\Om^\la$ are bounded, and
   every sequence $\var^{\la_j}\circ\gaa^{\la_j}(z)$ with $\la_j\to0$ has a
   subsequence converging uniformly
to some $\tilde\var(z,0)=\tilde\var^0\circ\gaa^0(z)$ on $\pd\Om$. It is clear that
  $\tilde\var^0, \var^0$
 satisfy the same equations \re{vaeq2} with $\la=0$. Therefore,
$\tilde\var^0=\var^0$ and consequently $\var^\la\circ\gaa^\la$ are continuous in $L^p$ norm.

{\bf Step 3. Estimates  in $L^p$ and  H\"older norms}. This step
works for a), b).
We first consider case a) and  derive \re{ptata+2} for $\beta>0$. We have $L=K$ or $-K$.
It suffices to verify it for $\beta=\all$.
 For $ z\in\pd\Om$, write $$d\sigma^\la(z^\la)= a^\la(z)\, d\sigma(z), \qquad
D(z)= \f{a^\mu(z)}{a^\la(z)}\varphi (z,{\mu})-  \varphi (z,{\la}).
$$
We set $\la=
\mu$ in \re{tphi-} and then multiply it by $a^\mu(z)/{a^\la(z)}$. We   subtract
the new equation by the original \re{tphi-} and  get
\ga \label{dint}
D(z)+\int_{\pd\Om}D(\zeta)  L^\la(\zeta^\lambda,z^\lambda)\, d\sigma^\la(\zeta^\la)=E_0(z)-
E_1(z)-E_2(z).
 \end{gather}
 with
 \ga
\label{e0za} E_0(z)= \f{a^\mu(z)}{a^\la(z)}\psi_0 (z,{\mu})-  \psi_0 (z,{\la}),\\
 \label{e1zp}
  E_1(z)=\int_{\zeta\in\pd\Om}   \varphi (\zeta,{\mu})  (L^{\mu}(\zeta^{\mu},z^{\mu})-
    L^{\la}(\zeta^{\la},z^{\la}))\,
    d\sigma^{\mu}(\zeta^{\mu}),\\
   \label{e2za}
    E_2(z)=\Bigl\{\f{a^\mu(z)}{a^\la(z)}-1\Bigr\}\int_{\zeta\in\pd\Om}
     \varphi (\zeta,{\mu})  L^{\mu}(\zeta^{\mu},z^{\mu})\,
    d\sigma^{\mu}(\zeta^{\mu}).
\end{gather}
Note that $| \f{a^\mu(\cdot)}{a^\la(\cdot)}-1|_{\all}\leq C_1|\gaa(\cdot,\mu)
-\gaa(\cdot,\la)|_{1+\all}$.
Immediately, we have
\gan
|E_0|_{\all}\leq C_{1}(|\psi_0^\mu|_{\alpha}|\gaa(\cdot,\mu)-\gaa(\cdot,\la)|_{1+\all}
+|\psi_0(\cdot,\mu)-\psi_0(\cdot,\la)|_{\alpha}).
\end{gather*}
 By \re{ktapi4} with $i=j'=0$, we obtain
\gan
|E_2|_{\alpha}\leq C|\gaa(\cdot,\mu)-\gaa(\cdot,\la)|_{1+\all}|\varphi^\mu|_{\alpha}.
\end{gather*}
Define trivial extension
$\tilde\varphi_\mu^\la(z^\la)=\varphi^\mu(z^\mu)$, so it
 is actually independent of $\la$. In particular, since $\varphi(\cdot,\mu)
\in\cL B^{\alpha}$ then
 $\tilde\varphi_\mu$ is of class $\cL C^{\alpha,0}(\pd\Om_\gaa)$.
Also, define $\tilde a_\mu^\la(z^\la)=a^\mu(z)$, so $\tilde a_\mu\in
\cL C^{k+\all,j}(\pd\Om_\gaa)$. By \rl{k1rg},
for $L=\e K^*$
\al\label{e1z2}
\e E_1(z)&=\var^\mu\Bigl(1-\f{a^\mu(z)}{a^\la(z)}\Bigr)-2\RE\{{\ta^{\mu}}
\cL C^{{\mu}}_+(\ov{\ta^{\mu}} \tilde\varphi_\mu^{\mu})\}
 +2\RE\Bigl\{{\ta^{\la}}
\cL C^\la_+\Bigl(\ov{\ta^{\la}}\f{\tilde a^\la_\mu}{a^\la}\tilde \varphi_\mu^{\la}\Bigr)\Bigr\}.
\end{align}
 By the Cauchy transform with parameter (\rl{c0pa}), we obtain
$$
|E_1|_{\alpha}\leq C|\var^\mu|_{\alpha}|\gaa(\cdot,\mu)-\gaa(\cdot,\la)|_{1+\all}.
$$
Applying \rp{bt} (ii) to \re{dint}, we obtain
\aln
| D|_{\alpha}&\leq   C(|D|_{L^1}+(|\varphi^\mu|_{\alpha}+
|\psi_0^\mu|_{\alpha})|\gaa(\cdot,\mu)-\gaa(\cdot,\la)|_{1+\all}
+|\psi_0(\cdot,\mu)-\psi_0(\cdot,\la)|_{\alpha})\\
&\leq C(|\varphi(\cdot,\mu)-\varphi(\cdot,\la)|_{L^1}+(|\varphi^\mu|_{L^1}+
|\psi_0^\mu|_{\alpha})|\gaa(\cdot,\mu)-\gaa(\cdot,\la)|_{1+\all}
\\ &\quad\quad+|\psi_0(\cdot,\mu)-\psi_0(\cdot,\la)|_{\alpha}).
\end{align*}
Here the last inequality is obtained by the definition of $D$ and  \re{houb+}.
 The proof of   \re{ptata+2}
is complete when $\beta>0$.

To verify \re{ptata+2Lp} for case a), we start with  \re{e0za} and get
 $|E_0|_{L^p}\leq C(|\psi_0(\cdot,\la)|_{L^p}|\gaa^\mu-\gaa^\la|_{1}+|\psi_0(\cdot,\mu)-\psi_0
(\cdot,\la)|_{L^p})$. By \re{kta3p-},
$$
|L^{\mu}(\zeta^{\mu},z^{\mu})-
    L^{\la}(\zeta^{\la},z^{\la})|\leq C|\gaa^\mu-\gaa^\la|_{1+\all}|\zeta-z|^{\all-1}.
$$
By H\"older inequality and Fubini's theorem  (or Young's inequality), we have
$
|E_1|_{L^p}\leq C|\var^\mu|_{L^p}|\gaa^\mu-\gaa^\la|_{1+\all}.
$
Also, $
|E_2|_{L^p}\leq C(|\var^\mu|_{L^p})|\gaa^\mu-\gaa^\la|_{1+\all}.
$
By \rp{bt} (ii), we have $|\var^\mu|_{L^p}\leq C(|\var^\mu|_{L^1}+|\psi_0^\mu|_{L^p})$. Thus,
$$
|(E_0,E_1,E_2)|_{L^p}
\leq C(|\psi_0(\cdot,\mu)-\psi_0
(\cdot,\la)|_{L^p}+(|\var^\la|_{L^1}+|\psi_0^\la|_{L^p})|\gaa^\mu-\gaa^\la|_{1+\all}).
$$
By \re{dint} and \rp{bt} (ii) again, we get \re{ptata+2Lp}. Note that
\re{ptata+2Lp} for the $L^\infty$ case gives us \re{ptata+2}
for $\beta=0$.

For b), the above arguments are still valid for \re{ptata+2Lp}-\re{ptata+2}
 after minor changes. The formula \re{e1z2} for $E_1$
  needs to be changed when $L=K$ or $-K$ (see~\rl{k1rg}). The use of Cauchy
  transform with parameter
  is, however,  valid, and the same estimate for $E_1$ holds. The proof for
  (ii) and (iii) is complete.
\end{proof}

\begin{rem} The norms of $\psi_1,\ldots, \psi_n$ do not appear in
 \re{ptata+2Lp}-\re{ptata+2}. However, when we use \re{ptata+2Lp}-\re{ptata+2},
we  need   $\var^\la$ to have bounded $L^1$ norms at least.
The boundedness is established via \rl{kl} (i), so    restrictions
on $\psi_i$ for $i>0$ enter.
\end{rem}

We want to use   \re{dint}-\re{e2za}   to compute
the derivatives in parameter.  Define
 \ga\label{tplv14}
 \tilde\pd_\la\var^\la(z^\la)\df  \pd_\la\{\var^\la(z^\la)\}+\var^\la(z^\la)
 \pd_\la\log |\pd_{\ta_z}\gaa^\la|,\\
\nonumber 
 {\cL L}_1^{\la*}\var(z )=
\int_{\pd\Om^\la}    \var^\la (\zeta^\la)\pd_\la
\left\{ L^{\la}(\zeta^\la,z^\la)\right\}\,d\sigma^\la(\zeta^\la),\\
\nonumber 
 {\cL L}_2^{\la*}\var(z)=(\pd_\la \log |\pd_{\ta_z}\gaa^\la| ) \int_{\pd\Om^\la}
     \var^\la (\zeta)  L^{\la}(\zeta^\la,z^\la)\,
    d\sigma^{\la}(\zeta^{\la}).
\end{gather}
Let $\ell_1,\ldots,\ell_n$ be the standard basis of $\ker(I+\cL L)$.

\le{kl2}
 Let $\gaa^\la$ embed  $\pd\Om$
onto $\pd\Om^\lambda$   with $\gaa\in{\cL B}^{2+\all,1}(\pd\Om)$.
\bpp
\item Keep assumptions in {\rm a)} of \rla{kl}.
 Assume further that $ \psi_0\in\cL C_*^{0,1}(\pd\Om)$ and $\psi_i\in\cL C^1([0,1])$
  for $i>0$. Then $\var \in\cL C_*^{0,1}(\pd\Om)$
 and
\ga\label{tplv-}
\tilde\pd_\la\var^\la +\cL L^{\la*}\tilde\pd_\la\var^\la =\tilde\pd_\la\psi_0^\la
-( {\cL L}_1^{\la*} + {\cL L}_2^{\la*} )
\var^\la,\\
\int_{\pd\Om^\la}(\tilde\pd_\la\var^\la)\ell_i^\la\, d\sigma^\la=
\pd_\la\psi_i^\la.
\label{tplv}
\end{gather}
\item  Keep assumptions in {\rm b)} of \rla{kl}.
 Assume further that $ \psi_0\in\cL C_*^{0,1}(\pd\Om)$ and $\psi_i\in\cL C^1([0,1])$.
  Then $\var \in\cL C_*^{0,1}(\pd\Om)$
and
\ga \label{tplv-o}
\tilde\pd_\la\var^\la +\cL L^\la\tilde\pd_\la\var^\la =\tilde\pd_\la\psi^\la-
( {\cL L}_1^\la + {\cL L}_2^\la )
\var^\la,\\
\int_{\pd\Om^\la}(\tilde\pd_\la\var^\la)\ell_i^\la\, d\sigma^\la=
\pd_\la\psi_i^\la-\int_{\pd\Om^\la}\var^\la
(\pd_\la\ell_i^\la)\, d\sigma^\la,\quad 1\leq i\leq n.
\label{tplvo}
\end{gather}
\item Let $0\leq\beta\leq\all$. In {\rm a)} and {\rm b)} of \rla{kl},
 we have $\var\in\cL B^{\beta,j}_*(\pd\Om_\gaa)$
  provided $\psi_0\in \cL B^{\beta,j}_*(\pd\Om_\gaa)$,
   $\psi_i\in\cL C^j([0,1])$, and $\gaa\in \cL B^{j+1+\beta,j}(\pd\Om)$;
the same assertion holds if $\cL C$ substitutes for $\cL B$.
\epp
\ele
\begin{proof} a). Let us recall some identities in the previous proof.
Fix $\la$. Recall that $\ell_i$
are locally constant when $L=K$ or $-K$. We also use notation $f^\la(z^\la)=f(z,\la)$.
By \re{dint}-\re{e2za},   the difference quotient
$$\hat D(z, \mu)=\f{1}{\mu-\la}\Bigl(
\f{a^\mu(z)}{a^\la(z)}\varphi(z,\mu)-\varphi(z,\la)\Bigr)$$
 satisfies
\ga \label{dint+}\nonumber
\hat D(z, \mu)+\int_{\pd\Om}\hat
D(\zeta, \mu)  L^\la(\zeta^\lambda,z^\lambda)\, d\sigma^\la(\zeta^\la)
=\hat E_0(z, \mu)-\hat E_1(z, \mu)
-\hat E_2(z, \mu),\\
\label{dmea+}\nonumber
 \int_{\zeta\in\pd\Om}\hat D(\zeta, \mu)\ell_i^\la\, d\sigma^\la(\zeta^\la)
 =\f{1}{\mu-\la}(\psi_i^\mu-\psi_i^\la),
  \quad 1\leq i\leq n
 \end{gather}
 with
 \ga\label{e1zp+}\nonumber
\hat E_0(z,\mu)= \f{1}{\mu-\la}\Bigl(\f{a^\mu(z)}{a^\la(z)}\psi_0 (z,{\mu})-
  \psi_0 (z,{\la})\Bigr),\\
  \hat E_1(z, \mu)=\int_{\zeta\in\pd\Om}   \varphi (\zeta,{\mu})
   \f{L^{\mu}(\zeta^{\mu},z^{\mu})-
 \nonumber   L^{\la}(\zeta^{\la},z^{\la})}{\mu-\la}\,
    d\sigma^{\mu}(\zeta^{\mu}),\\
   \label{e2za+}\nonumber
   \hat  E_2(z, \mu)= \f{a^\mu(z)-a^\la(z)}{a^\la(z)(\mu-\la)}
   \int_{\zeta\in\pd\Om}       \varphi (\zeta,{\mu})  L^{\mu}(\zeta^{\mu},z^{\mu})\,
    d\sigma^{\mu}(\zeta^{\mu}).
\end{gather}
As $\mu\to\la$, it is clear that
 $\hat \psi_0(z, \mu)$ converges uniformly to $\tilde\pd_\la\psi_0^\la$.
 We want to show that as $\mu$ tends to $\la$,
  $\hat D(z, \mu)$ converges uniformly to a continuous function. Then the
  existence of the limit function  implies  that $\pd_\la\var^\la$ exists
  and  the limit function must be $\tilde\pd_\la\var^\la$.

By \rl{kl}, $\var\in\cL C^{0,0}(\pd\Om_\gaa)$. It is easy to see that
 $\hat E_1,\hat E_2$ are continuous at $\mu\neq\la$.
Also $\hat E_2(z,\mu )$ converges uniformly to $(\cL L_2^{\la*}\phi)(z^\la)$
 as $\mu\to\la$; in particular,
 $\{\hat E_2^\mu\}$   extends
to an element in $\cL C^{0,0}$.
For   $\zeta\neq z$,
by the mean-value theorem and \re{kta3p--}  we obtain
$$ 
\Bigl|\f{K^{\mu}(\zeta^{\mu},z^{\mu})-
    K^{\la}(\zeta^{\la},z^{\la})}{\mu-\la}\Bigr|\leq C|\zeta-z|^{\all-1}.
$$ 
Thus $\hat E_1(z,\mu )$ converges uniformly to $\cL L_1^{\la*}\var(z^\la)$
as $\mu\to\la$, and  $\{\hat E_1^\mu\}$
   extends to an element
in $\cL C^{0,0}$. By \rl{kl} (ii) with $\beta=0$, we conclude that as $\mu\to\la$,
 $\hat D(\cdot,\mu)$ has
a limit $\tilde\pd_\la\var^\la$ satisfying \re{tplv-}-\re{tplv}.

b).  By a),   $\phi_0,\ldots,\phi_m$ are of class
 $\cL C^{1,1}$ when $\gaa\in\cL B^{2+\all,1}$. Thus, in all cases,
 we have $\ell_i\in \cL C^{1,1}$.
Fix $\la$. We need   some minor changes in the above arguments.   The difference quotient
$\hat D(z, \mu)$
 satisfies
\gan
\hat D(z, \mu)+\int_{\pd\Om}\hat
D(\zeta, \mu)  L^\la(z^\lambda,\zeta^\lambda)\, d\sigma^\la(\zeta^\la)=
\hat E_0(z, \mu)+\hat E_1^*(z, \mu)
+\hat E_2^*(z, \mu),\\
 \int_{\zeta\in\pd\Om}\hat D(\zeta, \mu)\ell_i^\la\, d\sigma^\la(\zeta^\la)
 =\f{\psi_i^\mu-\psi_i^\la }{\mu-\la}
 -\int_{\zeta\in\pd\Om}   \varphi (\zeta,{\mu})  \f{\ell_i^{\mu}(\zeta^{\mu})-
    \ell_i^{\la}(\zeta^{\la})}{\mu-\la}\,
    d\sigma^{\mu}(\zeta^{\mu})
 \end{gather*}
 with
 \gan
  \hat E_1^*(z, \mu)=\int_{\zeta\in\pd\Om}   \varphi (\zeta,{\mu})
  \f{L^{\mu}(z^{\mu},\zeta^{\mu})-
    L^{\la}(z^{\la},\zeta^{\la})}{\mu-\la}\,
    d\sigma^{\mu}(\zeta^{\mu}),\\
    \hat  E_2^*(z, \mu)= \f{a^\mu(z)-a^\la(z)}{a^\la(z)(\mu-\la)}
   \int_{\zeta\in\pd\Om}       \varphi (\zeta,{\mu})  L^{\mu}(z^{\mu},\zeta^{\mu})\,
    d\sigma^{\mu}(\zeta^{\mu}).
\end{gather*}
By \rl{kl}, $\var$ is continuous. It is easy to see that $\hat E_1^*,\hat E_2^*$
are continuous at $\mu\neq\la$.
Also $\hat E_2^*(z,\mu )$ converges uniformly to $\cL L_2^\la\var(z^\la)$ as $\mu\to\la$,
and $\{E_2^{*\mu}\}$   extends
to an element in $\cL C^{0,0}$.
Also $\hat E_1^*(z,\mu )$ converges to $\cL L_1^{\la}\var(z^\la)$ as $\mu\to\la$,
and $\{E_1^{*\mu}\}$
   extends  to an element
in $\cL C^{0,0}$. By \rl{kl} (ii) with $\beta=0$, we conclude that as $\mu\to\la$,
 $\hat D(z,\mu)$ converges uniformly to
a limit function, which is denoted by $(\tilde\pd_\la\var^\la)(z^\la)$ and satisfies
\re{tplv-o}-\re{tplvo}.

c).  When $j=0$ we get $\var\in\cL B^{\beta,0}$ from
 \rl{kl} (i) and (ii) and  \rp{bt} (ii) and (iii) and we further have $\var\in\cL C^{\beta,0}$
for $\psi_0\in\cL C^{\beta,0}$ and $\gaa\in\cL C^{1+\all,0}$. Assume that the
assertions hold when
  $j$ is replaced by $j-1$. Thus $\var\in\cL B_*^{\beta,j-1}$.
We first consider case a).  Then,  we
have \re{tplv-}-\re{tplv}. By \re{ktapi4} with $j'=j$, we know that
 $\cL L_1^{*}\var$ and $\cL L_2^{*}\var$ are in $\cL B_*^{\beta,j-1}$.
 Also $\{\tilde\pd_\la \psi_0^\la\}$
 is in $\cL B_*^{\beta,j-1}$ and $\{\pd_\la\psi_i^\la\}$ are in $\cL C^{j-1}([(0,1])$ for $i>0$.
 By the induction hypothesis,
we conclude that $\{\tilde\pd_\la\var^\la\}\in\cL  B_*^{\beta, j-1}$.
Hence $\{\pd_\la\var^\la\}\in\cL  B_*^{\beta, j-1}$.
Combining with $\var\in \cL B^{\beta, 0}$, we get $\var \in\cL  B_*^{\beta, j}$.
 We can also verify that $\var\in\cL C^{\beta,j}_*$
by \rl{kl} (iii) and \re{ktapi6}, when $\psi\in\cL C^{\beta,j}_*$ and
$\gaa\in\cL C^{j+1+\all,j}$.

For case b), we first apply results from a) and conclude that  $\phi_0,\ldots, \phi_m$
are in $\cL B_*^{\all,j}$. This shows that
$\{\int_{\pd\Om^\la}\var^\la\phi_i^\la\, d\sigma^\la\}$ are in $\cL C^{j-1}$
if $\var\in\cL C^{j-1}$. We substitute \re{tplv-o}-\re{tplvo} for \re{tplv-}-\re{tplv}
and substitute \re{ktapi1}-\re{ktapi2} for \re{ktapi4}-\re{ktapi6}   with $j'=j$.
With minor changes in the arguments, we verify the assertions for b) too. \end{proof}

By \rp{bt} (iv) and \nrc{ffdi}, we have proved all required regularity in higher order
derivatives
of solutions to the integral equations for a fixed parameter.
We are ready to study the regularity of higher order derivatives for the parameter case.
\pr{kl3} Let $\gaa^\la$ embed  $\pd\Om$
onto $\pd\Om^\lambda$. Let $k\geq j\geq0$ and $0\leq\beta\leq\all<1$. Let $\beta>0$ when $k>0$.
Suppose that $\gaa\in  \cL B^{k+1+\all,j}(\pd\Om)$,
$\psi_0\in\cL B^{k+\beta,j}(\pd\Om_\gaa)$, 
 $\psi_i \in\cL C^{j}([0,1])$ for $1\leq i\leq n$, and $\var^\la\in L^1(\pd\Om^\la)$.
\bpp
\item Let $\{\cL L^\la\}$ be   $ \{\cL K^{\la}\}$ or $ \{-\cL K^{\la}\}$.
Suppose that
$$ 
\var^\la+\cL L^{\la*}\var^{\la}=\psi_0^\la,
\qquad\jq{\var^\la, \ell_j^\la}\,=\psi_i^\la, \quad 1\leq i\leq n.
$$ 
Then $\var\in\cL B^{k+\beta,j}(\pd\Om_\gaa)$.
 \item Let $\{\cL L^\la\}$ be   $ \{\cL K^{\la*}\}$ or $ \{-\cL K^{\la*}\}$.
Suppose that
\eq{vlll2}
\var^\la+\cL L^{\la}\var^{\la}=\psi_0^\la,
\qquad\jq{\var^\la, \ell_j^\la}\,=\psi_i^\la, \quad 1\leq i\leq n.
\eeq
Then $\var\in\cL B^{k+\beta,j}(\pd\Om_\gaa)$.
\item
Let $l\leq k+1$ and $\beta>0$ when $l>0$. Let $\{\cL L^\la\}$ be
 $\{\cL K^\la\}$ or $\{-\cL K^\la\}$. Suppose that
$\psi_0\in\cL B^{l+\beta,j}(\pd\Om_\gaa)$.
Then the solution  $\var$ to   \rea{vlll2} is in
$\cL B^{l+\beta,j}(\pd\Om_\gaa)$.  
\epp
$a),b)$ and $c)$ remain true if the symbol
  $\cL C^\bullet$ substitutes for $\cL B^{\bullet}$ in all conditions and assertions.
\epr
\begin{proof}
The   proposition is proved when $k=0$ and $\beta=0$, by \rp{bt} (ii) and \rl{kl}
(i) and (ii). We may assume that $\beta=\all$.

a).
We first verify the assertions when $j=0$. When $k=0$ we get $\var\in\cL B^{\beta,0}$
by \rp{bt} (ii)
and $\var\in\cL C^{\beta,0}$ by
\rl{kl} (i) and (iii).  We apply \rp{ktapi}.
Then \re{ktapi4} implies    $\var\in\cL B^{k+\beta,0}(\pd\Om_\gaa)$
for $\gaa\in\cL B^{k+1+\all,0}$;
\re{ktapi6}  implies    $\var\in\cL C^{k+\beta,0}(\pd\Om_\gaa)$
for $\gaa\in\cL C^{k+1+\all,0}$ and $\psi_0\in\cL C^{k+\all,0}$.

For $j>0$,  assume that a) is valid when $j$ is replaced by $j-1$.
Thus, $\var\in\cL B^{k+\beta,j-1}(\pd\Om)$. And
$\var\in\cL C^{k+\beta,j-1}(\pd\Om_\gaa)$
for $\gaa\in\cL C^{k+1+\all,j-1}$ and $\psi_0\in\cL C^{k+\all,j-1}$. Since
 $\psi_0\in\cL C^{1,1}$ and $\psi_i\in\cL C^1([0,1])$ for $i>0$, \rl{kl2}
 implies that
\gan
\tilde\pd_\la \var^\la +\cL L^{\la*}\tilde\pd_\la\varphi^\la =\tilde\pd_\la\psi_0^\la
-( {\cL L}_1^{\la*} + {\cL L}_2^{\la*} )
\varphi^\la,\quad \jq{\tilde\pd_\la\var^\la,\ell_i^\la}=\pd_\la\psi_i^\la, \quad i>0.
\end{gather*}
Here $\{\pd_\la\psi_i\}\in\cL C^{j-1}$. Also,  $\tilde\pd_\la\psi_0\in
 \cL B^{k-1+\all,j-1}$ by
 $$ 
  \tilde\pd_\la\psi_0^\la(z^\la)= \pd_\la\psi_0^\la (z^\la)+ \psi_0^\la(z^\la)
\pd_\la\log |\pd_\ta\gaa^\la|,\ \{\pd_\la\log |\pd_\ta\gaa^\la|\}\in\cL B^{k-1+\all,j-1}.
$$ 
 Combining $\var\in\cL B^{k+\beta,j-1}(\pd\Om)\subset \cL B^{k-1+\beta,j-1}(\pd\Om)$ with
\gan
 {\cL L}_1^{\la*}\varphi(z )=
\int_{\pd\Om^\la}    \varphi^\la (\zeta)\pd_\la
\left\{ L^{\la}(\zeta^\la,z^\la)\right\}\,d\sigma^\la(\zeta^\la),\\
 {\cL L}_2^{\la*}\varphi(z)=(\pd_\la\log |\pd_\ta\gaa^\la(z)| ) \int_{\pd\Om^\la}
     \varphi^\la (\zeta)  L^{\la}(\zeta^{\la},z^{\la})\,
    d\sigma^{\la}(\zeta^{\la}),
\end{gather*}
we see  from \re{ktapi4} that
  $\cL L_i^*\var\in\cL B^{k-1+\all,j-1}$.
Thus $\{\tilde\pd_\la\var^\la\}$ and   $\{\pd_\la\var^\la\}$ are in $
 \cL B^{k-1+\all,j-1}(\pd\Om_\gaa)$.  Combining with $\{\pd_\la\var^\la\}
\in\cL B^{k+\all,0}(\pd\Om_\gaa)$, we see that  $\{ \var^\la\}
$ is in $\cL B^{k+\all,j}(\pd\Om_\gaa)$.
To verify $\var\in\cL C^{k+\all,j}$ for $\gaa\in\cL C^{k+1+\all,j}$,
we use \re{ktapi6} instead of \re{ktapi4}.

b). Note that the case where $k=l=0$ is
established by \rl{kl}. So we assume that $k\geq1$.
 Although we are dealing
with the same integral equations as in a), i.e.
 $\var^\la\pm \cL K^{\la*}\var^\la=\psi_0^\la$, the functions $\ell_i$
 appeared  in  $\jq{\var^\la,\ell_i^\la}\, =\psi_i^\la$
are no longer constants in general. Nevertheless,  a) implies that
$\phi_0, \phi_1,\ldots, \phi_m$
are of class $\cL B^{k+\all,j}$ or are of class $\cL C^{k+\all,j}$ when
 $\gaa\in\cL C^{k+1+\all,j}$.
In any case, we have $\ell_i\in\cL B^{k+\all,j}$. Then
\gan 
\int_{\pd\Om^\la}(\tilde\pd_\la\var^\la)\ell_i^\la\, d\sigma^\la=
\pd_\la\psi_i^\la-\int_{\pd\Om^\la}\var^\la
(\pd_\la\ell_i^\la)\, d\sigma^\la
\end{gather*}
are in $\cL C^{j-1}([0,1])$, if we know $\psi_i\in\cL C^j$ and $\var\in\cL C_*^{0,j-1}$.
 The latter is ensured
if $\var\in\cL B^{k-1+\all,j-1}$ with $k\geq j$ and $j\geq1$.
 Then, $\cL L_i(\cL B^{k-1+\all,j-1})$ is contained in $ \cL B^{k-1+\all,j-1}$ by
 \re{ktapi4} and
 \gan 
  {\cL L}_1^{\la}\var(z )=
\int_{\pd\Om^\la}    \var^\la (\zeta^\la)\pd_\la
\left\{ L^{\la}(z^\la,\zeta^\la)\right\}\,d\sigma^\la(\zeta^\la),\\
 {\cL L}_2^{\la}\var(z)=(\pd_\la\log |\pd_\ta\gaa^\la(z)| ) \int_{\pd\Om^\la}
      \var^\la (\zeta)  L^{\la}(z^\la,\zeta^\la)\,
    d\sigma^{\la}(\zeta^{\la}).
\end{gather*}
 Finally,
$\{\pd_\la\log |\pd_\ta\gaa^\la|\}$ is   in $\cL B^{k-1+\all,j-1}(\pd\Om)$,
which implies that
if $\{\tilde\pd_\la\var^\la\}\in
\cL B^{l-1+\all,j-1}(\pd\Om)$  then $\{\pd_\la\var^\la\}$ remains in
$\cL B^{l-1+\all,j-1}(\pd\Om)$.
With these observations,
the induction proof  in a) is valid without essential changes.
To verify $\var\in\cL C^{k+\all,j}$  when $\gaa\in\cL C^{k+1+\all,j}$ and $\psi_0\in\cL C^{k+\all,j}$,
we use \re{ktapi6} instead of \re{ktapi4}.

\medskip

c).
 To show
 $\var\in\cL B^{k+1+\all,j}$, we cannot use the
  induction proof in b) when $\psi_i\in\cL B^{k+1+\all,j}$. For
that $\{\tilde\pd_\la\var^\la\}\in
\cL B^{k+\all,j-1}(\pd\Om)$, defined by \re{tplv14}, does not imply that
$\{\pd_\la\var^\la\}$ is in $\cL B^{k+\all,j-1}(\pd\Om)$.

Instead, we   apply induction on $l$. If $l=0$,
 by \rl{kl2} c) we get $\var \in B_*^{\all, j}$ and  $\var\in\cL C^{\all,j}_*$
when $\psi_0\in\cL C^{\all,j}_*$ and $\gaa\in\cL C^{j+1+\all,j}$. In particular c)
is valid when $l=0$.
Assume that   c) is valid when $l$ with $>0$ is replaced by $l-1$.
We have
$$
\pd_{\ta^\la}\var^\la-\cL L^{\la*}\pd_{\ta^\la}\var^\la=\pd_{\ta^\la}\psi_0^\la, \quad
\int_{\gaa_i}\pd_{\ta^\la}\var^\la\, d\sigma^\la=0, \quad i\geq0.
$$
Note that $\{\pd_{\ta^\la}\psi_0^\la\}$ is in $\cL B^{l-1+\all,j}$ when $l-1\geq j$
 and it is in $\cL B^{j-1+\all,j-1}$
when $l=j$. By b), we conclude that $\pd_\ta\var\in\cL B^{l-1+\all,j}$ for
$k\geq l-1\geq j$ and it is in $\cL B^{j-1+\all,j-1}$
when $l=j$. Combining with $\var\in\cL B_*^{\all,j}$, we conclude that
$\var\in\cL B^{l+\all,j}$. We can also verify
that $\var\in\cL C^{l+\all,j}$ when $\psi_0\in\cL C^{l+\all,j}$ and $\gaa\in\cL C^{k+1+\all,j}$.

One can give another proof for c)  by using \rl{klem}
and a), which avoids the induction argument. We leave the details to the reader.
  \end{proof}

\setcounter{thm}{0}\setcounter{equation}{0}
\section{H\"older spaces for exterior domains with parameter}
\label{sec8}

In this section, we return to the definition  of H\"older spaces with parameter.
However, the reader can turn to
the proof of   \rt{dnpb} for
interior domains by skipping this section.
\rl{whitc}
shows that   elements in $\cL B^{k+\beta,j}(\pd\Om_\gaa)$ extend to  elements
 in $\cL B^{k+\beta,j}(\ov\Om_\Gaa)$. \rl{spac} shows that possibly by restricting
 $\la$ to a subinterval,  we can extend a family of embeddings $\gaa^\la$ of $\pd\Om$
with  $\gaa\in\cL B^{k+\beta,j}(\pd\Om)\cap\cL C^{1,0}(\pd\Om)$
to a family of
embeddings $\Gaa^\la$ of $\ov\Om$ with  $\Gaa\in\cL B^{k+\beta,j}(\ov\Om)
\cap\cL C^{1,0}(\ov\Om)$.
The two lemmas and \rl{chainp} form basic properties of   H\"older spaces for interior
 domains with parameter.
We   also define
H\"older spaces for   exterior domains with parameter. Finally, we   extend
estimates on Cauchy transforms and simple
and double layer potentials to exterior domains for our H\"older spaces. To extend
 families of finitely
smooth embeddings
from $\pd\Om$ into $\ov\Om$, we   apply a type of Whitney extensions with parameter.
 However, the real analytic
extension is more  subtle, for which we  need the real analyticity of  solutions to the
Dirichlet problem  with real analytic parameter. The connection between extensions of
functions and solutions
of Dirichlet problem was observed by Whitney~\ci{Whthfo}.
When an exterior domain $\Om'=\cc\setminus \ov\Om$ is considered,  we   assume  without
loss of generality  that $\Om$ is simply connected.

\le{whit} Let $J,K$ be non-negative integers or $\infty$,
 and  let  $0\leq\beta<1$.  For $0\leq k<K$, let $\e_{k+1}$ be
 decreasing positive numbers and $0\leq j_k<J+1$ be non-decreasing
 integers. Suppose that $j_k=J$ for some $k$ if $J<\infty$ and $j_k$ tends to $J$ if
  $J$ and $K$ are infinite.
Let $\Om$ be a bounded domain with $\pd\Om\in\cL C^{K+\beta}\cap\cL C^1$.
Suppose that $f_i\in\cL B_*^{K-i+\beta,J}(\pd\Om)$
{\rm(}resp. $\cL C_*^{K-i+\beta,J}(\pd\Om)${\rm)} for $0\leq i<K+1$.
  There exists
  $Ef\in\cL B_*^{K+\beta,J}(\ov\Om)$ {\rm(}resp. $\cL C_*^{K+\beta,J}(\ov\Om)${\rm)} satisfying
  $\pd^i_\nu Ef=f_i$ for $0\leq i< K+1$. Furthermore, $Ef$ has the following properties.
\bppp
\item
The extension operator $f\to Ef$ depends only on $i$, $\pd\Om$ and the  upper bound
$M_i$ of $\e_{i}^{-1}$   and
  $|f_{l}|_{i-l+\beta,j_{i-1}}$ for $0\leq l\leq i$ and $0<i<K+1$.
  Moreover, \ga \label{exm} |Ef|_{k+\beta,j_k}\leq
\e_{k+1}+ C_{k}(\e,f)\sum_{i\leq k}|f_i|_{k-i+\beta,j_k}, \quad k<
K\leq\infty,
\\
\label{exm+}
|Ef|_{K+\beta,j}\leq C_{K}(\e,f)\sum_{i\leq K}|f_i|_{K-i+\beta,j},
\quad K<\infty, \quad 0\leq j<J+1.\end{gather}
Here $C_k(\e,f)$ depends only on $k$, $\pd\Om$ and $M_l$  for $0<l\leq k$.
\item
Assume further that
$f_0$ is constant and $f_i$ vanish in a neighborhood
$U$ of $p$ in $\pd\Om\times[0,1]$   for all $i>0$ with $i+J\leq K$. Then $Ef$ is constant on   some neighborhood
$V$ of $p$ in $\ov\Om\times[0,1]$. Moreover, $V$ depends only on $U$.
\eppp
\ele
\begin{proof}  We cover $\pd\Om$ by   open subsets $U_p$ of $\ov\Om$
and find $\cL C^\infty$
 functions $\chi_p$ with compact support in $U_p$ such that
 $\sum_{p=1}^{p_0}\chi_p=1$ on $\pd\Om$. Here $p_0$ is finite.
  We may  assume that there
 exist an open subset $V_p$ of $\ov\Om$, which contains $U_p$, and
 a $\cL C^{K+\beta}\cap\cL C^1$ diffeomorphism $\psi_i$ on $\ov{V_p}$
 which maps $V_p$ onto $V^*=(-2,2)\times[0,2)$ and
 $U_p$ onto $U^*=(-1,1)\times[0,1)$.
 We also assume
  that $\psi_i$ sends $\pd_\nu$ into $A_p\pd_y+B_p\pd_x$ such
  that  $1/C<|A_p|<C$. Here $A_p$ and $B_p$ are of class $\cL C^{K-1+\beta}\cap\cL C^0$
 on $\ov{V^*}$.
It suffices to find $h_p\in\cL B_*^{K+\beta,J}(\ov\Om)$ such that
$\supp h_p\subset V_p\times[0,1]$ and $\pd_\nu^ih_p=\sum_{l=0}^i
\binom{i}{l} f_l\pd_\nu^{i-l}\chi_p=h_{pi}$ on $(V_p\cap\pd\Om)
\times[0,1]$.  Then $Ef=\sum h_p$ is a desired extension.

 We now drop the subscript $p$
in all expressions.
In the new coordinates, we still denote $f,  h_{i}$, $\chi$, and $\nu$ by the
 same symbols. For instance,
$h_i$ denotes $h_{pi}\circ\psi_p^{-1}$.
 We have
$$
\pd_\nu^i=A^i\pd_y^i+\sum_{0<m\leq i}\sum_{l\leq m,l<i}B_{ilm}\pd_x^{m-l}\pd_y^l, \quad
\pd_y^i=A^{-i}\pd_\nu^i+\sum_{0<m\leq i}\sum_{l\leq m,l<i}\tilde B_{ilm}\pd_x^{m-l}\pd_\nu^l.
$$
Note that $B_{ilm}$ and $\tilde B_{ilm}$ are in $\cL C^{K-i+(m-1)+\beta}\cap\cL C^0$
on $\ov{V^*}$.  To achieve
$\pd_\nu^if=h_i$ on $V^*$,  we need
\eq{pyif}
\pd_y^if=A^{-i}h_i+\sum_{0<m\leq i}\sum_{l\leq m,l<i} \tilde B_{ilm}\pd_x^{m-l}h_l.
\eeq
Changing notation, we write the above as $\pd_y^if=f_i$.
The support of $f_i$ is contained in $[-1,1]\times[0,1]$ and $f_i$ is in
$\cL B_*^{K-i+\beta,J}([-2,2])$.
 If necessary we will  replace $\e_k$ by $\e_k/C$ with $C$
depending the  numbers of sets $U_p$ and diffeomorphisms $\psi_p$.

Fix $0<\del<1/2$. Let $\phi$ be a smooth function on $\rr$ with support
in $(-\del,\del)$ such that
$
\int_{\rr
}\phi(y)\, dy=1.
$
 We  first need to replace $y^if_i(x,\la)$ by $y^ig_i(x,y,\la)$
to achieve the   $\cL B_*^{K+\beta,J}$ smoothness;   when $K=\infty$, we still
need the replacement to estimate the $|\cdot|_{k+\beta,j}$ norm of
$y^ig_i(x,y,\la)$ via $|f_i|_{k-i+\beta,j}$. This requires us to  correct $i$-th
  $y$-derivative  of $y^ig_i(x,y,\la)$  due to the presence of  $y^{i_1}g_{i_1}(x,y,\la)$ for
  ${i_1}<i$. Take a cutoff $\chi(y)$ which has support in $(-1,1)$
  and equals $1$ on $(-1/2,1/2)$.
Let $a_i\in\cL B_*^{K-i+\beta,J}([-2,2])$    have   support in $[-1,1]\times[0,1]$.
 With constants $\del_i>0$ to be determined,
consider
\eq{gixy}
g_i(x,y,\la)=\int_{\rr} a_i(x-yz,\la)\phi(z)\, dz, \quad b_i(x,y,\la)=
\f{1}{i!}y^ig_i(x,y,\la)\chi(\del_i^{-1}y).
\eeq
It is clear that $\pd_\la^jg_i(\cdot,\la)$ are   $\cL C^{\infty}$ away from $y=0$ and
  $g_i\in
\cL C_*^{0,J}(\ov{V^*})$.
Also $g_i$ have support in $V^*\times[0,1]$.

To show that $b_i\in\cL B_*^{K+\beta,j}(\ov{V^*})$, it suffices to
show that $\pd_\la^j\pd^Ib_i$
 extend to functions
in $\cL B_*^{\beta,0}(\ov{V^*})$
for all $I$ with $|I|=k<K+1$.
 We first  derive a formula for derivatives. Write $I=I_1+I_2$ with $|I_2|=\min\{k,K-i\}$. We have
  \al\nonumber
\pd^{I_2}\int a_i (x-yz,\la) \phi (z)\, dz& =\int \pd^{I_2}( a_i(x-yz,\la))
 \phi (z)\, dz\\
 &  =\int (\pd^{|I_2|} a_i)(x-yz,\la)\phi_{I_2}^{(0)}(z)
  \, dz.\nonumber
\end{align}
Here and for the rest of the proof, $\phi_{*}^{(m)}(z)$   denotes
 a linear combination of
 $z^l\phi^{(n)}(z)$ with   $l\leq k$ and   $n\leq m$.
Assume now that $y\neq0$.
Changing   variables and interchanging the differentiation
and integration, we get
$$\pd^{I_1}\int (\pd^{|I_2|} a_i)(x-yz,\la)\phi_{I_2}(z)
  \, dz=
  \int \f{1}{y^{|I_1|+1}}(\pd^{|I_2|} a_i)(z,\la)\phi_{I_2I_1}^{(|I_1|)}
  \Bigl(\f{x-z}{y}\Bigr)\, dz. $$
Changing  variables again for the last integral, we get the formula
\eq{pdi1i2}
\pd^{I_1+I_2}\int a_i (x-yz,\la) \phi (z)\, dz = y^{-|I_1|}
  \int (\pd^{|I_2|} a_i)(x-yz)\phi_{I_1I_2}^{(|I_1|)}(z)\, dz,\quad y\neq 0.
  \eeq

Consider first the case where $i\leq k< K+1$.  For  $|I|=k$ and $y\neq 0$, we have
$$
\pd^I(y^ig_i(x,y,\la))=\sum_{i_1+|I_2|=k}
 C_{i_1I_2}\pd_y^{i_1}y^i\pd^{I_2}\int a_i (x-yz,\la) \phi (z)\, dz
$$
with $i_1\leq i$. Write $I_2=I_3+I_4$ with $|I_3|=k-i$ and $|I_4|=i-i_1$.
 By \re{pdi1i2} we get
\al\label{pjli}
&\pd^I(y^ig_i(x,y,\la))=\sum_{i_1\leq i}
 \int(\pd^{k-i} a_i) (x-yz,\la) \phi^{(i-i_1)}_{Iii_1 }
 (z)\, dz, \quad |I|=k.
\end{align}
It is obvious that the right-hand side extends to a function of class
 $\cL B_*^{\beta,j}(\ov{V^*})$.
Also, the $\cL C^\beta(\ov{V^*})$ norm of $\pd_\la^jD^I(y^ig_i(x,y,\la))$ in variables $x,y$
 is bounded by $C_{k,i}
|a_i|_{k-i+\beta,j}$.
 By dilation,  it is easy
to see that for $\chi^{\del_i}(z)=\chi(\del_i^{-1}z)$ with $0<\del_i<1$, we have
   $|\chi^{\del_i}|_{k+\beta}<C_k\del_i^{-k-\beta}$. Therefore,
\eq{bik1}
|b_i|_{k+\beta,j}\leq C_{k,i}\del_i^{-k-\beta}|a_i|_{k-i+\beta,j}, \quad 0\leq i\leq k.
\eeq

Next, we want to verify that $\pd_y^i(y^ig_i(x,y,\la))=i!a_i(x,\la)$ at $y=0$.
Fix $x$. By
$a_i\in\cL C_*^{0,j}([-2,2])$, $\supp a_i\subset[-1,1]\times[0,1]$ and \re{pjli} with $k=i$,
  the value
of $\pd_y^i(y^ig_i(x,y,\la))$ at $y=0$
depends only on $a_i(x)$. However,  the identity holds trivially for any   $\del_i\in(0,1)$,
 when $a_i$ is   constant. We now determine $a_i$ by taking $a_0=f_0$, and
 \eq{aifi}
 a_i=f_i- \pd_y^i|_{y=0}(b_0(x,y,\la)+\cdots +b_{i-1}(x,y,\la)).
 \eeq
By \re{bik1}, for $j<J+1$ and $i+k<K+1$
we get
\al\label{aifi+}
|a_i|_{k+\beta,j}&\leq|f_i|_{k+\beta,j}+\sum_{l<i}
 C_{i+k}\del_{i-1}^{-i-k-\beta}|a_l|_{k+i-l+\beta,j}
\\ &\leq|f_i|_{k+\beta,j}+\del_{i-1}^{-(i+k+\beta)l}\sum_{l<i}  C_{i+k}'
|f_l|_{k+i-l+\beta,j}, \nonumber
\end{align}
Here we have assumed that $\del_l$ decreases.
In particular, $a_i$ is in $\cL B_*^{K-i+\beta,J}([-2,2])$.
We have achieved
\eq{1ipd}
 \pd_y^i(b_0(x,y,\la)+\cdots+b_i(x,y,\la))=f_i(x,\la), \quad y=0.
\eeq

Consider now the case where $i> k=|I|$. By the product rule and \re{pdi1i2},
\al\label{pjli+}
&\pd^I(y^ig_i(x,y,\la)\chi(\del_i^{-1}y))=\sum C_{i_1i_2I_3}\pd_y^{i_1}y^i\cdot
\pd_y^{i_2}\chi(\del_i^{-1}y)\cdot
\pd^{I_3}g_i(x,y,\la)\\
&\qquad =\sum\tilde C_{i_1i_2I_3}
\pd_y^{i_2}\chi(\del_i^{-1}y)\cdot
y^{i-i_1-|I_3|}
 \int a_i(x-yz,\la)\phi^{(|I_3|)}(z)\, dz.\nonumber
\end{align}
Here the summation runs over $i_1+i_2+|I_3|=k$.
The $\cL C^\beta$ norm of
$(\del_i^{-1}y)^{i-i_1-|I_3|}\cdot \chi^{(i_2)}(\del_i^{-1}y)$ is bounded by $C\del_i^{-\beta}$.
Thus for  any  $\del_i\in(0,1)$
\ga
\pd_y^k(b_i(x,y,\la))=0, \quad y=0, \label{bik2-} \quad k<i,\\
\label{bik2}
|b_i|_{k+\beta,j}\leq C_{k,i}\del_i^{i-k-\beta}|a_i|_{\beta,j},\quad k<i.
\end{gather}
By \re{bik2} we  inductively  choose decreasing
$\del_{i}$ such that
\eq{bik4}
|b_i|_{i-1+\beta,j_{i-1}}\leq\del_{i}^{1-\beta}|a_{i}|_{\beta,j_{i-1}}
\max_{k<i} C_{k,i}<\f{\e_{i}}{2^{i}},\qquad i\geq1.
\eeq
Take $Ef(x,y,\la)=\sum_{i+J\leq K} b_i(x,y,\la)$. By  \re{bik4},
 we get $\sum_{i>k}|b_i|_{k+\beta,j_k}<\e_{k+1}$
for $0\leq k<K$.
Combining it with \re{1ipd}, \re{bik2-} and \re{bik1}, we obtain $\pd_y^iEf=f_i$ at $y=0$ and
\re{exm}, respectively. Combining \re{bik1} for $k=K$ and \re{aifi+}  gives us
\re{exm+}. From \re{gixy} and \re{pjli}, we see that $b_k\in\cL C_*^{K+\beta,J}(\ov{V^*})$
 when $f_k\in\cL C_*^{K-k+\beta,J}(\ov{V^*})$. Using the convergence
of $\sum |b_i|_{k+\beta,j_k}$ again, we obtain $Ef\in\cL C_*^{K+\beta,J}(\ov{V^*})$.
The dependence of $Ef$ and $C_k(\e,f)$  on norms of $f_i$, as stated
in the lemma,  is determined by \re{pyif}, \re{aifi+} with $k=0$ and \re{bik4}.

Note that (ii) follows from the extension formulae immediately.
Indeed, the partition of unity for $\pd\Om$ preserves conditions $\pd_\nu^if=0$ for $i>0$
and $f_0$ being constant  in a neighborhood of $z_0$ in
$\pd\Om$, by starting with one of
$\chi_{p}$'s to be $1$ near $z_0$ and all other $\chi_p$'s
 to be $0$ near $z_0$. From \re{pyif}, we have $\pd_y^if=0$ for $i>0$ near $z_0$.
  By \re{gixy} and shrinking
the support of $\phi$ if necessary,
  $Ef$ is constant near $z_0$.
 \end{proof}

\le{whitc} Let $J,K,\beta,\e_{k+1},j_k$ be as in \rla{whit}. Assume further that
$J\leq K$, and $K> k+j_k$ for $0<k<K$.
Let $\Om$ be a bounded domain with $\pd\Om\in\cL C^{K+\beta}\cap\cL C^1$.
Suppose that $f_i\in\cL B^{K-i+\beta,J}(\pd\Om)$
{\rm(}resp. $\cL C^{K-i+\beta,J}(\pd\Om)${\rm)} for  all $i\geq0$
satisfying $i+J\leq K$.
  There exists
  $Ef\in\cL B^{K+\beta,J}(\ov\Om)$ {\rm(}resp. $\cL C^{K+\beta,J}(\ov\Om)${\rm)} satisfying
  $\pd^i_\nu Ef=f_i$ for $i+J\leq K$. Furthermore, \ga
\label{cexm}
\|Ef\|_{k+\beta,j_k}\leq \e_{k+1}+ C_{k}(\e,f)\sum_{i\leq k, i+J\leq K}
\|f_i\|_{k-i+\beta,j_k}, \quad k< K\leq\infty,
\\
\label{cexm+}
\|Ef\|_{K+\beta,j}\leq C_{K}(\e,f)\sum_{i\leq K-J}\|f_i\|_{K-i+\beta,j},
\quad K<\infty, \quad 0\leq j<J+1,\end{gather}
where the extension operator $f\to Ef$ depends only on $i$, $\pd\Om$, and the
upper bound $M_i$ of  $\e_{i}^{-1}$
and $\|f_{l}\|_{i-l+j_{i-1}+\beta,j_{i-1}}$ for $l\leq i$, $i>0$ and $i+J\leq K$.
Furthermore, $C_{k}(\e,f)$ depends on $k$, $\pd\Om$ and $M_l$ for $l\leq k$;
$Ef$ is constant near $(z_0,0)\in\ov\Om\times[0,1]$
if near $z_0\in\pd\Om$, $f_0-f_0(z_0)$ and $f_i$ vanish for $i>0$ and $i+J\leq K$.
\ele
\begin{proof} We   use $Ef=\sum_{i+J\leq K} b_i$ with $b_i$ being of the form \re{gixy}.
We still have $\pd_y^kb_i=0$ for $k<i$ and $\pd_y^ib_i=a_i$ at $\la=0$ as they hold for
  $\del_i\in(0,1)$, provided $a_i\in\cL
B^{K-i+\beta,J}(\pd\Om)(\subset\cL B^{K-J-i,J}_*(\pd\Om))$.
We rewrite  previous estimates in norms $\|\cdot\|$ instead of $|\cdot|$.
Assume that $i+J\leq K$, $j<J+1$, and  $(j\leq )\ k<K+1$.
By \re{bik1} and \re{bik2}, we have
\al\label{bik1+}
\|b_i\|_{k+\beta,j}&=\max_{l\leq j}|b_i|_{k-l+\beta,l}\leq
\max_{l\leq j}\{C_{k,i}\del_i^{-k-\beta}|a_i|_{k_1+\beta,l}\}\\
&\quad\leq
C_{k,i}\del_i^{-k-\beta}\|a_i\|_{k_2+\beta,j},\qquad i\leq k.
\nonumber
\end{align}
Here $k_1=\max\{k-l-i,0\}$ and $k_2=\max\{k-i,j\}\leq K-i$.  By \re{bik2} again, we have
\ga\label{bik2+}
\|b_i\|_{k+\beta,j}=\max_{l\leq j}|b_i|_{k-l+\beta,l}
\leq  C_{k}\del_i^{1-\beta}\|a_i\|_{j+\beta,j}, \quad i>k.
\end{gather}
Assume further that $l\leq i$. By \re{bik1+} we  get
\al
\nonumber
\|a_l\|_{ k+\beta,j}&\leq\|f_l\|_{ k+\beta,j}
+\sum_{m<l}\|b_m\|_{k+l+\beta,j}
\\
\nonumber
 &\leq\|f_l\|_{ k+\beta,j}+\del_{i-1}^{-(i+k+\beta)}\sum_{m<l} C_{i}
\|a_m\|_{ k+l-m+\beta,j},\\
\|a_l\|_{ k+\beta,j}&
\label{A824-}
\leq\|f_l\|_{ k+\beta,j}+\del_{i-1}^{-(i+k+\beta)i}\sum_{m<l} C_{i}'
\|f_m\|_{ k+l-m+\beta,j}.
\end{align}
Thus, by  \re{A824-} with $l=i$, $a_i$  is in $\cL B^{K-i+\beta,J}$;
by \re{bik1+}-\re{bik2+},  $b_i\in\cL B^{K+\beta,J}$.
 Therefore, by \re{bik2+}-\re{A824-},  we can inductively choose decreasing $\del_i$ such that
\gan
\|b_i\|_{i-1+\beta,j_{i-1}}\leq
  C_{i}\del_i^{1-\beta}\|a_i\|_{j_{i-1}+\beta,j_{i-1}}<\f{\e_i}{2^i},
\end{gather*}
where $\del_i$ depends on the upper bound of $C_i>1,\e_i^{-1}$ and $
\|f_l\|_{i-l+j_{i-1}+\beta,j_{i-1}}$ for $l\leq i$.
The rest of arguments in the previous proof is valid.
\end{proof}
The above proof
for non-parameter case without
estimate on norms   is in~\ci{Honize} (p.p.~16 and 18).
See also~\ci{BGR} for   different spaces with parameter.
For the proof of \rt{nrefl}, we need the following extension lemma.

\le{whit+} Let $J,K$ be integers or $\infty$
 and  let  $0\leq\beta<1$.
Let $\Om$ be a bounded domain with $\pd\Om\in\cL C^{K+\beta}\cap\cL C^1$.
\bppp
\item  Suppose that    $f_j\in\cL C^{K+\beta}(\pd\Om)$
   for $0\leq j<J+1$.
  There exists
  $Ef\in\cL C_*^{K+\beta,J}(\pd\Om)$ 
  satisfying
  $\pd^j_\la Ef=f_j$ for $0\leq j< J+1$ at $\la=0$.
  \item Let $J\leq K$.
Suppose that    $f_j\in\cL C^{K-j+\beta}(\pd\Om)$
 for $0\leq j<J+1$.
  There exists
  $Ef\in\cL C^{K+\beta}(\pd\Om\times[0,1])$
  satisfying
  $\pd^j_\la Ef=f_j$ for $0\leq j< J+1$ at $\la=0$; in particular
  $Ef\in\cL B^{K+\beta,K}(\pd\Om)$.
\item In $(i)$ and $(ii)$,   if  near
  $p\in\pd\Om$ $f_0$ is constant and $f_i$ vanish for $i>0$, then $Ef$ is
  constant on $V\times[0,1]$ for some neighborhood
  $V$ of $p$.
  \item  $(i)$, $(ii)$ and $(iii)$ hold if $\ov\Om$ substitutes for $\pd\Om$.
 \eppp
 \ele
\begin{proof} (i).  When $J$ is finite, we simply
take $Ef(x,\la)=\sum_{j=0}^J\la^j f_j(x)$. Assume that $J=\infty$. Let $\chi(\la)$
be a $\cL C^\infty$ function which has
  support in $[0,1/2]$ and equals $1$ near $\la=0$.
We choose  $0<\del_j<1/2$ satisfying
$
 \del_j|f_j|_{2j}|\chi|_j<2^{-j}.
$ Then
$
Ef(x,\la)=\sum\la^jf_j(x)\chi(\del_j^{-1}\la)
$
is a desired extension.

(ii)-(iii).
The extension $Ef$ is a special case of \rl{whit} where the parameter $\la$ is
 absent and the variable  $y$ in its proof
is replaced by $\la$.
We first find an extension $Ef\in\cL B^{K+\beta,J}(\pd\Om)$.
Using a partition of unity and  local change of coordinates of class
$\cL C^{K+\beta}\cap\cL C^1$, we may assume that
$\pd\Om$ contains $[-2,2]\times\{0\}$,  $\ov\Om$ contains $[-2,2]\times[0,1]$, and $f_i$ have
support in $[-1/4,1/4]\times\{0\}$.
 Locally we find an extension $Ef\in\cL C^{K+\beta}([-2,2]\times[0,1])$ such
 that $\pd_y^jEf(x,0)=f_j(x)$ and $\supp Ef\subset[-1,1]\times[0,1/2]$.
 Then $Ef(x,\la)$ is a desired extension. It is clear that (iii) follows from
 the extension formulae.

(iv).
For extension $Ef\in\cL C^{K+\beta}(\ov\Om\times[0,1])$, again by partition of
unity for $\ov\Om$, we may assume that all $f_i$
have   support in $(-1/4,1/4)\times[0,1/4)$. Next, we apply \rl{whit} for the
non-parameter version
and extend $f_i$ across the boundary of $\pd\Om$ to $(-1/2,1/2)\times(-1/2,1/2)$.
We still have $f_i\in\cL
C^{K-i+\beta}$ and $f_i$ have compact support.
We substitute   \re{gixy} with
\eq{gixl}
g_i(x,\la)=\int_{\rr^2} a_i(x-\la z)\phi(z)\, dz, \quad b_i(x,\la)=
\f{1}{i!}\la^ig_i(x,\la)\chi(\del_i^{-1}\la).
\eeq
where $a_i\in\cL C^{K+\beta-i}([-3/4,3/4]^2)$ and $\supp a_i\subset(-1/2,1/2)^2$.
The arguments
 in the proof of \rl{whit} are written for one   variable $x$.
However, when $x\in\rr^2$ or in higher dimensional Euclidean spaces, the identities
 require minor changes only. We will leave the details to the reader. In conclusion, one
can find $Ef(x,\la)=\sum b_i(x,\la)$ such that
$Ef\in\cL C^{K+\beta}([-2,2]^2\times[0,1])\subset\cL B^{K+\beta,K}([-2,2]^2)$,
$\supp Ef\subset[-1,1]^2\times[0,1]$ and $\pd_\la^jEf=f_j$.
\end{proof}

\begin{rem}
As shown in \rl{chainp},  the composition of functions
is   restrictive for spaces $\cL C^{k+\all,j}$.
We do not know if the $g_i$
in \re{gixl} are of class $\cL C^{K+\all,j}(\ov\Om)$ when $K+\all$ is finite
but  not an integer; therefore, we do not
know if  there exists an extension $Ef$ in (ii) of \rl{whit+} that is of class
 $\cL C^{K+\all,J}(\ov\Om)$.
\end{rem}

\le{extm} Let $1\leq k\leq\infty$ and $0\leq\beta<1$.
 Let $\Om_i$ be     bounded domains of  $\cL C^{k+\beta}$ boundary.
Let $\gaa$ be an orientation preserving $\cL C^{k+\beta}$ diffeomorphism from
$\pd\Om_1$ onto $\pd\Om_2$.
Then $\gaa$ extends to a $\cL C^{k+\beta}$ diffeomorphism from $\ov\Om_1$ onto $\ov\Om_2$
and it also extends to a $\cL C^{k+\beta}$ diffeomorphism from $\ov{\Om_1'}$ onto $\ov{\Om_2'}$
which is identity on $|z|>R$ when $R$ is sufficiently large.
\ele
\begin{proof} We first prove the assertions when $\Om_i$ are  simply connected.
 Let $\gaa_1\colon\pd\D\to \pd\Om_1$ be a $\cL C^{k+\beta}$
 parameterization. Approximate $\gaa_1$ in $\cL C^1$ norm by a $\cL C^\infty$
 parameterization $\tilde\gaa_1\colon
 \pd\D$ to $\pd\tilde\Om_1$. Then $\tilde\gaa_1\gaa_1^{-1}-I$ has a small  $\cL C^1$
 norm on $\pd\Om_1$.
 By Whitney's extension theorem, it extends to a $\cL C^{k+\beta}$ mapping $\var$
  mapping from $\ov{\Om_1}$ into $\cc$
 with small $\cL C^1$ norm. Then $I+\var$ is a $\cL C^{k+\beta}$ diffeomorphism mapping
  $\ov{\Om_1}$ onto $\ov{\Om_1^*}$ with $\cL C^\infty$
 boundary. Therefore, we may assume that $\pd\Om_i$ have $\cL C^\infty$ boundary. Thus, we may
 further assume that $\Om_i$ are the unit disc,  say,  by  Kellogg's  Riemann mapping theorem.
Since $\gaa$ preserves the orientation of the unit circle, then $\gaa(e^{i\theta})
=e^{i(\theta+a(\theta))}$.
Here $a$ is $2\pi$-periodic and $1+a'>0$.  Let $\rho\colon[0,\infty)\to [0,1]$ be a smooth
 function which has support in $(1/2,2)$ and equals $1$ near $1$.
Then $\Gaa_0(re^{i\theta})=re^{i(\theta+\rho(r)a(\theta))}$ is a desired extension.

To extend $\gaa$ to the unbounded component, using time-one mappings of vector fields of compact support,
we may assume that $0\in\Om_i$. Using the inversion
$\iota_{0}(z)=1/z$
it suffices to show that in the above arguments we can extend $\gaa$ to a $\cL C^{k+\beta}$
 diffeomorphism from
$\ov{\Om_1}$ onto $\ov{\Om_2}$, which is the identity map near the origin.
Composing $\Gaa_0$
with the time-one map of a vector field which vanishing near $\pd\Om_1$, we may assume
 that  $\Gaa_0(0)=0$.
   Using a dilation, we may assume that $\Gaa_0(z)=\hat\Gaa_0(z)+E(z)$, where
    $|E|+|\pd E|<\e$ on $|z|<1/2$ and $\hat\Gaa_0$ is the linear part of $\Gaa_0$ at $z=0$. Let
$\chi=0$ on $|z|<1/4$ and $\chi=1$ on $|z|>1/2$. When $\e$ is small, $\Gaa_1(z)
=\Gaa_0(z)+\chi(|z|)E(z)$ is still a $\cL C^{k+\beta}$ diffeomorphism. Now $\Gaa_1$ is
linear near $0$. Since $\Gaa_1'(0)$ preserves orientations,
 by the Jordan normal form of $2\times 2$ matrices  we find two flows
 $X^t$ and $Y^t$ of vector fields vanishing at $0$ such that
 $\Gaa_1'(0)=X^1\circ Y^1$.
Let $\rho$ be a cutoff function which equals $1$ near the origin and has support in a
 small neighborhood
of the origin.
Then $(\rho Y)^{-1}\circ(\rho X)^{-1}\circ\Gaa_1$ is a desired extension.

The general case for bounded domains
 is obtained  by  induction on $m+1$, the number
of components of $\pd\Om_i$.
We have proved the lemma when $m=0$. Let $C_1$ be a component of the
inner boundary of $\Om_1$. Let $C_2=\gaa(C_1)$. Let $\om_i$ be bounded components of
 $\cc\setminus C_i$. Applying
results proved in previous paragraph, we find an extension $\Gaa_1$ of $\gaa|_{\pd\om_1}$
to $\ov{\om_1'}$. Replacing
$\gaa$ by $\Gaa_1\circ\gaa$, we may assume that $\gaa$ is the  identity on $C_1$.
Using a diffeomorphism of
class $\cL C^{k+\beta}$ from $\ov{\om_1'}$ onto $\cc\setminus\D$, we may assume that
 $C_1=C_2$ is the unit circle.
Let $\tilde\Om_i=\pd\Om_i\cup\ov\D$.
We know that $\gaa$ extends to a $\cL C^{k+\beta}$ diffeomorphism $\Gaa_0$ from $\ov{\tilde
\Om_1}$ onto $\ov{\tilde
\Om_2}$.  By the argument in the previous paragraph, we may achieve $\Gaa_0$ to be the identity
on $|z|<\e$ for some $
0<\e<1$. Let $\Gaa_2$
be a $\cL C^\infty$ diffeomorphism on $\cc$ which is the identity on the complement of the disc
 $\D_\rho$
and sends $\ov\D$ into $\D_\e$.
Here $\rho>1$ and $\ov D_\rho$ is contained in $\tilde\Om_1$. Then
$\Gaa_2^{-1}\circ\Gaa_0\circ\Gaa_2$ is a desired extension of $\gaa$
to $\ov{\Om_1}$.
\end{proof}

The proof of next lemma needs \rt{dnpb}
for the Dirichlet problem for interior domains. Our arguments are valid because
\rt{dnpb} are for embeddings $\gaa^\la$ which are restrictions of $\Gaa^\la$.
\le{spac}  Let $j,k$ be non negative integers or $\infty$. Let $0\leq\beta<1$.
 Let $\Om$ be a bounded domain in $\cc$ with $\pd\Om\in\cL C^{k+\beta}\cap\cL C^1$.
 Let $\gaa^\la$ be a family of orientation-preserving
 embeddings from $\pd\Om$ onto $\pd\Om^\la$
 with $\gaa\in\cL C^{1,0}(\pd\Om)$.
  Assume that $\gaa^\la$ send outer boundary to outer boundary.
For each $\la_0\in[0,1]$, there   exists $\delta>0$ such that
 if
 $I=[0,1]\cap[\la_0-\del,\la_0+\del]$
   substitute for $[0,1]$ in all function spaces,
   then $\gaa^\la$ extend to $\cL C^1$ embeddings
$\Gaa^\la$
from $\ov{\Om}$ onto $\ov{\Om^\la}$ with $\Gaa\in  \cL C^{1,0}(\ov\Om)$. Furthermore,
    if $\gaa$ is in $\cL B_*^{k+\beta,j}(\pd\Om),\cL C_*^{k+\beta,j}(\pd\Om)$,
$\cL B^{k+\beta,j}(\pd\Om)\,(k\geq j)$ and $\cL C^{k+\beta,j}(\pd\Om)\, (k\geq j)$,
  there exists an extension $\Gaa$ in $\cL B_*^{k+\beta,j}(\ov\Om),\cL C_*^{k+\beta,j}(\ov\Om)$,
$\cL B^{k+\beta,j}(\ov\Om)$ and $\cL C^{k+\beta,j}(\ov\Om)$, respectively;
 and if $\pd\Om$ and $\gaa$  are real analytic, then   $\Gaa\in\cL C^\om(\ov\Om\times I)$.
\ele
\begin{proof} With $\del$ to be determined, set $I=[\la_0-\del,\la_0+\del]\cap[0,1]$.
As stated in the lemma the space $\cL C^{1,0}(\ov\Om)$ and others depend on $\del$.

(i). We apply \rl{extm} and extend $\gaa^{\la_0}$ to a $\cL C^1$
diffeomorphism $\Gaa_0^{\la_0}$ from $\ov\Om$ onto
$\ov{\Om^{\la_0}}$.    Approximate $\Gaa_0^{\la_0}$ by a smooth map $\Gaa_1^{\la_0}$
and set $\Gaa_1^\la=\Gaa_1^{\la_0}$ for
all $\la$.
  We have
$|\gaa^{\la}-\Gaa_1^{\la}|_{1}<\e<\e_0$ for $\la\in I$ when $\del$ is sufficiently small.
  We   apply \rl{whit}   and extend
$\gaa-\Gaa_1$ to an element $\Gaa_2\in\cL B_*^{k+\beta,j}(\ov\Om)\cap
\cL C^{1,0}(\ov\Om)$ such
that $ |\Gaa_2|_{1,0}<\e_0+C(\e_0)\e<2\e_0$. Then $\Gaa^\la=\Gaa_2^\la+\Gaa_1^\la$ are
 extensions of $\gaa^\la$.
Also $|\Gaa^{\la_0}-\Gaa_0^{\la_0}|_1\leq 2\e_0$. Since $\Gaa_0^{\la_0}$ is an embedding,
 then $\Gaa^{\la_0}$
is also an embedding when $\e_0$ is sufficiently small. By continuity in $\cL C^1$ norm,
we know that $\Gaa^\la$
are embeddings for $\la\in I$ when  $\del$ is sufficiently small.
 Analogously, we can find the extensions for  other three cases.

(ii). For the real analytic case, the proof in (i) via extension does not apply. Instead,
we solve a Dirichlet problem with parameter.
We extend $\gaa^{\la_0}$ to a smooth embedding $\Gaa^0$ and approximate $\Gaa^0$ by real
analytic embeddings $\Gaa^{1/j}$ such that $|\Gaa^{1/j}-\Gaa^0|_{3/2}<1/j$.
For $f\in\cL C^{3/2}(\pd\Om^{1/j})$, let $T_jf$ be the unique harmonic function  $\Om_j$ which
is continuous up to the boundary
and has boundary value $f$.  Thus $T_j$ maps $\cL C^{3/2}(\pd\Om^{1/j}))$ into
$\cL C^{3/2}(\ov{\Om^{1/j}})$.
 We know that $T_j$ is injective
and the range of $T_j$ is the Banach space of harmonic functions on $\Om_j$
of class $\cL C^{3/2}(\ov{\Om^{1/j}})$. The inverse mapping
of $T_j$ is the restriction mapping, which is obviously bounded.
By the open mapping theorem, $T_j$ is bounded
with norm $\|T_j\|$.
Next, we want to show that the norms $\|T_j\|$ are bounded too. Define
$$
\Gaa^{\theta/{(j+1)}+(1-\theta)/j}=\theta\Gaa^{1/{(j+1)}}+(1-\theta)\Gaa^{1/{j}},\quad
0\leq\theta\leq1.
$$
Then $\{\Gaa^\la\}\in\cL C^{3/2,0}(\ov\Om)$.  When $\la$ is sufficiently small,
 $\Gaa^\la$ embeds $\ov{\Om}$ onto
$\ov{\Om^\la}$. Assume for the sake of contradiction that $\|T_j\|$ are not   bounded. We find
$f^{1/j}\in\cL C^{3/2}(\pd\Om^{1/j})$ such that $|T_jf^{1/j}|_{3/2}=1$ and
$|f^{1/j}|_{3/2}\to0$ as $j\to\infty$. Define
$$
f^{\theta/{(j+1)}+(1-\theta)/j}\circ\Gaa^{\theta/{(j+1)}+(1-\theta)/j}=
\theta f^{1/{(j+1)}}\circ\Gaa^{1/{(j+1)}}+(1-\theta)f^{1/j}\circ\Gaa^{1/j}.
$$
Then $f\in\cL C^{3/2,0}(\pd\Om_\Gaa)$ for $f^0=0$. Let $u^\la$ be the harmonic function
on $\Om^\la$ which is continuous
up to boundary and has
  boundary value $f^\la$. Thus $u\in\cL C^{3/2,0}(\ov\Om_\Gaa)$ and $u^0=0$ because $f^0=0$.
However, $|u^{1/j}|_{3/2}=1$, a contradiction.

 Let $u_\la^{1/j}$ be harmonic on $\Om^{1/j}$ such that
  $v_{j}^\la=u^{1/j}_\la\circ\Gaa^{1/j}(z)=\gaa^\la(z)-\Gaa^{1/j}(z)$.
We have $|u_\la^{1/j}|_{3/2}\leq \|T_j\|\cdot|\gaa^\la-\Gaa^{1/j}|_{3/2}\to0$ as
 $j\to\infty$ and $\la\to0$.
Hence $v_j^\la+\Gaa^{1/j}$ approach to $\Gaa^0$ in $\cL C^{3/2}$ norms as $\la$
and $1/j$ tend to zero.
Fix  a  $j$
such that $v_j^\la+\Gaa^{1/j}$ are embeddings for all $|\la-\la_0|$ sufficiently small.
Then $\Gaa_0^\la
=v_j^\la+\Gaa^{1/j}$ are extensions of $\gaa^\la$. Finally, $\Gaa_0^\la(z)$ is a real
analytic function on $\ov\Om\times I$
by the analyticity of solutions of Dirichlet problem with parameter.
\end{proof}

We now introduce spaces for exterior domains. Let $\Om'=\cc\setminus\ov\Om$.
Without loss of generality, we assume that $\Om$ be bounded and simply connected.
 Motivated
by the definition that
 a function $h(z)$ is harmonic at
 $\infty$ if $h(1/z)$ is harmonic at the origin, we define   inversions
\eq{invs}\iota_a(z)=\f{1}{z-a}+a, \quad
 \Om_a=\{a\}\cup\iota_a\Om', \quad
 \Om_b^\la=\{b_\la\}\cup \iota_{b_\la}(\Om^\la)'\eeq
 for    $a\in\Om$ and  $b_\la\in\Om^\la$.
 For a family of embeddings $\Gaa^\la$ from    $\ov{\Om'} $ onto $\ov{(\Om^\la)'}$,
 define
 \eq{invs+}
 \Gaa^\la_{b}=\iota_{b_\la}
 \circ\Gaa^\la,\quad
 \Gaa_{a,b}^\la=\iota_{b_\la}
 \circ\Gaa^\la\circ \iota_a,\quad
 \gaa^\la_{b}=\iota_{b_\la}
 \circ\gaa^\la,\quad
 \gaa_{a,b}^\la=\iota_{b_\la}
 \circ\gaa^\la\circ\iota_a.
 \eeq
Set $\Gaa_{a,b}^\la(a)=b_\la$. Then $\Gaa^\la_{a,b}$ is a fractional linear map
from $\Om_a$ onto $\Om_b^\la$.

 We denote
  $f\in\cL C^{k+\all}(\ov{\Om'})$   if $f\circ\iota_a$,
which is not defined at $a$,
 extends to an element in $\cL C^{k+\all}(\ov{\Om_a})$. Denote
$f=\{f^\la\}\in \cL C^{k+\all,j}(\ov{\Om'})$ (resp. $\cL B^{k+\beta,j}({\ov{\Om'}})$), if $f\circ\iota_a$
 extends to an element in $\cL C^{k+\all,j}(\ov{\Om_a})$ (resp. $\cL B^{k+\beta,j}({\ov{\Om_a}})$).
We emphasis that  as in \re{invs}-\re{invs+} we require  $a\in\Om$.
 The extended functions are still denoted by $f\circ\iota_a$. It is easy to verify
  that the definitions are independent of the choices of $a$.
Let $\Gaa^\la$ be a family
 of  $\cL C^1$  embeddings from $\ov{\Om'} $ onto $\ov{(\Om^\la)'}$.
  Denote
$f\in\cL C^{k+\all,j}(\ov{\Om_\Gaa'})$ (resp. $\cL B^{k+\all,j}({\ov{\Om_\Gaa'}})$),
 if $\{f^\la\circ\Gaa^\la \}\in\cL C^{k+\all,j}(\ov{\Om'})$
 (resp. $\cL B^{k+\beta,j}({\ov{\Om'}})$). The spaces for functions on boundaries of
 exterior domains
 will be the same as those for boundaries of interior domains.

To use the spaces $\cL C^{k+\all,j}(\ov{\Om_\Gaa'})$ and $\cL B^{k+\all,j}({\ov{\Om_\Gaa'}})$,
we will need   good control  of embeddings $\Gaa^\la$ at infinity.
Suppose that $b_\la$ and $d_\la$ are in $\Om^\la$ and $a,c$ are in $\Om$.
It is obvious  that   $\Gaa_{a,b}^\la=\iota_{b_\la}
 \circ\Gaa^\la\circ\iota_a$
 extends to a $\cL C^1$ embedding from $\ov{\Om_a}$ onto $\ov{\Om^\la_b}$ if and only if
 $\Gaa_{c,d}^\la$
 extends to a $\cL C^1$ embedding from $\ov{\Om'}$ onto $\ov{\Om^\la_d}$ for any
 $c\in\Om$ and $d_\la\in\Om^\la$.
 By $\{b_\la\}\in\cL C^j([0,1])$, we   mean
that $\la\to b_\la$ is of class $\cL C^j([0,1])$. Then, $\Gaa_{a,b}\in\cL C^{k+\all,j}(\ov{\Om_a})$
 if and only if $\Gaa_{c,d}\in\cL C^{k+\all,j}(\ov{\Om_d})$, provided $b$ and $d$ are in $\cL C^{j}([0,1])$.

To put the above definitions  in context, we restate \rl{chainp}\, (iii)
as follows:   The space
$\cL B^{k+\beta,j}(\ov{\Om_\Gaa'})$,
which is obviously dependent  of $\{(\Om^\la)'\}$ and $\Om'$,  is   independent
of embeddings $\Gaa^\la$ from $\ov{\Om'}$ onto $\ov{(\Om^\la)'}$, provided there
exists $\{b_\la\}\in\cL C^j([0,1])$
such that $\Gaa_{a,b}^\la$ extend to $\cL C^1$ embeddings from $\ov{\Om_a}$ onto
$\ov{\Om_b^\la}$ for
some $a\in\Om$ and
$\Gaa_{a,b}\in\cL B^{k+\all,j}(\ov{\Om_a})\cap\cL C^{1,0}(\ov{\Om_a})$.
Finally,
 we always assume that $\gaa^\la$ are the restrictions of $\Gaa^\la$ on $\pd\Om$, which preserve
 orientation.

\begin{prop}\label{c0pa+} Let $k\geq j$ and $k+1\geq l\geq0$. Let $\Om$ be a bounded and
simply connected domain with   $\pd\Om\in\cL C^{k+1+\all}$.
Let $\Gamma^{\lambda}$ map $\ov{\Om'}$ onto $\ov{(\Om^\la)'}$ for $0\leq\la\leq1$.
Let $b_\la\in\Om^\la$ satisfy $\{b_\la\}\in\cL C^j([0,1])$ and let $a\in\Om$. Suppose that
$\Gaa_{a,b}^\la$ extend to $\cL C^1$ embeddings
from $\ov{\Om_a}$ onto $\ov{\Om_b^\la}$ with $\Gaa_{a,b}\in\cL C^{1,0}(\ov{\Om_a})$.
\bppp
\item If $\Gaa_{b}\in\cL B^{l+\all,j}(\ov{\Om'})$ and
$f\in\cL B^{l+\all,j}(\pd\Om_\gaa)$,
then $\{ \cL C_-^\la f \}\in\cL B^{l+\all,j}(\ov{\Om_\Gaa'})$. The analogous assertion
 holds if $\cL C^{l+\all,j}$ substitutes for
$\cL B^{l+\all,j}$.
\item
If $\pd\Om\in\cL C^\om$, $\Gamma_{a,b}
\in\cL C^\om(\ov{\Om_a}\times[0,1])$ and $\{f\circ\Gaa^\la\circ\iota_a\}\in
\cL C^\om(\pd{\Om_a}\times[0,1])$, then
$\{{\cL C}_-^\la f\circ\Gaa^\la\circ\iota_a\}\in\cL C^\om(\ov{\Om_a}\times
[0,1])$.
\eppp
\end{prop}
\begin{proof} By our definition of orientations of boundaries,
$\iota_a$  reverses the orientations of $\pd\Om$ and $\pd\Om_a$ for $a\in\Om$.
Let $b_\la\in\Om^\la$,
$z^\la\in \cc\setminus
\ov{\Om^\la}$, and $z_\la\in\Om_b^\la$. Applying the inversion $\iota_{b_\la}$ to replace
$\zeta^\la-b_\la$ by $(\zeta_\la-b_\la)^{-1}$, we get
\gan
\cL C_-^\la f(z^\la)=-\f{z_\la-b_\la}{2\pi i}\int_{\pd\Om_b^\la}
\f{(\zeta_\la-b_\la)^{-1}f^\la(\iota_{b_\la}(\zeta_\la))}{\zeta_\la-z_\la}\, d\zeta_\la,\\
(\cL C_-^\la f)\circ\Gaa^\la\circ\iota_a(z)=-\f{\Gaa_{a,b}^\la(z)-b_\la}{2\pi i}
\int_{\pd\Om_a}\f{(\Gaa_{a,b}^\la(\zeta)-b_\la)^{-1}f^\la
\circ\Gaa^\la\circ\iota_a(\zeta)}
{\Gaa_{a,b}^\la(\zeta)-\Gaa_{a,b}^\la(z)}\, d\Gaa_{a,b}^\la(\zeta).
\end{gather*}
We know that $\{f^\la\circ\iota_b\circ\Gaa^\la_{a,b}\}=\{f^\la\circ\gaa^\la\circ\iota_a\}$
is in $\cL B^{k+\all,j}(\pd\Om_a)$ and $\cL C^{k+\all,j}(\pd\Om_a)$,
when $f$ is in $\cL B^{k+\all,j}(\pd\Om_\gaa)$
and $\cL C^{k+\all,j}(\pd\Om_\gaa)$, respectively. The lemma  follows from \rp{c0pa}.
\end{proof}

\begin{prop}\label{c0paU+++} Keep assumptions in \rpa{c0pa+}. Let $f\in\cL C_*^{0,j}(\pd\Om_{\gaa})$.
\bppp
\item Assume that $\int_{\pd\Om^\la}f^\la\,d\sigma^\la=0$.
If  $\Gaa_{b}\in\cL C_*^{1,j}(\ov{\Om'}_\Gaa)$, then $W_- f
\in\cL C_*^{0,j}(\ov{\Om'}_{\Gaa})$. Assume   that $\pd\Om\in\cL C^{k+1+\all}$,
 $\Gaa_b\in\cL B^{k+1+\all,j}(\ov{\Om'}_\Gaa)$
and $f\in\cL B^{k+\all,j}(\pd\Om_\gaa)$.
Then $W_-f\in\cL B^{k+1+\all,j}(\ov{\Om'}_\Gaa)$. The analogous assertion holds
if $\cL C$ substitutes for $\cL B$.
Assume further that $\pd\Om\in\cL C^\om$, $\{b_\la\}\in\cL C^\om$, $\Gaa_{a,b}
\in\cL C^\om(\ov{\Om_a}\times[0,1])$ and $\{f^\la\circ\gaa^\la\}\in\cL C^\om(\pd{\Om}\times[0,1])$.
Then
$\{(W_-^\la f)\circ\Gaa^\la\circ\iota_a\}\in\cL C^\om(\ov{\Om_a}\times[0,1])$.
\item
 If $\Gaa_{b}\in\cL B_*^{1+\all,j}(\ov{\Om'}_\Gaa)$, then $U_- f
\in\cL C_*^{0,j}(\ov{\Om'}_\Gaa)$.
\eppp
\end{prop}
\begin{proof}  Let $A$ be an orientation-preserving map from $\pd\hat\Om$
onto $\pd\Om$. Let $\hat\gaa(t)$ be a parameterization of $\pd\hat\Om$. Then
$\gaa(t)=A(\hat\gaa(t))$ is
a parameterization of $\pd\Om$. Assume that $dt$ agrees with the orientation
 of $\pd\Om$ and $A$ extends to a $\cL C^1$ map defined near $\pd\Om$. We have
\aln
d\sigma&=|\pd_tA(\hat\gaa(t))|\, dt
=|\pd_zA+\hat\gaa'(t)^{-1}\ov{\hat\gaa'(t)}\pdoz A|\, d\hat\sigma.
\end{align*}
Let $d\sigma_{b}^{\la}$ be the arc-length element on $\pd\Om_b^\la$.
Since $\iota_{b_\la}\colon z^\la\to z_\la$ reverses the orientations of $\pd\Om^\la$ and
 $\pd\Om_b^\la$, we obtain
\eq{ds2r}
d\sigma^\la=-\f{d\sigma_b^\la }{|\zeta_\la-b_\la|^2}
\eeq
on $\pd\Om_b^\la$ or   $\pd\Om^\la$ (via pull-back or push-forward). By \re{Ufz+}, a simple computation yields
\aln
 W_-^\la f(z^\la)&=\f{1}{\pi}\int_{\pd\Om^\la} f^\la(\zeta^\la)
\log|\zeta_\la-z_\la|\, d\sigma^\la\\
&\quad -\f{1}{\pi}\int_{\pd\Om^\la} f^\la(\zeta^\la)
\log|(\zeta_\la-b_\la)(z_\la-b_\la)|\, d\sigma^\la, \quad z_\la\neq b_\la.
\end{align*}
Since $\int_{\pd\Om^\la}f^\la\,d\sigma^\la=0$, we can remove $(z_\la-b_\la)$ and the restriction
$z_\la\neq b_\la$
from the last integral.
By \re{ds2r}, we get
\aln
W_-^\la f(z^\la)&=-\f{1}{\pi}\int_{\pd\Om_b^\la} f^\la(\iota_{b_\la}(\zeta_\la))
 \f{1}{|\zeta_\la-b_\la|^2} \log|\zeta_\la-z_\la|
 \,d\sigma_{b}^{\la}\\
&\quad +\f{1}{\pi}\int_{\pd\Om_b^\la} f^\la(\iota_{b_\la}(\zeta_\la))
  \f{1}{|\zeta_\la-b_\la|^2} \log|\zeta_\la-b_\la|
 \,d\sigma_{b}^{\la},\quad
z_\la\in\Om_b^\la.
\end{align*}
 Let
  $\ta_{b,\zeta}^\la$ be the unit tangent vector of $\pd\Om^\la_b$ at $\zeta_\la$. Fixing $z\in\Om$, we have $d\arg(z-\zeta^\la)=(\pd_{\ta^\la_{b,\zeta}}\arg(\zeta_\la-z_\la)-\pd_{\ta^\la_{b,\zeta}}
  \arg(\zeta_\la-b_\la))
\, d\sigma_b^\la(\zeta_\la)$ and
\aln
U_-^\la f(z^\la)&=-\f{1}{\pi}\int_{\pd\Om_b^\la}f^\la(\iota_{b_\la}(\zeta_\la))
\pd_{\tau_{b,\zeta}^\la}\arg(\zeta_\la-z_\la)\, d\sigma_b^\la(\zeta_\la)\\
&\quad+\f{1}{\pi}\int_{\pd\Om_b^\la}f^\la(\iota_{b_\la}(\zeta_\la))
\pd_{\tau_{b,\zeta}^\la}\arg(\zeta_\la-b_\la)\, d\sigma_b^\la(\zeta_\la),\quad z_\la\in\Om_b^\la.
\end{align*}
The  assertions follow from
\rp{c0paU} and the last two formulae.
\end{proof}

\setcounter{thm}{0}\setcounter{equation}{0}
\section{Main results and proofs}
\label{sec9}
In this section, we   first prove the real analyticity of solutions to real analytic
integral equations arising from the Dirichlet and Neumann problems. We   then
  collect results from previous sections to formulate the solutions
of Dirichlet and Neumann problems with parameter. Finally, we     prove \rt{nrefl}.
\pr{kl3an} Let $\Om$ be a bounded domain with $\pd\Om\in\cL C^\om$.
 Let $\gaa^\la$ embed  $\pd\Om$
onto $\pd\Om^\lambda$ with $\gaa\in\cL C^\om(\pd\Om\times[0,1])$.
Let $\cL L$ be one of $K,-K,K^*$, and $-K^*$. Let $\psi\in\cL C^\om(\pd\Om\times[0,1])$.
Suppose that $\var^\la\in L^1(\pd\Om^\la)$ satisfy
\eq{vlllan}
\var^\la+\cL L^{\la}\var^{\la}=\psi^\la, \quad\{\jq{\var^\la,\ell_j^\la}\}\in\cL C^\om,
\quad  1\leq j\leq n.
\eeq
Then $\var\in\cL C^\om(\pd\Om\times[0,1])$. Furthermore, the functions $\phi_0,\ldots,\phi_m$
in \rpa{kern}
are   in $\cL C^\om(\pd\Om\times[0,1])$.
\epr
\begin{proof}  We already know that $\var,\phi_i$ are of class $\cL C^\infty$.
We apply Cauchy majorant methods to estimate the growth of their Taylor coefficients.
By Taylor's theorem,
 a function $f$ on $\pd\Om\times[0,1]$ is real analytic if and
 only if $$
 \max_{t,\la}\bigl|\pd_t^i\pd_\la^jf(\hat\gaa(t),\la)\bigr|\leq Ci!j!\rho^{i+j},
 $$
 where $\hat\gaa$ is a real analytic parameterization of $\pd\Om$ and $C,\rho$ are constants.
We first need uniform  bounds for solutions operators in sup-norms.
Let $\{\ell_1^\la,\ell_2^\la,\ldots,\ell_n^\la\}$ be the basis of $\ker (I+\cL L^\la)$ described
after the proof of \rp{kern}.
By \rl{kl} (i), we know that   $\ell_1,\ldots,\ell_n$ are in $\cL C^{0,0}(\pd\Om_\gaa)$.
Then $\cL L_0$ sends $\cL C^{0,0}(\pd\Om_\gaa)$ into $(\cL C^0([0,1]))^n$, where
$$ 
\cL L_0^\la\var=(\jq{\var^\la,\ell_1^\la},\jq{\var^\la,\ell_2^\la},\ldots,\jq{\var^\la,\ell_n^\la}).
$$ 
Consider  bounded linear maps
\gan
(I+\cL L,\cL L_0)\colon \cL C^{0,0}(\pd\Om_\gaa)\to
(C^{0,0}(\pd\Om_\gaa)\cap\ker (I+\cL L^*)^\perp)\times(\cL C^0([0,1]))^n=X_{\cL L};\\
(I+\cL L^*,\cL L_0)\colon\cL C^{0,0}(\pd\Om_\gaa)\to X_{\cL L^*},\quad
\text{$\cL L=\cL K$ or $-\cL K$}.
\end{gather*}
It is clear that $(I+\cL L,\cL L_0)$ is injective. By \rp{kern} (iii),
the second map is injective too for both cases.
 By \rp{kern} (i) and \rl{kl} (i), $I+\cL L$
maps $\cL C^{0,0}(\pd\Om_\gaa)$ onto $\cL C^{0,0}\cap(\ker(I+\cL
L^*))^\perp$. Since $\ell_1^\la,\ldots, \ell_n^\la$ are linearly
independent for each $\la$, then
$(\jq{\ell_i^\la,\ell_j^\la})_{1\leq i,j\leq n}$ are invertible.
Since $\ell_i$ are in $\cL C^{0,0}(\pd\Om_\gaa)$,   given
$c  \in(\cL C^{0}([0,1]))^n$ we can find $\tilde c\in(\cL C^{0}([0,1]))^n$
 such that $\jq{\sum_j
\tilde c_{j}(\la)\ell_j^\la,\ell_i^\la}=c_i(\la)$. This shows that $(I+\cL L,\cL
L_0)$ is surjective. That $(I+\cL L^*,\cL L_0)$ is surjective for
$\cL L=\cL K$ or $-\cL K$ follows  from
$\int_{\pd\Om^\la}e_i\phi_j^{\la}\, d\sigma^\la=\delta_{ij}$ for
$1\leq i,j\leq m$, and $\int_{\pd\Om^\la}e_0\phi_0^{\la}\,
d\sigma^\la=1$.  By the open mapping theorem, we have \ga \label{vcss}
|\var|_{0,0}\leq C_*(|(I+\cL L^*)\var|_{0,0}+|\cL L_0\var|_{0}), \quad \text{$\cL L=\cL K$
 or $-\cL K$};\\
\label{vssl}
|\var|_{0,0}\leq C_*(|(I+\cL L)\var|_{0,0}+|\cL L_0\var|_{0}), \quad \text{$\cL L=\cL K,
 -\cL K,\cL K^*$, or $-\cL K$}.
\end{gather}
Here
$\var$ are in $\cL C^{0,0}(\pd\Om_\gaa)$ and $C_*$
is independent of $\var$.

(i).  We first consider the case where $\cL L=\cL K$ or $-\cL K$.  We express \re{vlllan} as
\ga\label{vazl}
\var(z,\la)+\int_{\pd\Om}\var(\zeta,\la)L(z,\zeta,\la)\, d\sigma(\zeta)=\psi_0(z,\la),\\
\int_{\pd\Om}\var(\zeta,\la)\ell_\alpha^\la a(\zeta,\la)\, d\sigma(\zeta)=\psi_\all^\la,
\quad \alpha=1,\ldots, n.
\label{vaz1+}\end{gather}
Note that $L(z,\zeta,\la)$ is real analytic on $\pd\Om\times\pd\Om\times[0,1]$
and $a(\zeta,\la)=\pd_{\ta^\la}\gaa^\la(\zeta)$ is real analytic on $\pd\Om\times[0,1]$.
We know that $\ell_i$ are locally constants. However,
we want to reason in such a way that the proof is valid
whenever $\ell_i^\la(\zeta)$ are real analytic in $\la$ and $\zeta$. Thus, the proof applies to
  $\cL L=\cL K^*$ or $-\cL K^*$ after we prove (ii).
 Differentiating \re{vazl}-\re{vaz1+} yields
\al\label{vazl2}
\pd_\la^k\var(z,\la)&+\cL L^{\la}\pd_\la^k\var^\la=\pd_\la^k\psi_0(z,\la)\\
\nonumber &\hspace{5em}-\sum_{l=0}^{k-1}\binom{k}{l}
\int_{\pd\Om}\pd_\la^l\var(\zeta,\la)\pd_\la^{k-l}
L(z,\zeta,\la)\, d\sigma(\zeta),\\
\label{vazl2+}
\cL L_0\pd_\la^k\var^\la&=\pd_\la^k\psi (\la)
-\sum_{l=0}^{k-1}\binom{k}{l}\int_{\pd\Om}\pd_\la^l\var(\zeta,\la)  \pd_\la^{k-l}
(\ell ^\la(\zeta) a(\zeta,\la))\, d\sigma(\zeta).
\end{align}
 Set
$
A_{k}=\f{1}{k!}\max_{\zeta,\la}|\pd_\la^k\var(\zeta,\la)|$ and
$$
 a_k=\f{1}{k!}\max_{\zeta,z,\la,\alpha}\Bigl\{|\pd_\la^kL(\zeta,z,\la)|,
|\pd_\la^k\psi_0(\zeta, \la)|,|\pd_\la^k\psi_\all(\la)|,
|\pd_\la^k(\ell_\all^\la a(\zeta, \la))|\Bigr\}.
$$
We have $|\ell_\all|_{L^2}^{-2}|\ell_\all|\leq C_1$. Denote by $|\pd\Om|$ the length of $\pd\Om$.
 Then we obtain from \re{vazl2+}, \re{vssl} and \re{vazl2}
\ga
\nonumber
\f{1}{k!}|\cL L_0^\la\pd_\la^k\var^\la|\leq  a_k+|\pd\Om|\sum_{l=0}^{k-1}A_la_{k-l},\quad
A_k\leq 2C_*\Bigl\{a_k+|\pd\Om|\sum_{l=0}^{k-1}  A_la_{k-l}\Bigr\}.
\end{gather}
Denote $\sum A_I w^I\prec\sum b_Iw^I$ if $A_I\leq b_I$ for $|I|\geq0$. The above implies that
$$
\sum A_kw^k \prec 2C_*\sum a_kw^k+2C_*|\pd\Om| w\sum A_kw^k\sum
a_{k+1}w^k.
$$
Therefore, $\sum A_kw^k$   converges near the origin. Set
$
B_{kj}=\f{1}{k!j!}\max_{t,\la}|\pd_t^j\pd_\la^k\var(\hat\gaa(t),\la)|$ and
$$ b_{kj}=\f{1}{k!j!}
\max_{\zeta,t,\la}\Bigl(|\pd_t^j\pd_\la^kL(\hat\gaa(t),\zeta,\la)|,
|\pd_t^j\pd_\la^k\psi_0(\hat\gaa(t),\la)|\Bigr).
$$
Taking $\pd_t^j$ directly onto the real analytic
 kernel $\pd_\la^kL(\hat\gaa(t),\zeta,
\la)$ in \re{vazl2},
 we get
\gan
\pd_t^j\pd_\la^k\var(\hat\gaa(t),\la)=\pd_t^\la\pd_\la^k\psi_0(\hat\gaa(t),\la)
-\sum_{l=0}^{k}\binom{k}{l}
\int_{\pd\Om}\pd_\la^l\var(\zeta,\la)\pd_t^j\pd_\la^{k-l}
L(\hat\gaa(t),\zeta,\la)\, d\sigma(\zeta),
\\
B_{kj}\leq b_{kj}+|\pd\Om|\sum_{l=0}^k A_lb_{(k-l)j},\qquad k,j\geq0,\\
\sum B_{kj}w_1^kw_2^j\prec\sum b_{kj}w_1^kw_2^j + |\pd\Om|\sum
A_{k}w_1^k\sum b_{kj}w_1^kw_2^j.
\end{gather*}
Obviously,  $\sum B_{kj}t^j\la^k$   converges near $(t,\la)=0$.

(ii). We still consider $L=K$ or $-K$. The elements in the base $\{\phi_i^\la\}$
 of $\ker(I+\cL L^*)$ are not constant, so
we need to establish their analyticity first.
 Recall that
\gan
\phi_i^\la+\cL L^{\la*}
\phi_i^\la=0, \quad \int_{\pd\Om}\phi_i^\la\ell_j^\la a(\la,\zeta)\, d\sigma=\delta_{ij},
 \quad 1\leq i,j
\leq n.
\end{gather*}
We write   both in   row vectors and get
\ga
\label{dkin2+}
(I+\cL L^{\la*})\pd_\la^k\phi(z,\la)= -\sum_{l=0}^{k-1}\binom{k}{l}
\int_{\pd\Om}\pd_\la^l\phi(\zeta,\la)\pd_\la^{k-l}
L(\zeta,z,\la)\, d\sigma(\zeta),\\
\cL L_0^\la\pd_\la^k\phi^\la=\pd_\la^k (1,\ldots,1)
-\sum_{l=0}^{k-1}\binom{k}{l}\int_{\pd\Om}\pd_\la^l\phi(\zeta,\la)
 \pd_\la^{k-l}(\ell ^\la a(\zeta,\la))\, d\sigma(\zeta).
\nonumber\end{gather}
We   use \re{vcss} instead of \re{vssl} and get,
for
$A_{k}=\f{1}{k!}\max_{\zeta,1\leq\la\leq n}|\pd_\la^k\phi_\all(\zeta,\la)|$,
\gan
\f{1}{k!}|\cL L_0^\la\pd_\la^k\phi^\la|\leq  1+|\pd\Om|\sum_{l=0}^{k-1}A_la_{k-l},\quad
A_k\leq 2C_*\Bigl\{1+|\pd\Om|\sum_{l=0}^{k-1}  A_la_{k-l}\Bigr\}.
\end{gather*}
Therefore, $
\sum A_kw^k \prec 2C_*+2C_*|\pd\Om| w\sum A_kw^k\sum a_{k+1}w^k$ and  $
\sum A_kw^k$   converges near the origin. Next, we apply $\pd_t^j$ to \re{dkin2+} and get
$$
\pd_t^j\pd_\la^k\phi(\hat\gaa(t),\la)= -\sum_{l=0}^{k}\binom{k}{l}
\int_{\pd\Om}\pd_\la^l\phi(\zeta,\la)\pd_t^j\pd_\la^{k-l}
L(\zeta,\hat\gaa(t),\la)\, d\sigma(\zeta).
$$
As before, we obtain real analyticity of $\phi(\hat\gaa(t),\la)$.

With the real analyticity of $\phi_i$, the proof in (i) is valid for $\cL L=\cL K^*$ or
$-\cL K^*$.
\end{proof}
The   Dirichlet problem for exterior domains with parameter is
$$
 \Delta u^\la=0 \quad \text{on $(\Om^\lambda)'$}, \qquad
      u^\la=f^\la \quad \text{on $\pd\Om^\lambda$}.
$$
To ensure that the solutions are unique, we require that $u^\la$ be harmonic
at $\infty$, i.e., that $u^\la(1/z)$ is harmonic in a neighborhood of $0$.
The Neumann problem for exterior domains with parameter is
$$
 \Delta v^\la=0 \quad \text{on $(\Om^\lambda)'$}, \qquad
     \pd_{\nu^\la} v^\la=g^\la \quad \text{on $\pd\Om^\lambda$}.
$$
Here $\nu^\la$  is
the unit outer normal vector of $\pd\Om^\la$. Again, we require that $v^\la$
be harmonic at $\infty$. For the existence and uniqueness
 of solutions $v^\la$, we impose conditions
\eq{npb++}
\int_{\gaa_i^\la}g^\la\, d\sigma^\la=0,\qquad\int_{\gaa_i^\la}v^\la\, d\sigma^\la=0,
\quad 0\leq i\leq m.
\eeq
The $v^\la$ which satisfy conditions \re{npb+} or \re{npb++} are called normalized solutions.
By Hopf's lemma if $u$ is harmonic on $\Om$
and continuous up to the boundary with $\pd\Om\in\cL C^{1+\all}$,
then $\pd_\nu u$ determines $u$ up to a constant. In fact, one can locally reduce to
the case where $\Om$ is a unit disc by Kellogg's theorem; see also~\ci{Miseze}, p.~7.
Thus, the normalized solutions are
unique.

We now summarize the solutions to the Dirichlet and Neumann problems as follows.
Recall that function spaces for interior domains are defined in section~\ref{sec2}
and   function spaces for exterior domains are defined in section~\ref{sec8}.
The reader is referred to \rl{chainp} for independence of spaces
$\cL B^{k+\beta,j}(\pd\Om_\gaa)$ and $\cL B^{k+\beta,j}(\pd\Om_\Gaa)$
on $\gaa$ and $\Gaa$ for $k\geq j$, respectively. Recall that \rl{spac} shows
 the existence of extensions of   $\gaa^\la$
to   $\Gaa^\la$.
\th{dnpb} Let $0\leq j\leq k$, $0<\all<1$, and $j\leq l\leq k+1$.
 Let $\Om$ be a connected bounded domain in $\cc$ with $\pd \Om\in\cL C^{k+1+\all}$.
Let $\Gaa^\la$ embed  $\ov\Om$ onto $\overline{\Omega^\lambda}$
with $\Gaa\in\cL B^{k+1+\all,j}(\ov\Om)$ for   interior Dirichlet and Neumann problems.
Let $\Gaa^\la$ embed $\ov{\Om'}$ onto $\ov{(\Om^\la)'}$ such that $\iota_{b^\la}
\circ\Gaa^\la\circ\iota_a$
extends to $\cL C^1$ embeddings from $\ov{\Om_a}$ onto $\ov{\Om_b^\la}$ with
$\Gaa_{b}\in\cL B^{k+1+\all,j}(\ov{\Om_\Gaa'})$ for
exterior Dirichlet and Neumann problems. Here $a\in\Om$, $b_\la\in\Om^\la$
and $\{b_\la\}\in\cL C^{j}([0,1])$. Let $\gaa^\la$ be
the restriction of $\Gaa^\la$ on $\pd\Om$.  Suppose that $\{f^\la\}\in\cL C_*^{0,j}
(\pd\Om_\gaa)$. \bppp
\item
{\bf (Interior Dirichlet problem.)}
There exists a unique harmonic function
$u^\la$ on $\Om^\la$ such
that $u\in\cL C_*^{0,j}(\ov \Om_\Gaa)$  and $u^\la=f^\la$
on $\pd\Om^\la$. Moreover,
\ga\label{idpf}
u^\la=U_+^\la{\var}+\sum_{i,j=1}^mc_i^\la \mu_{ij}^\la W^\la_+\phi_j,\\
\nonumber 
\var^\la+\cL K^\la\var^\la=g^\la,\quad \var^\la\perp\ker(I+\cL K^\la),
\quad g^\la= f^\la-\sum_{j=1}^m c_i^\la e_i,\\
\nonumber 
(W^\la{\phi_i}|_{\gaa_j^\la})_{1\leq i,j\leq m}=(\mu_{ij}^\la)^{-1},
\quad c_i^\la=\int_{\pd\Om^\la} f^\la\phi_i^\la\, d\sigma^\la.
\end{gather}
\item{\bf (Exterior Dirichlet problem.)} Assume that $\Om^\la$ are simply connected.
There exists a  unique harmonic function
$u^\la$ on $(\Om^\la)'\cup\{\infty\}$ such
that  $u\in\cL C_*^{0,j}(\ov {\Om'_\Gaa})$
and  $u^\la=f^\la$
on $\pd\Om^\la$. Moreover,
\gan
u^\la=U_-^\la{\var}+ \int_{\pd\Om^\la} f^\la \phi_0^\la\, d\sigma^\la,\\
\var^\la-\cL K^\la\var^\la=g^\la,\quad \var^\la\perp\ker(I-\cL K^\la), \quad g^\la=
 f^\la- \int_{\pd\Om^\la}
f^\la\phi_0^\la \, d\sigma^\la.
\end{gather*}
\item
{\bf (Interior Neumann problem.)}
That  $\int_{\pd\Om^\la} f^\la\, d\sigma^\la=0$ are
the necessary and sufficient conditions for the existence of    functions $u^\la$ which
 are harmonic
on $(\Om^\la)'\cup\{\infty\}$ and satisfy
$u\in\cL C_*^{0,j}(\ov\Om_\Gaa)$ and $\pd_{\nu^\la}u^\la=f^\la$. The  normalized
 solutions are given by
$$
u^\la=W_+^\la{\var},\quad \var^\la-\cL K^{\la*}\var^\la=f^\la,\quad \var^\la\perp
\ker(I-\cL K^{\la*}).
$$
\item
{\bf (Exterior Neumann problem.)} Assume that $\Om^\la$ are simply connected.
That     $\int_{\gaa_j^\la}f^\la\, d\sigma^\la=0$
for all $j\geq0$ are
the necessary and sufficient conditions  for the existence of  functions $u^\la$ which are harmonic
on $(\Om^\la)'\cup\{\infty\}$ and satisfy $u\in\cL C_*^{0,j}(\ov{\Om'_\Gaa})$ and
 $\pd_{\nu^\la}u^\la=f^\la$.   The  normalized  solutions
are given by
$$
u^\la=W_-^\la{\var},\quad
\var^\la+\cL K^{\la*}\var^\la=f^\la,\quad \var^\la\perp\ker(I+\cL K^{\la*}).$$
\item {\bf (Regularity.)}
If $f\in\cL B^{l+\all,j}(\pd\Om_\gaa)$,  then $u\in\cL B^{l+\all,j}(\ov\Om_\Gaa)$
for {\rm($i$)} and $u\in\cL B^{l+\all,j}(\ov{\Om'}_{\Gaa})$
for {\rm($ii$)}; if $f\in\cL B^{k+\all,j}(\pd\Om_\gaa)$ then $u\in\cL B^{k+1+\all,j}
(\ov\Om_\Gaa)$ for {\rm($iii$)}
and $u\in\cL B^{k+1+\all,j}(\ov{\Om'}_{\Gaa})$ for {\rm($iv$)}.
Assume further that
$\Gaa\in\cL C^{k+1+\all,j}(\ov\Om)$ and $\Gaa_{b}\in\cL C^{k+1+\all,j}(\ov\Om)$.
If $f\in\cL C^{l+\all,j}(\pd\Om_\gaa)$,  then $u\in\cL C^{l+\all,j}(\ov\Om_\Gaa)$
for {\rm($i$)} and $u\in\cL C^{l+\all,j}(\ov{\Om'}_{\Gaa})$
for {\rm($ii$)}; if $f\in\cL C^{k+\all,j}(\pd\Om_\gaa)$ then $u\in\cL C^{k+1+\all,j}
(\ov\Om_\Gaa)$ for {\rm($iii$)}
and $u\in\cL C^{k+1+\all,j}(\ov{\Om'}_{\Gaa})$ for {\rm($iv$)}.    Assume further that
$\pd\Om\in\cL C^\om$,  $\Gaa\in\cL C^\om(\ov\Om\times[0,1])$,  $\Gaa_{a,b}\in
\cL C^\om(\ov{\Om_a}\times
[0,1])$, $\{b^\la\}\in\cL C^\om([0,1])$, and $ f\circ\gaa
\in\cL C^\om(\pd\Om\times[0,1])$.
Then $u^\la\circ\Gaa^\la(z)$ is in $\cL C^\om(\ov\Om\times[0,1])$ for {\rm($i$)} and {\rm($iii$)}, and
$u^\la\circ\Gaa^\la\circ\iota_a(z)$ is in $\cL C^\om(\ov{\Om_a}\times[0,1])$ for {\rm($ii$)} and {\rm($iv$)}.
\eppp
\end{thm}
\begin{proof}   For the smoothness in parameter, we need to compute
the coefficients in the solution formulae.
We recall  results from \rp{kern}.
We have $e_i=1$ on $\gaa_i$   and $e_i=0$ on $\pd\Om^\la\setminus\gaa^\la_i$
for $i>0$, and   $e_0=1$ on $\pd\Om^\la$. Also
 $(\int_{\gaa_j^\la}\phi_i^\la\, d\sigma^\la)_{1\leq i,j\leq m}=I$,
 $\int_{\pd\Om^\la}\phi_0^\la e_0\, d\sigma^\la=1$ and $\phi_0=0$
on $\gaa_i^\la$ for $i>0$.
We also know that, on $\pd\Om^\la$, $W^\la_+\phi_0$ is constant and $W^\la_-\phi_i$
 are locally constant for $i>0$. On $\pd\Om^\la$ and for
$i>0$, we have
$$
W^\la_-\phi_i=\sum_{j>0}\nu_{ij}^\la e_j, \quad \nu_{ij}^\la
=W^\la_-\phi_i|_{\gaa_j^\la}, \quad\det(\nu_{ij}^\la )_{1\leq i,j\leq m}\neq0.
$$
(The latter needs $m>0$.) Thus   for $j>0$ we have
 $e_j=\sum_{i=1}^m \mu_{ji}W^\la_-\phi_i$. By \rp{kl3} a), we know that
 $\phi_0,\phi_1,\ldots, \phi_m$ are in $\cL B^{k+\all,j}
(\pd\Om_\gaa)$.  Thus,
  $\nu_{il}$ and   $\mu_{il}$ are in $\cL C^j([0,1])$.   Let
   $c_i^\la=\int_{\pd\Om^\la} f^\la\phi_j^\la\, d\sigma^\la$. Then $c_i\in\cL C^j([0,1])$
   and
$$
f^\la=g^\la+c_1^\la e_1+\cdots+c_m^\la e_m, \  g^\la\perp\ker(I+\cL K^{\la*});
\quad  f^\la=g^\la+c_0^\la,
 \  g^\la\perp \ker(I-\cL K^{\la*}).
$$
It is clear that  $g_i\in\cL C_*^{0,j}(\pd\Om_\gaa)$.
By \rp{kern} (i)  and \rl{kl2} c), we get $\var\in\cL C_*^{0,j}(\pd\Om_\gaa)$ for (i)-(iv).

For (i) and (ii) with $f\in\cL B^{l+\all,j}(\pd\Om_\gaa)$ and $l\leq k+1$, we still   have
 $g\in\cL B^{l+\all,j}(\pd\Om_\gaa)$ as $f-g\in\cL C^{\infty,j}(\pd\Om_\gaa)$.
 Thus, $\var\in\cL B^{l+\all,j}(\pd\Om_\gaa)$
  by \rp{kl3} c). Hence,
    $U_+\var=2\RE\cL C\var\in\cL B^{l+\all,j}(\ov{\Om}_\Gaa)$ by \rp{c0pa} and $U_-\var=
  2\RE\cL C\var  \in\cL B^{l+\all,j}(\ov{\Om'}_\Gaa)$
  by \rp{c0pa+}. Also, $W_+\phi_i\in\cL B^{k+1+\all,j}(\ov{\Om}_\Gaa)$ by \rp{c0pa} and
  $W_-\phi_i\in\cL B^{k+1+\all,j}(\ov{\Om'}_\Gaa)$ by \rp{c0paU+++}. The coefficients
  $c_i,\mu_{il}$ in \re{idpf} are in $\cL C^{\infty,j}$. We conclude that
   $u\in\cL B^{l+\all,j}(\ov{\Om}_\Gaa)$ for (i) and   $u\in\cL B^{l+\all,j}(\ov{\Om'}_\Gaa)$
   for (ii).

 For (iii) and (iv) with $f\in\cL B^{k+\all,j}$, we get
 $\var\in\cL B^{k+\all,j}(\pd\Om_\gaa)$ by \rp{kl3} b). Hence, $W_+\var\in
 \cL B^{k+1+\all,j}(\ov{\Om}_\Gaa)$ by \rp{c0pa} and
  $W_-\var\in\cL B^{k+1+\all,j}(\ov{\Om'}_\Gaa)$ by \rp{c0paU+++}.

Finally, the real analytic results follow from \rp{kl3an}, \rp{c0paU+++}, \rp{c0pa},
 \rp{c0pa+}, and
the solutions formulae of the Dirichlet and Neumann problems.
\end{proof}

\co{countex} Let $k\geq0 $ be an integer. Let $0<\beta<\all<1$.
 Let $\Om$ be a bounded domain with $\pd\Om\in\cL C^{k+1+\all}$.
Let $f\in\cL C^{k+1+\beta}(\pd\Om)\setminus\cL C^{k+1+\all}(\pd\Om)$.
Then $Wf $
defines two harmonic functions on $\Om$ and $\Om'$,
  which have the same boundary value.
 $Wf|_{\pd\Om}$ is in $\cL C^{k+1+\beta}$, but not in $\cL C^{k+1+\all}$.
 Moreover, $Wf\in\cL C^{1-\e}(\cc)$ for any $\e>0$.
\eco
As observed in~\ci{BGR},  if the above   $Wf$ is in $\cL C^1(\cc)$ then \re{dnwf}
 implies that $f$ and
 $Wf$
are  zero.  It is trivial that if a continuous function which
is holomorphic on both sides of a real curve in the complex plane,
 the function is holomorphic near the curve.
The reader is referred to~\ci{BGR} where
regularities   for functions for two-sided almost complex structures  are in contrast to
  \nrc{countex}.

As a consequence of \rt{dnpb}, we have the following version of Kellogg's Riemann
mapping theorem with parameter.
\co{krm} Let $j,k$  be non negative integers or $\infty$ satisfying   $0\leq j\leq k$.
 Let $0<\all<1$.
 Let $\Om$ be a simply connected bounded domain in $\cc$ with $\pd \Om\in\cL C^{k+1+\all}$
and let $\Gaa^\la$ embed  $\ov\Om$ onto $\overline{\Omega^\lambda}$ and
satisfy  $\Gaa\in\cL C^{k+1+\all,j}(\ov{\Om})$ {\rm(}resp.
$\cL B^{k+1+\all,j}(\ov{\Om})${\rm)}.  There exist Riemann mappings $R^\la$  from
$\Om^\la$ onto $\D$ such that $\{R^\la\circ\Gaa^\la\}\in\cL C^{k+1+\all,j}(\ov\Om)$
{\rm(}resp.
$\cL B^{k+1+\all,j}(\ov{\Om})${\rm)}. Assume
further that $\pd\Om\in\cL C^\om$ and  $\Gaa\in\cL C^\om(\ov\Om\times[0,1])$.
Then the function $R^\la\circ\Gaa^\la(z)$ is real analytic on $\ov
\Om\times[0,1]$.
\eco
\begin{proof} The proof is standard for the non-parameter case. Since we need it for
next proof, we recall the construction.
Fix $a\in\Om$ and let $a^\la=\Gaa^\la(a)$. Let $u^\la(z^\la)$ be the harmonic function
 on $\Om^\la$ whose
boundary value is $-\log|z^\la-a^\la|$. Let $v^\la$ be the harmonic conjugate of
$u^\la$ on $\Om^\la$ with $v^\la(a^\la)=0$.
Then $z^\la\to
(z^\la-a^\la)e^{u^\la(z^\la)+iv^\la(z^\la)}$ is a Riemann mapping $R^\la$ sending
 $\Om^\la$ onto
$\D$. By \rt{dnpb}, we know that $u\in\cL C^{k+1+\all,j}(\ov\Om)$. Also $$v^\la(z^\la)=\int_{a^\la}^{z^\la}
 \Bigl(-\pd_{y^\la}u^\la(z^\la)\, dx^\la+\pd_{x^\la}u^\la(z^\la)\, dy^\la\Bigr),$$
where the path of integration is   any $\cL C^1$ curve of the form $(x^\la,y^\la)
=\Gaa^\la(\rho(t))$ with $\rho(0)=a$
and $\rho(1)=z$. Using the integral formula we can verify that $ v^\la
\in\cL C_*^{0,j}(\ov\Om)$.
Then $\pd_{x^\la}v^\la=-\pd_{y^\la}u^\la$ and $\pd_{y^\la}v^\la=\pd_{x^\la}y^\la$
imply that  $v$ is in
$\cL C^{k+1+\all,j}(\ov\Om)$. The same argument is valid for the real analytic case.
\end{proof}
We now turn to the proof of \rt{nrefl}, for which we need a third-order invariant.
\le{rmp} Let $\Om$ be a bounded simply connected domain with $\pd\Om\in\cL C^{2+\all}$.
 Assume that at $1$, $\pd\Om$
and $\pd\D$ are tangent and have the same exterior normal vector.
There exists a unique biholomorphism $S$ from $\ov\Om$ onto $\ov\D$
such that $S(1)=1$, $S'(1)=1$ and $S''(1)\in \rr$. Let $R$ be a Riemann
mapping from $\ov\Om$ onto $\pd\D$ with $R(1)=1$. Then
$S''(1)=R'(1)^{-1}\RE R''(1)+1-R'(1)$.
Assume further that $\pd\Om\in\cL C^{3+\all}$.
 Then
 at $1$
\aln
S'''&=(R')^{-1}\{R'''+3(1-R')R''+\f{3}{2}(1-R')^2R'\}\\
&\quad +\f{3}{2}(R')^{-2}(\IM R'')^2-3i\{(R')^{-1}\RE R''+(1-R')\}(R')^{-1}\IM R''.
\end{align*}
\ele
\begin{proof} Let $R$ be a Riemann mapping from $\ov\Om$ onto $\ov\D$ with $R(1)=1$.
The fractional
linear transformations that preserve $\D$ and $1$ are of the form
$$
L_a(z)=\f{1-\ov a}{1-a}\cdot\frac{z-a}{1-\ov az}, \quad |a|<1.
$$
 We have
\ga\nonumber
(L_a\circ R)'=\f{1-\ov a}{1-a}\cdot\f{1-|a|^2}{(1-\ov aR)^2}R',\\
(L_a\circ R)''=\f{1-\ov a}{1-a}\cdot\f{1-|a|^2}{(1-\ov aR)^2}\left(R''+
\f{2\ov a(R')^2}{1-\ov aR}\right).
\label{2der}
\end{gather}
Note that $R'(1)>0$. We have
 $R_1'(1)=1$ for $R_1=L_a\circ R$ with
 $$
a=\f{1-R'(1)}{1+R'(1)}.
 $$
  We   further
 determine $L_b$ under the restriction
  $1-|b|^2=|1-b|^2$, i.e. $b=\cos\theta(\cos\theta+i\sin\theta)$
with $\theta\in(-\pi/2,0)\cup(0,\pi/2]$. Thus we still have
  $(L_b\circ R_1)'(1)=1$. Then $R_1(1)=R_1'(1)=1$ imply that
$$
(L_b\circ R_1)''(1)=R_1''(1)-2i\cot \theta.
$$
Hence, there is a unique $\theta\in(-\pi/2,0)\cup(0,\pi/2]$ such that $(L_b\circ
 R_1)''(1)\in\rr$. At $1$,
$$
\f{2\ov aR'}{1-\ov aR}=1-R', \quad
\f{2\ov b }{1-\ov b}=-i\IM R_1''.
$$
Therefore,  $S$ equals $L_b\circ L_a\circ R$. By \re{2der}, we get at $1$
\aln
S''&=(L_b\circ R_1)''=\RE R_1''=(R')^{-1}\{\RE R''+(1-R')R'\}.
\end{align*}
Also $\IM R_1''(1)=R'(1)^{-1}\IM R''(1)$. Differentiating \re{2der}, we obtain at $1$
\aln
R_1'''&=(R')^{-1}\{R'''+3(1-R')R''+\f{3}{2}(1-R')^2R'\},\\
S'''&=(L_b\circ R_1)'''=R_1'''-2iR_1''\IM R_1''-\yt(\IM R_1'')^2-i\RE R_1''\IM
R_1''.
\end{align*}
Expressing $R_1''(1)$ and $ R_1'''(1)$ in $R'(1), R''(1)$ and $ R'''(1)$ yields the identity.
\end{proof}

\noindent
{\bf Proof of \rt{nrefl}.}
We need to find a family of embeddings $\Gaa^\la$ from $\ov\D$ onto $\ov{\Om^\la}$ satisfying
the following: {(a)} $\Gaa$ is in
$\cL C^{\infty}
(\ov\D\times[0,1])$ and real analytic at $(1,0)\in\ov\D\times[0,1]$,   {
(b)} for any family of Riemann mappings
from $\ov{\Om^\la}$ onto $\ov\D$, $R\circ\Gaa$ is not real analytic  at
$(1,0)\in\ov\D\times[0,1]$.

It is convenient
 not to use arc-length.
 Consider a $\cL C^\infty$ family of simply-connected bounded domains $\Om^\la$ bounded by
$$
\gaa(t,\la)=\rho(t,\la)e^{it}, \quad \rho(0,\la)=1=\rho(t,0),
$$
where $\rho$ is a positive $\cL C^\infty$ function satisfying $\rho(t+2\pi,\la)=\rho(t,\la)$.
To achieve the analyticity, we will require that $\rho-1$ vanishes near $t=0$ and $\la=0$.
 As complex valued functions, the outer unit
normal vector $\nu(t,\la)$ of $\pd\Om^\la$ is $-i\gaa'(t,\la)/{|\gaa'(t,\la)|}$. We have
\gan
k(s,t,\la)= \f{1}{\pi}\f{N(s,t)}{|\gamma(s,\la)-\gamma(t,\la)|^2}, \\
 N(s,t,\la)=\RE\{\ov{\nu(t,\la)}(\gaa(t,\la)-\gaa(s,\la))\}.
\end{gather*}
In the above and the
remaining computation, the derivatives are in $s,t$ variables only.
The derivatives in $\la$ at $\la=0$ are
indicated in the formal Taylor expansion about $\la=0$. For instance,
\gan
\gaa(t,\la)\sim \sum\gaa_n(t)\la^n,\quad \gaa_0(t)=e^{it}; \qquad k(s,t,\la)\sim
\sum k_n(s,t)\la^n.
\end{gather*}
 We will derive identities
for coefficients of formal power series in $\la$ and those identities are therefore valid
when they arise from $\cL C^\infty$ functions. We will also denote by $\rho_{(n)}^{(j)}(s)$
the collection of   $\pd_s^i\rho_l(s)$ with $i\leq j, l\leq n$ and  by $\rho_{(n)}$
the collection of   $\rho_l$ with $ l\leq n$.
 We will denote by $Q(\rho_{(n)}^{(j)})$ a function in $s$
and $t$
which depends on $\rho_{(n)}^{(j)}$  such that
\eq{pdstq}
|\pd_s^i\pd_t^{l-i}Q(\rho_{(n)}^{(j)})(s,t)|\leq C(n,j,l,|\rho_{(n)}|_{j+l})
\df C(|\rho_{(n)}|_{j+l}).
\eeq
 To simplify notation, the $Q$ might be different
when it reappears.

We express
\gan
\gaa'(t,\la)=ie^{it}(\rho(t,\la)-i \rho'(t,\la)),\quad\gaa(t,\la)-\gaa(s,\la)=
B(s,t,\la)(e^{it}-e^{is}),\\
B(s,t,\la)=\rho(s,\la)+(\rho(t,\la)-\rho(s,\la))(1-e^{i(s-t)})^{-1}.
\end{gather*}
Note that $B_0(s,t)=1=|\gaa'(t,0)|$. We also have
\aln
N(s,t,\la)&=\RE\Bigl\{\ov{\nu(t,\la)}(\gaa(t,\la)-\gaa(s,\la)))\Bigr\}=|e^{is}
-e^{it}|^2A(s,t,\la),\\
A(s,t,\la)&=|e^{is}-e^{it}|^{-2}\RE\Bigl\{\ov{\nu(t,\la)}(\gaa(t,\la)-\gaa(s,\la)
-i\gaa'(t,\la)(e^{i(s-t)}-1))\Bigr\}
\\ &\quad +|e^{is}-e^{it}|^{-2}(1-\cos(s-t))|\gaa'(t,\la)|.
\end{align*}
Therefore,
\eq{pdab}
|\pd_s^j\pd_t^{k-t}A_n(s,t)|+|\pd_s^j\pd_t^{k-t}B_n(s,t)|\leq C(|\rho_n|_{k+2}).
\eeq
It is clear that $A(s,t,\la)$, $B(s,t,\la)$ and $k(s,t,\la)=
A(s,t,\la)/{(\pi |B(s,t,\la)|^2)}$ are $\cL C^\infty$ in $(s,t,\la)$.
Using $B_0=1$, we compute derivatives of $k(s,t,\la)$ in $\la$ at $\la=0$. We find
   $k_0(s,t)=\f{1}{2\pi}$.
By \re{pdab} we get $k_n(s,t) =Q_n(\rho_{(n)}^{(2)})(s,t)$, which satisfies \re{pdstq}.
 We also have
$d\sigma(t,\la)=a(t,\la)\, dt$ with $a(t,\la)=|\gaa'(t,\la)|$. Then $a_0=1$
and $a_n=Q(\gaa_{(n)}')$.

Let $u^\la(z^\la)$ be the harmonic function on $\Om^\la$ with boundary value
$-\log|z^\la|$ on $\pd\Om^\la$.
To compute $u^\la$, set
  $f(s,\la)=-\log|\gaa(s,\la)|=-\log\rho(s,\la)$ and consider
$$
\var(s,\la)+\int_0^{2\pi}\var(t,\la)K(s,t,\la)a(t,\la)\, dt=f(s,\la).
$$
We have $f_0=0$ and $f_n(s)=-\rho_n(s)+Q(\rho_{(n-1)})(s)$.
  We obtain $\var_0=0$ and
\eq{varo}
\var_n(s)=-\yt\rho_n(s)+Q(\rho_{(n-1)}^{(2)})(s),\quad n>0.
\eeq
Recall that $\var$ is real-valued and
\aln
(U\var)(z,\la)&=\f{1}{\pi}\int_{\pd\Om^\la}   \var(s,\la)\pd_{\ta^\la}
\arg(\zeta^\la-z^\la)  \, d\sigma^\la=\RE\cL C^\la \var.
\end{align*}
 Let $z=r\in(-1,1)$. We get
\aln
\pd_r^j\pd_\la^n\cL C^\la\var(r^\la)&=\f{1}{2\pi i}\sum_{i=0}^{n-1}\binom{n}{i}
\int_{0}^{2\pi}\pd_\la^i(\var(s,\la))\pd_r^j\pd_\la^{n-i}
\Bigl\{\f{\pd_s\gaa(s,\la) }{\gaa(s,\la)-r^\la}\Bigr\}\, ds\\
&\quad +\f{1}{2\pi i}\pd_r^j
\int_{\pd\Om^\la}  (\pd_\la^{n}\var(s,\la))
 \f{d\zeta^\la}{\zeta^\la-r^\la}=I_1^\la(r^\la)+I_2^\la(r^\la).
\end{align*} We want to emphasize that  $\Gaa_n(z)$ is not determined by
$\rho_1,\ldots,\rho_n$.  Nevertheless, we want to show that, when restricted on the unit circle,
 $(U\var)_n$ and all derivatives $(\pd_r^i(U\var)_n)$
 depend only on $\rho_1,\ldots, \rho_n$. For $I_1$, we apply
 Stokes' theorem to transport all derivatives on the Cauchy kernel onto
 derivatives in $s$. After removing all derivatives on the Cauchy kernel, we set $\la=0$ and let
 $r\to1^-$. By $\var_0=0$, \re{varo} and a crude estimate on  orders of derivatives, we obtain
 $$
 I_1^\la(r^\la)=\cL C^0Q(\rho_{(n-1)}^{(n+j+2)})(1),\quad |I_1^\la(r^\la)|
 \leq C(|\rho_{(n-1)}|_{n+j+3}),
  \quad \la=0,\ r=1.
 $$
To compute  $I_2^\la$, we   express for $r\in(-1,1)$
\aln
&\pd_r
\int_{\pd\Om^\la}  f^\la(\zeta^\la)
 \f{d\zeta^\la}{\zeta^\la-r^\la}=
 \pd_rr^\la
\int_{\pd\Om^\la} \{ (\ov{\ta^\la}\pd_{\ta^\la})f^\la(\zeta^\la)\}
 \f{d\zeta^\la}{\zeta^\la-r^\la},\\
& \pd_r^j
\int_{\pd\Om^\la}  f^\la(\zeta^\la)
\f{d\zeta^\la}{\zeta^\la-r^\la}=
 (\pd_rr^\la)^j
\int_{\pd\Om^\la} \{ (\ov{\ta^\la}\pd_{\ta^\la})^jf^\la(\zeta^\la)\}
 \f{d\zeta^\la}{\zeta^\la-r^\la}\\
 &\hspace{10ex}
  +\sum_{i>1, l<j}\pd_r^ir^\la Q_{jl}(\pd_r^{(j-i)}r^\la)\int_{\pd\Om^\la}
   \{ (\ov{\ta^\la}\pd_{\ta^\la})^lf^\la(\zeta^\la)\}
 \f{d\zeta^\la}{\zeta^\la-r^\la}.
\end{align*}
Recall that
 $\gaa_0(s)=e^{it}$. Write $\gaa^\la(e^{it})=\gaa(t,\la)$.
  We further require that the extension $\Gaa^\la(z)$ of $\gaa^\la(z)$ satisfy
$\Gaa^0(z)=z$.  Thus at $(r,\la)=(1,0)$, we have $\pd_rr^\la=1$ and $\pd_r^jr^\la=0$
 for all $j>1$.
Set  $\la=0$,
let  $r\to1^-$ in $I_2^\la$ and apply  the jump formula for Cauchy transform
on the unit circle. We get
\aln
I_2^0(1)&=\f{1}{2\pi i}\int_0^{2\pi}\Bigl\{(-ie^{-is}\pd_s )^j\var_n(s)-
(-ie^{-it}\pd_s)^j\var_n(t))|_{t=0}\Bigr\}
 \f{ie^{is}ds}{e^{is}-1}\\
 &\quad +(\pd_t \cdot ie^{-it})^r\var_n(t)|_{t=0}\\
&=-\f{1}{4\pi i}\int_0^{2\pi}\Bigl\{(-ie^{-is}\pd_s)^j\rho_n(s)-(-ie^{-it}
\pd_t)^j\rho_n(t))|_{t=0}\Bigr\}
 \f{ie^{is}ds}{e^{is}-1}\\
 &\quad -\yt(-ie^{-it}\pd_t)^j\rho_n(t)|_{t=0}+\cL C^0Q(\rho_{(n-1)}^{(2+j)})(1)
 +Q(\rho_{(n-1)}^{(2+j)})\\
 &=-\f{(-1)^jj!}{4\pi i}\int_0^{2\pi}\rho_n(s)
 \f{ie^{is}ds}{(e^{is}-1)^{j+1}} +Q(\rho_{(n-1)}^{(3+j)}).
\end{align*}
Here $\cL C^0$ stands for the Cauchy transform on the unit circle.
Recall in notation \re{pdstq}, we have $|Q(\rho_{(n-1)}^{(3+j)})|\leq
C(|\rho_{(n-1)}|_{j+3})$. Here
 the second last identity is obtained via integration by parts under the
 additional conditions that $n>0$  and $\rho_n$ vanish
 near $s=0$.
Therefore, we get for $n>0$
\al\label{pdrju}
\pd_r^j(U\var)_n(1)=-\f{(-1)^jj!}{4\pi}\RE\int_0^{2\pi}\rho_n(s)
 \f{e^{is}ds}{(e^{is}-1)^{j+1}} +Q(\rho_{(n-1)}^{(n+6)}).
\end{align}

 We use the Riemann mapping $R^\la$ satisfying $R^\la(0)=0$ and $(R^\la)(1)=1$.
 Near $(z,\la)=(1,0)$, we have $\gaa^\la(z)=z$ and
 $$
 R(z,\la)=R^\la(\gaa^\la(z))=z e^{h^\la(z)},
 \quad h^\la(z)=u^\la(z)+iv^\la(z)-u^\la(1)-iv^\la(1).
 $$
 Here $v^\la$ is a harmonic conjugate of $u^\la=U\var$. Since $(U\var)_0=0$,
  then $(U\var)_0$ is
 identically zero. Hence $R_0(z)=z$.  At $z=1$, we have
\gan
R'=1+\pd_r u^\la,\quad
R''=(h^\la)''+((h^\la)')^2+2(h^\la)',\\
R'''=(h^\la)'''+3(h^\la)'(h^\la)''+((h^\la)')^3+3(h^\la)''+3((h^\la)')^2.
\end{gather*}
 We get
\ga\label{rrr}
(R')_0=1,\quad
(R'')_0=0,\quad
(R''')_0=0,\\  \RE R_n'''(1) =\pd_r^3(u^\la)_n(1)+3\pd_r^2(u^\la)_n(1)
+  Q(\rho_{(n-1)}^{(3)}).
\label{rrr1}\end{gather}
 By \rl{rmp}, there exists a unique  Riemann mapping $S^\la$  for $\pd\Om^\la$ that
  satisfies $$S^\la(1)=(S^\la)'(1)=1,\quad
  (S^\la)''(1)\in\rr.
  $$
Thus,
$(S^\la)_n'''(1)=R_n'''(1)$ by \re{rrr}-\re{rrr1} and
the last identity in \rl{rmp}.  For $n>0$ we obtain
\aln
\RE R_n'''(1)&=\f{3!}{4\pi}\RE\int_0^{2\pi}\rho_n(s)
\Bigl(\f{e^{is}}{(1-e^{is})^4}+\f{e^{is}}{(1-e^{is})^3}\Bigr)\,
ds+  Q(\rho_{n-1}^{(n+6)} )\\
&=\f{3!}{4\pi}\int_0^{2\pi}\f{\rho_n(s)\cos(2s)}{16\sin^4(s/2)} \,
ds+  Q(\rho_{n-1}^{(n+6)} ).
\nonumber
\end{align*}
One can inductively choose $\rho_n(s)=\tilde\rho_n(s)\sin^4(s/2)\cos(2s)$ with
$\tilde\rho_n\geq0$
  such that $\rho_n(s)=0$ on $|s|<\pi/2$ and  $R_n'''(1)>(n!)^2$ for $n>0$.
This shows that $(S^\la)'''(1)$ is not real analytic at $\la=0$, provided that
$\rho_n(s)$ can be realized
via a family of embeddings $\Gaa^\la$ satisfying all the requirements. To
achieve the latter, we apply a non-parameter
version of \rl{whit} to the unit disc $\D$ and find
$\tilde\rho_n\in\cL C^\infty(\ov\D)$ such that $\tilde\rho_n(e^{is})=\rho_n(s)$.
Moreover, all $\tilde\rho_n$
vanish in a fixed neighborhood of $1\in\ov\D$. Applying \rl{whit+}, we find $\tilde
\rho\in\cL C^\infty(\ov\D\times[0,1])$
such that $\tilde\rho(z,\la)$ vanishes near $(z,\la)=(1,0)$ and
 $\pd_\la^{n-1}\tilde\rho(z,\la)=(n-1)!\tilde\rho_n(z)$ at $\la=0$.
 Let $\Gaa(z,\la)=(1+\la\tilde\rho(z,\la))z$.
As we already mentioned, we can
extend $\rho(t,\la)$ to be identically $1$ near $(1,0)\in\ov\D\times I$.
Thus $\Gaa(z,\la)$ is real analytic
near $(1,0)$.
Replacing $\Gaa^\la$ by $\Gaa^{\del\la}$ if necessary, $\Gaa^\la$ embeds
$\ov\D$ into $\ov{\Om^\la}$,
when $\del>0$ is sufficiently small and $0\leq\la\leq1$.

We now consider any family of Riemann mappings $R^\la$ from $\Om^\la$ onto $\D$. Assume
for the sake of contradiction that $R$ is real analytic at $(1,0)\in\ov\D\times[0,1]$.
 Replace $R^\la$
by $\ov {R^\la(1)}R^\la$. By \rl{rmp},
$(S^\la)'''(1)$
is real analytic at $\la=0$,  which is a contradiction.
\hfill $\square$

\medskip

We conclude the paper with a remark when
the domains are fixed and only the boundary values vary with a parameter. In this case
we can reduce
the solutions to the case without   parameter.
Recall that the solution for the Dirichlet and Neumann problems consists of
solving the integral equations and estimating  the simple and double layer potentials
via Cauchy transform.
When we differentiate integral equations or Cauchy transform in parameter $\la$,
 the kernels are unchanged for fixed domains. The difficulties with the chain
rule in our arguments disappear.
More specifically, the estimates for
  the integral equations in \rp{kl3} (without restriction $k\geq j$) extend to spaces
  of types $\cL B_*$
and $\cL C_*$. The
estimates on the layer potentials via Cauchy transform   in \rp{c0pa} (without restriction
$k\geq j$) extend
to spaces of types $\cL B_*$
and $\cL C_*$ too. Thus, we have the following.
\pr{easy} Let $k,j$ and $l$ be non negative integers. Assume that $l\leq k+1$ and $0<\all<1$.
Let $\Om$ be a bounded domain in the complex plane with $\pd\Om\in\cL C^{k+1+\all}$.
 Let $u^\la$ be harmonic functions on $\Om$ which
 are continuous up to   boundary. If $u\in\cL B^{l+\all,j}_*(\pd\Om)$
  {\rm(}resp. $\cL C_*^{l+\all,j}(\pd\Om)${\rm)}, then
 $u\in\cL B^{l+\all,j}_*(\ov\Om)$ {\rm(}resp. $\cL C_*^{l+\all,j}(\ov\Om)${\rm)}.
 If $\int_{\pd\Om}u^\la\, d\sigma=0$
 and $\{\pd_\nu  u^\la\}$ is in
 $\cL B^{k+\all,j}_*(\pd\Om)$ {\rm(}resp. $\cL C_*^{k+\all,j}(\pd\Om)${\rm)}, then
 $u\in\cL B^{k+1+\all,j}_*(\ov\Om)$ {\rm(}resp. $\cL C_*^{k+1+\all,j}(\ov\Om)${\rm)}.
\epr


\newcommand{\Tsfini}{\bibitem{Tsfini} M. Tsuji,
{\it Potential theory in modern function theory},
 Maruzen Co., Ltd., Tokyo, 1959.
 }

 \newcommand{\Wathtw}{\bibitem{Wathtw} S.E. Warschawski,
{\it \"Uber einen Satz von O.D. Kellogg},
G\"ottinger Nachrichten, Math.-Phys. Klasse, 1932, 73-86.
 }

 \newcommand{\Wathtwb}{\bibitem{Wathtwb}\bysame,
 {\it \"Uber das Randverhalten der Ableitung der Abbildungsfunktion bei konformer Abbildung},
 Math. Z. {\bf 35}(1932), 321-456.}

\newcommand{\Befise}{\bibitem{Befise}L. Bers,
Riemann Surfaces (mimeographed lecture notes), New York
University, (1957-1958). }

\newcommand{\Vesitw}{\bibitem{Vesitw}
I.N. Vekua,  {\em Generalized analytic functions},
 Pergamon Press, London-Paris-Frankfurt; Addison-Wesley
Publishing Co., Inc., Reading, Mass. 1962. }

\newcommand{\Fonifi}{\bibitem{Fonifi} G.B. Folland, {\it
Introduction to partial differential equations},
second edition. Princeton University Press, Princeton, NJ, 1995.
}

\newcommand{\Kezeei}{\bibitem{Kezeei} O.D. Kellogg, {\it
Potential functions on the boundary of their regions of definition},
  Trans. Amer. Math. Soc.  {\bf 9}(1908),  no.~1, 39--50.
}

\newcommand{\Kezeeib}{\bibitem{Kezeeib}\bysame, {\it
Double distributions and the Dirichlet problem},
Trans. Amer. Math. Soc.  {\bf 9}(1908),  no.~1, 51--66.
}

\newcommand{\Keontw}{\bibitem{Keontw} \bysame, {\it
Harmonic functions and Green's integral},
  Trans. Amer. Math. Soc.  {\bf 13}(1912),  no.~1, 109--132.
}

\newcommand{\Ketwni}{\bibitem{Ketwni} \bysame, {\it
Foundations of potential theory},
 Reprint from the first edition of 1929.
 Die Grundlehren der Mathematischen Wissenschaften,
Band 31 Springer-Verlag, Berlin-New York 1967.}

\newcommand{\Plzefo}{\bibitem{Plzefo}
J. Plemelj, {\em \"Uber lineare Randwertaufgaben der Potentialtheorie},
  I. Teil.  Monatsh. Math. Phys.  {\bf 15}(1904),  no. 1, 337--411.
}

\newcommand{\BGR}{\bibitem{BGR} F. Bertrand, X. Gong and J.-P. Rosay,
{\it Common boundary values of holomorphic functions for two-sided complex structures},
submitted.
}

\newcommand{\Mirseze}{\bibitem{Mirseze}
C. Miranda, {\it
Partial differential equations of elliptic type},
Second revised edition.
Translated from the Italian by Zane C. Motteler.
Ergebnisse der Mathematik und ihrer Grenzgebiete,
Band 2. Springer-Verlag, New York-Berlin 1970.}

\newcommand{\Miseze}{\bibitem{Miseze} S.G.~Mikhlin, {\em Mathematical Physisc:
an advanced course}, North Holland, Amsterdam, 1970.
}

\newcommand{\Honize}{\bibitem{Honize}
L. H\"ormander, {\em The analysis of linear partial differential
operators.
 I. Distribution theory and Fourier analysis.}
 Springer-Verlag, Berlin, 1990.}

\end{document}